\documentclass[a4paper]{amsart}
\usepackage{amsmath,amsthm,amssymb,mathtools,mathrsfs}
\usepackage{bm}
\usepackage{tikz-cd}
\usepackage[all]{xy}
\usepackage{hyperref}
    \definecolor{urlcolor}{rgb}{0,.145,.698}
    \definecolor{linkcolor}{rgb}{.7,0.10,0.2}
    \definecolor{citecolor}{rgb}{.12,.54,.11}
\hypersetup{
      breaklinks=true, 
      colorlinks=true,
      urlcolor=urlcolor,
      linkcolor=linkcolor,
      citecolor=citecolor,
}

\usepackage{setspace}
\usepackage{xcolor}
\usepackage[shortlabels]{enumitem}
\usepackage[utf8]{inputenc}
\numberwithin{equation}{section}
\usepackage[top=2.4cm,
bottom=2.4cm,
left=2.4cm,
right=2.4cm]{geometry}

\usepackage[
style=alphabetic,
doi=false,
isbn=false,
url=false,
eprint=false,
block = space,
sorting=anyt,
maxnames=50]{biblatex}

\newtheorem{theorem}{Theorem}[section]
\newtheorem{corollary}[theorem]{Corollary}
\newtheorem{proposition}[theorem]{Proposition}
\newtheorem{proposition-definition}[theorem]{Proposition-Definition}
\newtheorem{lemma}[theorem]{Lemma}

\newtheorem{conjecture}[theorem]{Conjecture}

\theoremstyle{definition}

\newtheorem{theorem-definition}[theorem]{Theorem-Definition}

\theoremstyle{remark} 
\newtheorem{remark}[theorem]{Remark}
\newtheorem{example}[theorem]{Example}

\newcommand{\BD}{\mathbb{D}}

\newcommand{\BQ}{\mathbb{Q}}

\newcommand{\BoA}{\mathbb{A}}
\newcommand{\BoC}{\mathbb{C}}
\newcommand{\BoF}{\mathbb{F}}

\newcommand{\BoN}{\mathbb{N}}
\newcommand{\BoP}{\mathbb{P}}
\newcommand{\BoQ}{\mathbb{Q}}
\newcommand{\BoZ}{\mathbb{Z}}

\newcommand{\CA}{\mathcal{A}}

\newcommand{\CD}{\mathcal{D}}

\newcommand{\CF}{\mathcal{F}}

\newcommand{\CH}{\mathcal{H}}

\newcommand{\CP}{\mathcal{P}}
\newcommand{\CQ}{\mathcal{Q}}

\newcommand{\CS}{\mathcal{S}}

\newcommand{\FM}{\mathfrak{M}}
\newcommand{\FN}{\mathfrak{N}}

\newcommand{\FX}{\mathfrak{X}}

\newcommand{\SF}{\mathscr{F}}
\newcommand{\SG}{\mathscr{G}}

\renewcommand{\geq}{\geqslant}
\renewcommand{\leq}{\leqslant}
\renewcommand{\subset}{\subseteq}
\renewcommand{\setminus}{\smallsetminus}
\renewcommand{\tilde}{\widetilde}

\newcommand{\ad}{\operatorname{ad}}

\newcommand{\rmb}{\mathrm{b}}
\newcommand{\BMo}{\mathrm{BM}}
\newcommand{\crit}{\mathrm{crit}}

\newcommand{\rmd}{\mathrm{d}}

\newcommand{\rmc}{\mathrm{c}}
\newcommand{\codim}{\mathrm{codim}}

\newcommand{\CCy}{\mathrm{CC}}
\newcommand{\dd}{\mathbf{d}}
\newcommand{\ee}{\mathbf{e}}

\newcommand{\Eu}{\mathrm{Eu}}
\newcommand{\res}{\mathrm{res}}

\newcommand{\For}{\mathrm{For}}
\newcommand{\Fr}{\mathrm{Fr}}

\newcommand{\GL}{\mathrm{GL}}
\newcommand{\gl}{\mathfrak{gl}}

\newcommand{\Fun}{\operatorname{Fun}}
\newcommand{\Hom}{\operatorname{Hom}}
\newcommand{\imm}{\mathrm{im}}
\newcommand{\iso}{\mathrm{iso}}
\newcommand{\hyp}{\mathrm{hyp}}

\newcommand{\IC}{\mathcal{IC}}

\newcommand{\Irr}{\mathrm{Irr}}

\newcommand{\id}{\operatorname{id}}

\newcommand{\Ind}{\operatorname{Ind}}

\newcommand{\GK}{\mathrm{K}}
\newcommand{\Lagr}{\operatorname{Lagr}}

\newcommand{\op}{\mathrm{op}}

\newcommand{\pH}[1]{{^\mathfrak{p}\CH^{#1}}}

\newcommand{\real}{\mathrm{real}}
\newcommand{\Rep}{\operatorname{Rep}}

\newcommand{\pr}{\mathrm{pr}}
\newcommand{\rank}{\operatorname{rank}}
\newcommand{\Res}{\mathrm{Res}}

\newcommand{\SSup}{\operatorname{SS}}
\newcommand{\SSN}{\mathcal{SSN}}

\newcommand{\sph}{\mathrm{sph}}
\newcommand{\supp}{\operatorname{supp}}
\newcommand{\Tan}{\mathrm{T}}
\newcommand{\toph}{\mathrm{top}}

\newcommand{\Tr}{\operatorname{Tr}}

\newcommand{\vir}{\mathrm{vir}}

\newcommand{\rmZ}{\mathrm{Z}}

\DeclareMathOperator{\Sym}{Sym}
\newcommand{\pt}{\mathrm{pt}}
\DeclareMathOperator{\HO}{H}
\DeclareMathOperator{\UEA}{\mathbf{U}}

\title[Unipotent enveloping algebra of a quiver]{Geometric realisations of the unipotent enveloping algebra of a quiver}
\date{\today}
\author{Lucien Hennecart}

\address{School of Mathematics, University of Edinburgh, Edinburgh, UK}
\email{lucien.hennecart@ed.ac.uk}


\begin{document}
\begin{abstract}
We compare and generalise the various geometric constructions (due to Ringel, Lusztig, Schofield, Bozec, Davison...) of the unipotent generalised Kac--Moody algebra associated with an arbitrary quiver. These constructions are interconnected through several geometric operations, including the stalk Euler characteristic of constructible complexes, the characteristic cycle, the Euler obstruction map, and the intersection multiplicities of Lagrangian subvarieties. We provide a proof that these geometric realisations hold for the integral form of the Lie algebra. Furthermore, by modifying the generators of the enveloping algebra, we ensure compatibility with the natural coproducts that can be defined in terms of restriction diagrams.

As a result, we establish that the top cohomological Hall algebra of the strictly seminilpotent stack is isomorphic to the positive part of the enveloping algebra of the generalised Kac-Moody algebra associated with the quiver.  This appears to be one of the cornerstones needed to describe the BPS algebra of very general $2$-Calabi--Yau categories, which is the subject of the author's work on BPS Lie algebras for $2$-Calabi--Yau categories with Davison and Schlegel Mejia.
\end{abstract}

\maketitle

\setcounter{tocdepth}{1}

\tableofcontents

\section{Introduction}
\subsection{Quivers and Lie algebras}
Quivers are ubiquitous objects in representation theory, category theory, and geometry. Their path algebras give rise to a wide class of smooth algebras of broad interest in noncommutative geometry \cite{kontsevich2000noncommutative}. The geometric study of their moduli spaces provides geometric constructions of enveloping algebras and quantum groups associated with (generalised) Kac--Moody algebras, and of their (lowest weight) representations. These geometric constructions provide naturally various kinds of bases known as \emph{(semi-)canonical bases} \cite{lusztig1990canonical,lusztig1991quivers,kashiwara1997geometric,lusztig2000semicanonical}, whose existence can be proven combinatorially using involved recursive procedures \cite{kashiwara1990crystalizing,kashiwara1993global}.

Quivers with or without relations also serve as models for more general categories. Quivers without relations give local models for one-dimensional categories, allowing a local description of their stack of objects, while preprojective algebras (or their derived versions) act as building blocks for $2$-Calabi--Yau (CY) categories \cite{davison2021purity}. Quivers with potential serve as local models for $3$-Calabi--Yau categories \cite{ginzburg2006calabi,bocklandt2008graded}. These properties explain part of the ubiquity of quivers throughout mathematics and in particular representation theory and algebraic geometry. In particular, the $2$-CY version is at the heart of the explicit determination of the BPS algebra of arbitrary $2$-CY categories by reducing to preprojective algebras of quivers (\cite{davison2023bps}, but also \cite{arbarello2018singularities} and the references therein for previous instances of local quiver descriptions). This paper focuses on the path algebras of quivers and their preprojective algebras, establishing a connection between the associated (cohomological) Hall algebras.

\subsection{The Lie algebra associated to a quiver}
Let $Q=(Q_0,Q_1)$ be a quiver with set of vertices $Q_0$ and set of arrows $Q_1$. The Cartan matrix of $Q$ is $A=(a_{i,j})_{i,j\in Q_0}$ where $a_{i,j}\coloneqq2\delta_{i,j}-\lvert\{\alpha\in Q_1\mid \{s(\alpha),t(\alpha)\}=\{i,j\}\}\rvert$. We let $Q_0=Q_0^{\real}\sqcup Q_0^{\imm}$ be the partition of $Q_0$ in real vertices (having no loops) and imaginary vertices (having at least one loop). To a quiver, Bozec \cite{bozec2015quivers}with set of vertices $Q_0$ and set of arrows $Q_1$ associated a Borcherds datum where positive simple roots are parametrised by $I_{\infty}\coloneqq(Q_0^{\real}\times\{1\})\sqcup (Q_0^{\imm}\times\BoZ_{\geq 1})$. The Lie algebra $\mathfrak{g}_Q$ associated to $Q$ has generators $e_{(i',n)}, f_{(i',n)}$ for $(i',n)\in I_{\infty}$ and $h_{i'}$ for $i'\in Q_0$, satisfying Kac--Moody--Borcherds type relations (\S\ref{subsection:borcherdsbozecLiealgebra}). We refer to \S \ref{section:generalisedKMofaquiver} for the precise definitions. The Lie algebra $\mathfrak{g}_Q$ has a triangular decomposition $\mathfrak{g}_Q=\mathfrak{n}_Q^-\oplus\mathfrak{h}\oplus\mathfrak{n}_Q^+$. Its enveloping algebra is $\UEA(\mathfrak{g}_Q)=\UEA(\mathfrak{n}_Q^-)\otimes \UEA(\mathfrak{h})\otimes \UEA(\mathfrak{n}_Q^+)$. It has various integral form $\UEA^{\BoZ}(\mathfrak{g}_Q)$ defined in terms of divided powers with respect to different sets of generators. The integral forms also have triangular decompositions. The positive part $\UEA(\mathfrak{n}^+_Q)$ has a new set of generators $\tilde{e}_{(i',n)}$, $(i',n)\in I_{\infty}$ that satisfy the same relations as the generators $e_{(i',n)}$ (that is Serre relations) but is more convenient when comparing the combinatorially defined comultiplication with the geometric restriction functors. This new set of generators also allows us to define a different integral form $\tilde{\UEA}^{\BoZ}(\mathfrak{n}_Q^+)$ of $\UEA(\mathfrak{n}_Q^+)$. We are interested in geometric realisations of $\UEA^{\BoZ}(\mathfrak{n}^+_Q)$ appearing in the literature and the comparison between them. These geometric realisations are in terms of 
\begin{enumerate}
 \item constructible functions on the stack of representations of $Q$ (\cite{ringel1990hall,schofield,lusztig1991quivers} for quivers without loops; the case of quivers with loops had not been explicitly studied before),
 \item constructible complexes on the stack of representations of $Q$ (\cite{lusztig1990canonical} for finite type quivers, \cite{lusztig1991quivers} for loop-free quivers, \cite{lusztig1993tight,bozec2015quivers} for quivers with loops),
 \item constructible functions on the strictly seminilpotent stack (\cite{lusztig2000semicanonical} for quivers without loops, \cite{bozec2015quivers,bozec2016quivers} for quivers with loops),
 \item the top-cohomological Hall algebra of the strictly seminilpotent stack (the construction of the CoHA is given in \cite{schiffmann2020cohomological}, it is determined in \cite{davison2020bps} for quivers without loops).
\end{enumerate}

\subsection{Other categorical situations}
\label{subsection:othercategorical}
In this paper, we focus on the case of quivers as our main goal is a description of the top-CoHA of the strictly seminilpotent variety as the enveloping algebra of a generalised Kac--Moody Lie algebra (Theorem \ref{theorem:topcohagkm}). Our results admit analogues in various other contexts:
\begin{enumerate}
 \item (Langlands theory) Spherical Eisenstein perverse sheaves on the stack of coherent sheaves on a smooth projective curve and the global nilpotent cone.
 \item (Noncommutative curves) Perverse sheaves on the stack of representations of a general smooth algebra.
 \item \label{item:Springer} (Springer situations) Equivariant perverse sheaves on a representation of a connected reductive group.
\end{enumerate}
The proof that the characteristic cycle map is an algebra morphism (or compatible with the induction operations) follows the exact same steps as for quivers and so we do not reproduce it.

The various objects: constructible Hall algebras, their categorification in terms of perverse sheaves, and the cohomological Hall algebras, are defined in all these contexts (in \eqref{item:Springer}, one may want to set the reductive group to $\GL_n$, or work with the different algebraic structure obtained from inductions from Levi subgroups) and in particular for the stacks of coherent sheaves and Higgs bundles associated to curves (a non-exhaustive list of references is \cite{schiffmann2004noncommutative,schiffmann2006canonical,sala2020comological}). In the latter situation, the category $\mathcal{P}$ is the category of spherical Eisenstein sheaves over $\mathfrak{Coh}(X)$, the stack of coherent sheaves on a smooth projective curve X, and $\mathfrak{N}_{\Pi_Q}^{\SSN}$ is replaced by the global nilpotent cone, $\mathfrak{N}_X$. Our proofs can be adapted in a straighforward way to prove that the characteristic cycle map
 \[
  \CCy\colon \widehat{\GK_0(\mathcal{P})}\rightarrow \HO^{\BMo}_{\toph}(\mathfrak{N}_X,\BoZ)
 \]
is an algebra morphism. The hat-symbol indicates some completion, necessary since the stack of rank $r$ and degree $d$ coherent sheaves is not of finite type if $r>0$. Its bijectivity is known in a few cases (curves of genus $\leq 1$), see for example \cite{hennecart2022perverse}, but is an open conjecture in general. The consequences of this will be described elsewhere.

\subsection{The main results}
Let $Q=(Q_0,Q_1)$ be a quiver. For $i'\in Q_0$, we let $e_{i'}\in\BoZ^{Q_0}$ be the basis vector. We let $\mathfrak{M}_Q\coloneqq\bigsqcup_{\dd\in\BoN^{Q_0}}X_{Q,\dd}/\GL_{\dd}$ be the stack of representations of $Q$, $\Pi_Q$ the preprojective algebra of $Q$ and $\mathfrak{M}_{\Pi_Q}$ the stack of representations of $\Pi_Q$ (\S\ref{subsection:stacksassociatedtoquivers}). It contains the strictly seminilpotent stack $\mathfrak{N}_{\Pi_Q}^{\SSN}$. We have a natural inclusion $\FN_{\Pi_Q,\dd}^0\coloneqq(\Tan^*_{X_{Q\dd}}X_{Q,\dd})/\GL_{\dd}\subset\mathfrak{N}_{\Pi_Q,\dd}^{\SSN}$ of the zero-section.

Lusztig defined in \cite{lusztig1990canonical,lusztig1991quivers} (for loop-free quivers) and \cite{lusztig1993tight} (for arbitrary quivers) a category $\mathcal{P}$ of semisimple perverse sheaves on $\mathfrak{M}_Q$. The simple objects of $\mathcal{P}$ are the simple constituents of the various inductions of constant sheaves on $\mathfrak{M}_{Q,\dd}$, $\dd\in\BoN^{Q_0}$ (Lusztig's construction is parallel to that of character sheaves for reductive Lie algebras or algebraic groups). The additive category $\mathcal{Q}$ is defined as the category of direct sums of shifts of objects of $\mathcal{P}$.

For the readability of the introduction, we split into Theorems \ref{theorem:CCalgmap}, \ref{theorem:topcohagkm} and \ref{theorem:realisationdiagram} the big Theorem \ref{theorem:maintheoremexpanded}.

\subsubsection{The characteristic cycle map}

\begin{theorem}[$\subset$ Theorem \ref{theorem:maintheoremexpanded}]
\label{theorem:CCalgmap}
 Let $Q$ be a quiver. Let $\mathcal{P}$ be Lusztig's category of perverse sheaves on the stack of representations of $Q$ and $\mathfrak{N}_{\Pi_Q}^{\SSN}$ be the strictly seminilpotent stack. Then, the characteristic cycle map
 \[
  \CCy\colon \GK_0(\mathcal{P})\rightarrow \HO^{\BMo}_{\toph}(\mathfrak{N}_{\Pi_Q}^{\SSN},\BoZ)
 \]
 is an algebra isomorphism, where the vector space $\HO^{\BMo}_{\toph}(\mathfrak{N}_{\Pi_Q}^{\SSN},\BoZ)$ has the (restriction of the) cohomological Hall algebra product (\S\ref{subsubsection:cohaproduct}) and $\GK_0(\CP)$ has Lusztig's induction product (\S\ref{subsubsection:theinductionppsheaves}).
\end{theorem}

\subsubsection{The top-CoHA of a quiver}
\label{subsubsection:topCoHAquiver}
We let $\HO^{\BMo}_{\toph}(\mathfrak{N}_{\Pi_Q}^{\SSN},\BoZ)\coloneqq\bigoplus_{\dd\in\BoZ^{Q_0}}\HO^{\BMo}_{\toph}(\mathfrak{N}_{\Pi_Q,\dd}^{\SSN},\BoZ)$ be the top-Borel--Moore homology of the strictly seminilpotent stack. It has a $\BoZ$-basis given by the fundamental classes of the irreducible components of $\mathfrak{N}_{\Pi_Q}^{\SSN}$ and a product (\S\ref{subsubsection:cohaproduct}) induced by the cohomological Hall algebra product on $\HO^{\BMo}_*(\mathfrak{N}_{\Pi_Q,\dd}^{\SSN})$ \cite{schiffmann2020cohomological}. The integral form of the positive part of the quantum group of the quiver, $\UEA^{\BoZ}(\mathfrak{n}^+_Q)$ has a set of generators $e_{(i',n)}$, $(i',n)\in Q_0\times\BoZ_{\geq 1}$ (\S\ref{subsection:borcherdsbozecLiealgebra}).

We let $\langle-,-\rangle$ be the Euler form of the quiver, $(-,-)$ the symmetrised Euler form (\S\ref{subsection:borcherdsbozecLiealgebra}), and we fix a multiplicative bilinear form $\Psi\colon \BoZ^{Q_0}\times \BoZ^{Q_0}\rightarrow\BoZ$ such that for any $\dd,\ee\in\BoZ^{Q_0}$, $\Psi(\dd,\ee)\Psi(\ee,\dd)\equiv(-1)^{(\dd,\ee)}$. For example, one can choose $\Psi=(-1)^{\langle-,-\rangle}$. This particular case is of importance as this is the twist that appears implicitly in the literature \cite{lusztig1991quivers}.
\begin{theorem}[$\subset$ Theorem \ref{theorem:maintheoremexpanded}]
\label{theorem:topcohagkm}
 Let $Q$ be a quiver and $\mathfrak{N}_{\Pi_Q}^{\SSN}$ the strictly seminilpotent stack. There exists an algebra morphism
 \[
 \begin{matrix}
  \beta&\colon& \UEA^{\BoZ}(\mathfrak{n}_Q^+)&\rightarrow& \HO^{\BMo}_{\toph}(\mathfrak{N}_{\Pi_Q}^{\SSN},\BoZ)^{\Psi}\\
  &&e_{i',n}&\mapsto&[\FN_{\Pi_Q,ne_{i'}}^0].
 \end{matrix}
 \]
 It is an isomorphism. 
\end{theorem}

The upperscript $\Psi$ indicates a twist of the multiplication, where $\Psi\colon \BoZ^{Q_0}\times\BoZ^{Q_0}\rightarrow \BoZ$ is a multiplicative bilinear form: If $\star$ denotes the classical multiplication on $\HO^{\BMo}_{\toph}(\mathfrak{N}_{\Pi_Q}^{\SSN},\BoZ)$ and $\star_{\Psi}$ the multiplication on $\HO^{\BMo}_{\toph}(\mathfrak{N}_{\Pi_Q}^{\SSN},\BoZ)^{\Psi}$, then $u\star_{\Psi}v\coloneqq \Psi(\deg(u),\deg(v))(u\star v)$ for any $u,v\in \HO^{\BMo}_{\toph}(\mathfrak{N}_{\Pi_Q}^{\SSN},\BoZ)$ homogeneous of respective degrees $\deg(u),\deg(v)\in\BoN^{Q_0}$. We will give more details on twists of algebras in \S\ref{subsection:Psi-twist}. As this is clear in the notation, we prefer not to include the twist in the definition of the cohomological Hall algebra product, following the definitions of \cite{schiffmann2020cohomological,davison2022bps} (which are natural from a geometric viewpoint). This theorem generalises \cite[Theorem 6.6]{davison2020bps} in two directions: our theorem applies to the whole quiver $Q$ and not only its \emph{real subquiver} (the full subquiver generated by loop-free vertices), and it applies to the cohomology with integral coefficients (and not only rational coefficients). Our method of proof is different than that of \cite{davison2020bps} since it uses the characteristic cycle and the category of Lusztig sheaves. It has the advantage of being applicable to other situations, for example Eisenstein perverse sheaves and the global nilpotent cone (see \S\ref{subsection:othercategorical}). The cohomological Hall algebras of the stack of Higgs sheaves of curves were defined in \cite{minets2020cohomological,sala2020comological} and have been used to prove deep results in nonabelian Hodge theory and Langlands theory revolving around the geometry of the Hitchin system. One can mention in particular the proof of the celebrated $\mathrm{P}=\mathrm{W}$ conjecture \cite{hausel2022p}.

\begin{remark}
 The coalgebra structure on $\HO^{\BMo}_{\toph}(\mathfrak{N}_{\Pi_Q}^{\SSN},\BoZ)$ is not yet constructed directly (one could obviously transfer the coproduct of the enveloping algebra using $\beta$). We expect a geometric construction of it.
\end{remark}

\subsubsection{Geometric realisations of $\UEA^{\BoZ}(\mathfrak{n}_Q^+)$}
If $\mathfrak{X}$ is a stack, we let $\Fun(\mathfrak{X})$ be the set of $\BoZ$-valued constructible functions on $\mathfrak{X}$. As for varieties, for each $\varphi\in\Fun(\mathfrak{X})$, there is a stratification of $\mathfrak{X}$ by locally closed substacks on which $\varphi$ takes a single value.

We fix a bilinear form $\Psi$ as in \S\ref{subsubsection:topCoHAquiver}.

\begin{theorem}[$\subset$ Theorem \ref{theorem:maintheoremexpanded}]
\label{theorem:realisationdiagram}
 We have a commutative diagram of $\BoZ$-algebra isomorphisms
 \[
   \begin{tikzcd}
	{\UEA^{\BoZ}(\mathfrak{n}_Q^+)} \\
	& {\GK_0(\mathcal{Q})^{\Psi}} & {\HO^{\BMo}_{\toph}(\mathfrak{N}_{\Pi_Q}^{\SSN},\BoZ)^{\Psi}} \\
	& {\Fun^{\sph}(\mathfrak{M}_Q)^{\Psi}}\\
	&\Fun^{\sph}(\mathfrak{N}_{\Pi_Q}^{\SSN})^{\Psi}
	\arrow["\chi"', from=2-2, to=3-2]
	\arrow["{\Tan^*_{[-]}\mathfrak{M}\circ \Eu^{-1}}"'{pos=0.5}, from=3-2, to=2-3]
	\arrow["\CCy", from=2-2, to=2-3]
	\arrow["\gamma"',bend right=30, from=1-1, to=3-2]
	\arrow["\alpha", from=1-1, to=2-2]
	\arrow["\beta",bend left=30, from=1-1, to=2-3]
	\arrow["\res",from=4-2, to=3-2]
	\arrow["\delta"', bend right=40,from=1-1, to=4-2]
\end{tikzcd}
 \]
 where
 \begin{enumerate}
  \item $\Fun^{\sph}(\mathfrak{M}_Q)$ is the set of spherical constructible functions on $\mathfrak{M}_Q$,
  \item $\Fun^{\sph}(\mathfrak{N}_{\Pi_Q}^{\SSN})$ is the set of spherical constructible functions on $\mathfrak{N}_{\Pi_Q}^{\SSN}$,
  \item $\CCy$ is the characteristic cycle map of constructible complexes,
  \item $\chi$ is the stalk Euler characteristic,
  \item $\res$ is the restriction of constructible functions (using that $\mathfrak{M}_Q\subset\mathfrak{N}_{\Pi_Q}^{\SSN}$ is a substack),
  \item $\Tan^*_{[-]}\mathfrak{M}\circ \Eu^{-1}$ is the composition of the inverse of the Euler obstruction map with the cotangent cycle map (the map $\overline{Z}\mapsto \overline{\Tan^*_ZX}$ for $Z\subset X$ a smooth locally closed subvariety),
  \item $\alpha$ is determined by the condition $\alpha(e_{i',n})=[\mathcal{IC}(\mathfrak{M}_{Q,ne_{i'}})]$,
  \item $\beta$ is determined by the condition $\beta(e_{i',n})=[\FN_{\Pi_Q,ne_{i'}}^0]$,
  \item $\gamma$ is determined by the condition $\gamma(e_{i',n})=(-1)^{\langle ne_{i'},ne_{i'}\rangle}1_{\mathfrak{M}_Q,ne_{i'}}$,
  \item $\delta$ is determined by the condition $\delta(e_{i',n})=(-1)^{\langle ne_{i'},ne_{i'}\rangle}1_{\FN_{\Pi_Q,ne_{i'}}^0}$.
 \end{enumerate}
\end{theorem}
\begin{remark}
\label{remark:ICnotprimitive}
 We need to modify the generators to obtain coalgebra morphisms since $\IC(\FM_{Q,ne_{i'}})$ is not primitive if $n\geq 2$ for the natural coproduct given by the restriction of complexes. This is what we explain in \S\ref{subsection:comultiplications}.
\end{remark}

\begin{remark}
\begin{enumerate}
 \item The term ``spherical" indicates that we take subsets of the sets $\Fun(\FM_Q)$ or $\Fun(\FN_{\Pi_Q}^{\SSN})$ of all constructible functions. The definitions of these subsets involve the algebra structures on these spaces, see \S\ref{subsubsection:sphericalsubalgebraQ}, \ref{subsubsection:spherical}.
 \item The definitions of the multiplications on the objects $\GK_0(\CQ)$, $\Fun^{\sph}(\FM_Q)$, $\Fun^{\sph}(\FN_{\Pi_Q}^{\SSN})$ and $\HO^{\BMo}_{\toph}(\FN_{\Pi_Q}^{\SSN},\BoZ)$ and of the comultiplications on $\GK_0(\CQ)$, $\Fun^{\sph}(\FM_Q)$ and $\Fun^{\sph}(\FN_{\Pi_Q}^{\SSN})$ will be recalled in this paper (in \S\ref{subsubsection:theinductionppsheaves}, \ref{subsubsection:theproductcstbleQ}, \ref{subsubsection:productcstblenilstack}, and \ref{subsubsection:cohaproduct} for the multiplications, \ref{subsubsection:restrictionfunctor}, \ref{subsection:restrictioncstbleQ} and \ref{subsubsection:coproductseminilpotent} for the comultiplications) to make the sign conventions explicit.
 
 \item We did not incorporate the $\Psi$-twist of the multiplications in the definitions of these multiplications to stay close to the definitions in the literature and to emphasize the importance of this twist (as it will also appear in \cite{davison2022bps,davison2023bps} in a very crucial way). For example, the twisted quantum group $\UEA_q^{\BoZ}(\mathfrak{n}_Q^+)^{\Psi}$ is isomorphic to $\UEA_{-q}^{\BoZ}(\mathfrak{n}_Q^+)$.
\end{enumerate}
\end{remark}

\subsection{Comultiplications}
\label{subsection:comultiplications}
We can modify the morphisms $\alpha, \gamma$ and $\delta$ (with the help of \emph{noncommutative symmetric functions}, \S\ref{subsection:noncommutativesymfunct}) so that they become \emph{bialgebra} morphisms (see Remark \ref{remark:ICnotprimitive}).

Namely, we define new generators $\tilde{e}_{i,n}$ of $\UEA(\mathfrak{n}^+_Q)$ (\S\ref{subsection:newgenerators}). They statisfy the condition
\[
 \Delta\tilde{e}_{i',n}=\sum_{r+s=n}\tilde{e}_{i',r}\otimes \tilde{e}_{i',s}.
\]
These generators give rise to an integral form $\tilde{\UEA}^{\BoZ}(\mathfrak{n}_Q^+)$ of $\UEA(\mathfrak{n}_Q^+)$ (defined using divided powers), which usually differs from $\UEA^{\BoZ}(\mathfrak{n}_Q^+)$. Moreover, the generators $\tilde{e}_{i',n}$ satisfy the same relations as the generators $e_{i',n}$: we have an \emph{algebra} isomorphism
\[
\begin{matrix}
 g&\colon&\tilde{\UEA}^{\BoZ}(\mathfrak{n}_Q^+)&\rightarrow&\UEA^{\BoZ}(\mathfrak{n}_Q^+)\\
 &&\tilde{e}_{i',n}&\mapsto&e_{i',n}.
\end{matrix}
\]
Of course, it is not a coalgebra isomorphism as the generators $e_{i',n}$ are primitive while the generators $\tilde{e}_{i',n}$ are usually not.
\begin{theorem}
\label{theorem:comultiplications}
We let $\Psi=\langle-,-\rangle_Q$ be given by the Euler form of the quiver $Q$. The morphisms
\[
 \alpha\circ g\colon \tilde{\UEA}^{\BoZ}(\mathfrak{n}_Q^+)\rightarrow\GK_0(\mathcal{Q})^{\Psi}
\]
\[
 \gamma\circ g\colon \tilde{\UEA}^{\BoZ}(\mathfrak{n}_Q^+)\rightarrow\Fun^{\sph}(\mathfrak{M}_Q)^{\Psi}
\]
\[
 \delta\circ g\colon \tilde{\UEA}^{\BoZ}(\mathfrak{n}_Q^+)\rightarrow\Fun^{\sph}(\mathfrak{N}_{\Pi_Q}^{\SSN})^{\Psi}
\]
are bialgebras isomorphisms.
\end{theorem}
The upgrade of the morphism $\alpha\circ g$ to the integral part of the quantum group (as opposed to the integral part of the enveloping algebra) is nontrivial when $Q$ has vertices with at least two loops. This would involve a fine understanding of $q$-deformations of the algebra of noncommutative symmetric functions. We leave this question for further investigations.

\subsection{Cohomological Hall algebras}
This work is motivated by and serves as the cornerstone for a general description of the cohomological Hall algebra of a $2$-Calabi--Yau (CY) category $\mathcal{A}$. This objective is pursued in the papers \cite{davison2022bps,davison2023bps}, where the BPS algebra of $\mathcal{A}$ is characterised as the enveloping algebra of the positive part of a generalised Kac--Moody Lie algebra. The intersection cohomology of certain connected components of the coarse moduli space of the stack of objects in the category $\mathcal{A}$ provides a set of generators, while the Serre relations are determined by the Euler form of category $\mathcal{A}$. The aforementioned works even establish stronger results by considering the \emph{relative cohomological Hall algebras}, which are objects in the derived category of mixed Hodge modules over $\mathcal{M}_{\mathcal{A}}$, the moduli space of semisimple objects in the category $\mathcal{A}$.

The proof of \cite{davison2023bps} proceeds in two steps. The first step is the reduction of the case of a general $2$-CY category $\CA$ to the case of the categories of representations of preprojective algebras of quivers. The second step is an induction, which relies on the explicit description of the top-CoHA of the strictly seminilpotent variety, Theorem \ref{theorem:topcohagkm}.

The relative versions proven in \cite{davison2023bps} retain a lot more information than the absolute ones: it determines the cohomological Hall algebras of \emph{all} the Serre subcategories of $\Rep(\Pi_Q)$. In particular, it determines the cohomological Hall algebras of the strictly seminilpotent stack $\FN_{\Pi_Q}^{\SSN}$, and therefore also the top-part of it, $\HO^0(\FN_{\Pi_Q}^{\SSN},\BD\BoQ^{\vir})=\HO^{\BMo}_{\toph}(\FN_{\Pi_Q}^{\SSN},\BoQ)$ (which we aim to study in this paper, see Theorem \ref{theorem:topcohagkm}). By restricting the relative cohomological Hall algebra to the strictly seminilpotent stack and then taking the derived global sections and the degree $0$ cohomology, one could deduce the rational version of Theorem \ref{theorem:topcohagkm} as a corollary of Theorem \cite{davison2023bps}. It turns out, conversely, that in \cite{davison2023bps}, we rely on Theorem \ref{theorem:topcohagkm}.

The consequences of the main theorem of \cite{davison2023bps} and its generalisation to all $2$-CY categories satisfying some mild assumptions are multiple. A certain number of them are described in \cite{davison2022bps} and \cite{davison2023bps}. In particular, we obtain the positivity conjecture of cuspidal polynomials of quivers of Bozec--Schiffmann, \cite[Theorem 1.4]{davison2023bps} (a strengthening of Kac's positivity conjecture), formulated in \cite{bozec2019counting}. It allows a description of the whole cohomology of smooth Nakajima quiver varieties as a module over some generalised Kac--Moody Lie algebra, in terms of the intersection cohomology of some singular Nakajima quiver varieties \cite[Theorem 1.14]{davison2022bps}. We also obtain the cohomological integrality conjecture for 2-CY categories and their 3-CY completions, for example for local K3 surfaces. This cohomological integrality allowed us to construct a nonabelian Hodge isomorphism between the Borel--Moore homologies of Dolbeault and Betti stacks for a smooth projective curve \cite[Theorem 1.7]{davison2022bps}. It is therefore extremely surprising and of high interest that these geometric properties eventually rely on Theorem \ref{theorem:topcohagkm}, which is of rather combinatorial nature and proven using the Hall algebra of constructible sheaves on the moduli stack of representations of the quiver $Q$. Of course, additional powerful geometric ingredients such as the \emph{local neighbourhood theorem} \cite{davison2021purity} are used in \cite{davison2022bps,davison2023bps}.

\subsection{Notations and conventions}
\begin{enumerate}
 \item The Borel--Moore homology of varieties and stacks is normalised so that $\HO^{\BMo}_{\dim\FX}(\FX,\BoZ)$ has a basis given by irreducible components of $\FX$ of dimension $\dim \FX$. We denote this space by $\HO^{\BMo}_{\toph}(\mathfrak{X},\BoZ)$.
 
\item If $G$ is a connected algebraic group and $X$ is a $G$-variety, then the equivariant Borel--Moore homology of $X$ is normalised so that $\HO^{\BMo,G}_{\dim X-\dim G}(X,\BoZ)$ has a basis given by fundamental classes of irreducible components of $X$. We do this so that it matches with the Borel--Moore homology $\HO^{\BMo}(X/G)$ of the quotient stack.
\item A quiver $Q$ is a pair $(Q_0,Q_1)$ of a set of vertices $Q_0$ and a set of arrows $Q_1$. For an arrow $\alpha\in Q_1$, $s(\alpha)$ denotes its source and $t(\alpha)$ its target.
\item If $X$ is an algebraic variety, $H\subset G$ algebraic groups and $H$ acts on $X$, we let $G\times^HX$ be the (free) quotient of $G\times X$ by the action of $H$ given by $h\cdot(g,x)=(gh,h^{-1}x)$. It is naturally a $G$-variety.
\item Let $X$ be a (complex) algebraic variety. We denote by $\CD^{\rmb}_{\rmc}(X)$ the constructible derived category of $X$ of sheaves of $\BoQ$-vector spaces.
\end{enumerate}

\subsection*{Acknowledgements}
The author would like to thank, among many others, Ben Davison, Joel Kamnitzer, Sasha Minets and Olivier Schiffmann for many useful discussions. The author is very grateful to Ben Davison for pointing out the missing twist in the product of the CoHA in an earlier version of this paper, without which the main results of this paper fail to hold. This twist led to further developments of this paper, and appears to be a feature, crucial in \cite{davison2022bps} and \cite{davison2023bps}. The author is also grateful to Joel Kamnitzer for pointing out a counterexample to a conjecture we made in a previous version of this paper.

The author was supported by the ERC Starting Grant ``Categorified Donaldson-Thomas Theory" No. 759967 whose Principal Investigator was Ben Davison.

\section{Characteristic cycle and singular support}
\label{section:CCandSS}
In this section, we recall the notions of characteristic cycle and singular support, together with the essential functorial properties for smooth pullback and proper pushforward. We explain in details how the pullback and pushforwards of Lagrangian cycles are constructed, as these constructions are crucial to prove that induction and characteristic cycle are two commuting operations, and the normalisation conventions for the characteristic cycle vary in the literature.

\subsection{Characteristic cycle}
\label{subsection:CCmap}
Let $X$ be a smooth algebraic variety of dimension $n$. Let $\Tan^*X$ be its cotangent bundle. It has a $\mathbb{G}_{\mathrm{m}}$-action rescaling the fibers. A closed subvariety $\Lambda\subset \Tan^*X$ is called Lagrangian if its smooth locus is Lagrangian in $\Tan^*X$. Obviously, $\dim\Lambda=n$. It is conical if it is invariant under the $\mathbb{G}_{\mathrm{m}}$-action. The relation of inclusion of closed, conical, Lagrangian subvarieties of $\Tan^*X$ induces a direct system on their Borel--Moore homologies $\HO^{\BMo}_*(\Lambda,\BoZ)$. The vector space $\varinjlim_{\Lambda}\limits \HO_n^{\BMo}(\Lambda,\BoZ)$ is the target space of the characteristic cycle map. We let $\Lagr^{\mathbb{G}_{\mathrm{m}}}(\Tan^*X)=\varinjlim_{\Lambda}\limits \HO_n^{\BMo}(\Lambda,\BoZ)$ be the \emph{space of Lagrangian cycles} on $\Tan^*X$.

The source of the characteristic cycle map is the Grothendieck group of the derived category of constructible sheaves of $\BoQ$-vector spaces on $X$, $\GK_0(\CD^{\rmb}_{\rmc}(X,\BoQ))$. As we work with rational coefficients for constructible sheaves throughout the paper, we sometimes drop the letter $\BoQ$.

The characteristic cycle map is a morphism of Abelian groups
\begin{equation}
 \label{equation:characteristiccyclemap}
 \CCy\colon \GK_0(\CD^{\rmb}_{\rmc}(X,\BoQ))\rightarrow \varinjlim_{\Lambda}\limits \HO^{\BMo}_n(\Lambda,\BoZ)
\end{equation}
which can be constructed as follows. Let $\mathscr{F}\in \CD^{\rmb}_{\rmc}(X,\BoQ)$ be a constructible complex. We fix a Whitney stratification $\mathcal{S}=(S_{\alpha})_{\alpha\in A}$ of $X$ for which $\mathscr{F}$ is constructible. We let $\Lambda_{\alpha}=\overline{\Tan^*_{S_{\alpha}}X}$ and $\Lambda=\bigcup_{\alpha\in A}\Lambda_{\alpha}$. We define $\CCy(\mathscr{F})$ as an element of $\HO^{\BMo}_n(\Lambda,\BoZ)$, that is as a linear combination
\begin{equation}
\label{equation:cc}
 \CCy(\mathscr{F})=\sum_{\alpha\in A}m_{\alpha}[\Lambda_{\alpha}].
\end{equation}
The multiplicity $m_{\alpha}$ is defined in the following way. Let $(x,\xi)\in\Lambda_{\alpha}$ be a general point (we can take $(x,\xi)\in \Lambda_{\alpha}\setminus\bigcup_{\beta\neq\alpha}\Lambda_{\beta}$). Let $g$ be a complex valued function on a neighbourhood of $x\in X$ such that
\begin{enumerate}
 \item $g(x)=0$,
 \item $\mathrm{d}g(x)=\xi$,
 \item $x$ is a nondegenerate critical point of $g_{|S_{\alpha}}$.
 
\end{enumerate}
Then,
\[
 m_{\alpha}=\chi((\phi_g\mathscr{F}[-1]_x),
\]
the Euler characteristic of the stalk at $x$ of the vanishing cycle functor $\phi_g[-1]$ applied to $\mathscr{F}$ (see for example \cite[\S 2]{schmid1996characteristic}, \cite[Remark 2.2]{massey2011calculations}, \cite{kashiwara2013sheaves}). We would like to emphasize that there is one other normalisation of the characteristic cycle also appearing in the literature, which is the one obtained when changing the multiplicity $m_{\alpha}$ in the sum \eqref{equation:cc} by the sign $(-1)^{\dim S_{\alpha}}$. For the normalisation chosen here, the multiplicities $m_{\alpha}$ are nonnegative when $\mathscr{F}$ is a perverse sheaf on $X$ (since $\phi_g\SF[-1]$ sends perverse sheaves to perverse sheaves supported on $\crit(g)=\{x\}$).

If $f\colon Y\rightarrow X$ is a map between smooth algebraic varieties, we have the cotangent correspondence
\[
 \begin{tikzcd}
	{\Tan^*X} & {\Tan^*X\times_XY} & {\Tan^*Y}
	\arrow["{{(\rmd f)^*}}", from=1-2, to=1-3]
	\arrow["{\pr_1}"', from=1-2, to=1-1]
\end{tikzcd}
\]
where $\pr_1$ is the projection on the first factor and $(\rmd f)^*$ is the pullback of covectors.

The characteristic cycle map satisfies the following properties:

\subsubsection{Distinguished triangles} \label{item:triangle} If $\mathscr{F}\rightarrow \mathscr{G}\rightarrow\mathscr{H}\rightarrow$ is a distinguished triangle in $\CD^{\rmb}_{\rmc}(X)$, then
\[
 \CCy(\mathscr{F})+\CCy(\mathscr{G})+\CCy(\mathscr{H})=0.
\]

\subsubsection{Shift} \label{item:shift} For any $\mathscr{F}\in \CD^{\rmb}_{\rmc}(X)$ and $n\in\BoZ$, $\CCy(\mathscr{F}[n])=(-1)^n\CCy(\mathscr{F})$.
 
 \subsubsection{Normalisation} \label{item:normalisationCC} If $\mathscr{F}=\mathscr{L}$ is a local system on a smooth equidimensional closed subvariety $Y\subset X$, $\CCy(\mathscr{F})=(-1)^{\dim Y}\rank(\mathscr{L})[\Tan^*_YX]$,
 
  \subsubsection{Smooth pullback} \label{item:smoothpb}If $f\colon Y\rightarrow X$ is smooth of relative dimension $d$, the map $\pr_1\colon \Tan^*X\times_XY\rightarrow \Tan^*X$ is smooth of relative dimension $d$, as being obtained by pullback from $f$.
 
 Let $\mathscr{F}\in \CD^{\rmb}_{\rmc}(X,\BoQ)$ and $\Lambda=\supp(\CCy(\mathscr{F}))\subset \Tan^*X$. The pullback map $\pr_1^{-1}(\Lambda)\rightarrow\Lambda$ induced by $\pr_1$ is smooth of relative dimension $d$. We then have a pullback in Borel--Moore homology
 \begin{equation}
 \label{equation:pbforsmpb}
  \pr_1^*\colon \HO^{\BMo}_*(\Lambda)\rightarrow \HO^{\BMo}_{*+d}(\pr_1^{-1}(\Lambda))
 \end{equation}
 and the map $(\rmd f)^*$ is a closed immersion, so that ${(\rmd f)^*}\pr_1^{-1}(\Lambda)$ is a closed subvariety of $\Tan^*Y$ and the pushforward by ${(\rmd f)^*}$ induces a morphism
 \begin{equation}
 \label{equation:pfforsmoothpb}
  {((\rmd f)^*)}_*\colon \HO^{\BMo}_*(\pr_1^{-1}(\Lambda))\rightarrow \HO^{\BMo}_*({(\rmd f)^*}\pr_1^{-1}(\Lambda)).
 \end{equation}
 
By composing \eqref{equation:pbforsmpb} and \eqref{equation:pfforsmoothpb}, we obtain a map
\[
 f^*\coloneqq{((\rmd f)^*)}_*\pr_1^*\colon \HO^{\BMo}_*(\Lambda)\rightarrow \HO^{\BMo}_{*+d}({(\rmd f)^*}\pr_1^{-1}(\Lambda)).
\]

 We have (\cite[\S 2]{schmid1996characteristic},\cite{ginsburg1986characteristic})
 \begin{equation}
 \label{equation:pbccsmooth}
  \CCy(f^*\mathscr{F})=(-1)^d{((\rmd f)^*)}_*\pr_1^*\CCy(\mathscr{F})\in \HO^{\BMo}_{n+d}({(\rmd f)^*}\pr_1^{-1}(\Lambda))
 \end{equation}
 
\begin{remark}
An implicit fact is that ${(\rmd f)^*}\pr_1^{-1}(\Lambda)$ is a closed, conical, Lagrangian subvariety of $\Tan^*Y$. It can be explained as follows. If $\Lambda\subset \Tan^*X$ is a closed conical Lagrangian subvariety, we have $\Lambda=\overline{\Tan^*_ZX}$ for some locally closed smooth subvariety $Z$ of $X$. Then, ${(\rmd f)^*}\pr_1^{-1}(\Lambda)=\overline{\Tan^*_{f^{-1}(Z)}Y}$ is also closed, conical and Lagrangian.

It is sufficient to ask $f$ to satisfy a weaker property (namely $f$ \emph{non-characteristic for $\SSup(\mathscr{F})$}, \cite[Proposition 9.4.3]{kashiwara2013sheaves}) to have the pullback functoriality of the characteristic cycle, but we will not make use of this more general framework.

The presence of the sign $(-1)^d$ can be recovered as follows. If $\SF$ is a perverse sheaf on $X$, then $f^*\SF[d]$ is a perverse sheaf on $Y$ and so its characteristic cycle has to be positive according to our sign conventions.
\end{remark}
\begin{example}[Illustration of the pullback formula]
 Let $X$ be a smooth variety and $f\colon X\rightarrow \pt$ the unique map. Then, $\CCy(\underline{\BoQ}_{X})=(-1)^{\dim X}[\Tan^*_XX]$ and $f^*\underline{\BoQ}_{\pt}=\underline{\BoQ}_{X}$. Therefore we have $\CCy(f^*\underline{\BoQ}_{\pt})=(-1)^{\dim X}[\Tan^*_XX]$. The cotangent correspondence reads
 \[
  \Tan^*X\xleftarrow{\rmd f^*}X\xrightarrow{\pr}\pt
 \]
where $\rmd f^*$ is the zero section. Therefore, $(\rmd f^*)_*\pr^*\Tan^*_{\pt}\pt=[\Tan^*_XX]$ and so \eqref{equation:pbccsmooth} is verified in this case. 
\end{example}

\subsubsection{Proper pushforward} \label{item:properpf} This is the least trivial functorial property of the characteristic cycle among those presented here. Let $f\colon Y\rightarrow X$ be a proper map between smooth algebraic varieties. In particular, $f$ is local complete intersection and can be factorised as $Y\xrightarrow{g}Z\xrightarrow{h}X$, where $g$ is a regular closed immersion and $h$ is smooth (one may classically choose $Z=Y\times X$, $h$ the second projection and $g$ the graph of $f$). Then, the map $(\rmd f)^*\colon \Tan^*X\times_XY\rightarrow \Tan^*Y$ is local complete intersection of codimension $\dim Y-\dim X$. Indeed we can provide a factorisation of $(\rmd f)^*$ as the composition of a closed immersion of a smooth variety and a smooth map as follows.

The morphism $\Tan^*X\times_XY\rightarrow \Tan^*X\times_XZ$ is a closed immersion thanks to the diagram with Cartesian squares
   \begin{equation}
   \label{equation:closedimmpfCC}
    \begin{tikzcd}
	{\Tan^*X\times_XY} & Y \\
	{\Tan^*X\times_XZ} & Z \\
	{\Tan^*X} & X
	\arrow["g", from=1-2, to=2-2]
	\arrow["h", from=2-2, to=3-2]
	\arrow[from=2-1, to=3-1]
	\arrow[from=3-1, to=3-2]
	\arrow[from=2-1, to=2-2]
	\arrow[from=1-1, to=1-2]
	\arrow[from=1-1, to=2-1]
	\arrow["\lrcorner"{anchor=center, pos=0.125}, draw=none, from=2-1, to=3-2]
	\arrow["\lrcorner"{anchor=center, pos=0.125}, draw=none, from=1-1, to=2-2]
\end{tikzcd}
   \end{equation}

The morphism $\Tan^*X\times_XZ\xrightarrow{(\rmd h)^*}\Tan^*Z$ is a closed immersion since $h$ is smooth.

The composition $\Tan^*X\times_XY\rightarrow \Tan^*X\times_XZ\rightarrow \Tan^*Z$ factors through the closed immersion $\Tan^*Z\times_ZY\rightarrow \Tan^*Z$ thanks to the commutative diagram
\begin{equation}
\label{equation:properpfCC}
 \begin{tikzcd}
	{\Tan^*X\times_XY} & {\Tan^*X\times_XZ} \\
	{} & {} & \star \\
	&& {\Tan^*Z\times_ZY} & {\Tan^*Z} \\
	&& Y & Z
	\arrow[from=3-3, to=3-4]
	\arrow[from=3-3, to=4-3]
	\arrow["g"', from=4-3, to=4-4]
	\arrow[from=3-4, to=4-4]
	\arrow[from=1-1, to=1-2]
	\arrow[dashed, from=1-1, to=3-3]
	\arrow[bend right=30, from=1-1, to=4-3]
	\arrow[bend left=30, from=1-2, to=3-4]
\end{tikzcd}
\end{equation}
The dashed map is a closed immersion since all other maps in the square indicated by $\star$ are closed immersions.

The map $(\rmd g)^*\colon \Tan^*Z\times_ZY\rightarrow \Tan^*Y$ is smooth since $g$ is a closed immersion of a smooth subvariety.

Altogether, we found a factorisation of $(\rmd f)^*$
\[
 \Tan^*X\times_XY\rightarrow \Tan^*Z\times_ZY\rightarrow \Tan^*Y
\]
as a closed immersion followed by a smooth morphism. The variety $\Tan^*Z\times_ZY$ is smooth by the Cartesian square in \eqref{equation:properpfCC} since $Y$ and $\Tan^*Z\rightarrow Z$ are smooth. The variety $\Tan^*X\times_XY$ is smooth thanks to the outer Cartesian square in \eqref{equation:closedimmpfCC} and the fact that $\Tan^*X\rightarrow X$ and $Y$ are smooth.

Therefore, we have a pullback map
\begin{equation}
\label{equation:pbforproperpf}
 ((\rmd f)^*)^!\colon \HO^{\BMo}_{*}(\Tan^*Y)\rightarrow \HO^{\BMo}_{*+\dim X-\dim Y}(\Tan^*X\times_XY).
\end{equation}

and for a closed subset $\Lambda\subset \Tan^*Y$ fitting in a Cartesian square
\[
 \begin{tikzcd}
	{({(\rmd f)^*})^{-1}(\Lambda)} & {\Tan^*X\times_XY} \\
	\Lambda & {\Tan^*Y}
	\arrow[from=2-1, to=2-2]
	\arrow["{(\rmd f)^*}", from=1-2, to=2-2]
	\arrow[from=1-1, to=2-1]
	\arrow[from=1-1, to=1-2]
\end{tikzcd},
\]
a refined pull-back
\begin{equation}
 \label{equation:refinedpbforproperpf}
 ((\rmd f)^*)^!\colon \HO^{\BMo}_{*}(\Lambda)\rightarrow \HO^{\BMo}_{*+\dim X-\dim Y}(({(\rmd f)^*})^{-1}(\Lambda))
\end{equation}
constructed as in \cite[Chapter 6]{fulton2013intersection}.

The projection $\pr_1\colon \Tan^*X\times_XY\rightarrow \Tan^*X$ is proper since it fits in the Cartesian diagram
\[
 \begin{tikzcd}
	{\Tan^*X\times_XY} & Y \\
	{\Tan^*X} & X
	\arrow["\pr_1"',from=1-1, to=2-1]
	\arrow[from=2-1, to=2-2]
	\arrow["f", from=1-2, to=2-2]
	\arrow[from=1-1, to=1-2]
	\arrow["\lrcorner"{anchor=center, pos=0.125}, draw=none, from=1-1, to=2-2]
\end{tikzcd}
\]

The restriction $({(\rmd f)^*})^{-1}(\Lambda)\rightarrow \pr_1({(\rmd f)^*})^{-1}(\Lambda)$ of $\pr_1$ is also proper and therefore induces a pushforward in Borel--Moore homology
\begin{equation}
 \label{equation:pfforpfCC}
 (\pr_1)_*\colon \HO^{\BMo}_*(({(\rmd f)^*})^{-1}(\Lambda))\rightarrow \HO^{\BMo}_*(\pr_1({(\rmd f)^*})^{-1}(\Lambda)).
\end{equation}

By composing \eqref{equation:refinedpbforproperpf} and \eqref{equation:pfforpfCC} and specialising the degree $*$ to $\dim Y$, we get a map
\begin{equation}
 \label{equation:pfCC}
 f_*=(\pr_1)_*((\rmd f)^*)^!\colon \HO^{\BMo}_{\dim Y}(\Lambda)\rightarrow \HO^{\BMo}_{\dim X}(\pr_1({(\rmd f)^*})^{-1}(\Lambda)). 
\end{equation}
 
 Let $\mathscr{F}\in \CD^{\rmb}_{\rmc}(Y,\BoQ)$ be a constructible sheaf. Let $\mathcal{S}_X$ be a Whitney stratification of $X$, $\mathcal{S}_Y$ a Whitney stratification of $Y$ for which $\mathscr{F}$ is constructible, such that $f$ is stratified for these stratifications (in the sense of \cite[I.1.1.6]{goresky1988stratified}). By \cite[1.7]{goresky1988stratified}, such stratifications exist.

\begin{lemma}
\label{lemma:onclosedLagr}
 Let $f\colon Y\rightarrow X$ be a proper morphism between smooth algebraic varieties. Recall the cotangent correspondence
 \[
  \begin{tikzcd}
	{\Tan^*X} & {\Tan^*X\times_XY} & {\Tan^*Y}
	\arrow["{{(\rmd f)^*}}", from=1-2, to=1-3]
	\arrow["{\pr_1}"', from=1-2, to=1-1]
\end{tikzcd}.
 \]
Let $\mathcal{S}_Y$ and $\mathcal{S}_Y$ be Whitney stratifications of $X$, $Y$ such that $f$ is stratified (in the sense of \cite{goresky1988stratified}). Then, if $\Lambda\coloneqq\bigcup_{S\in \mathcal{S}_Y}\Tan^*_SY\subset \Tan^*Y$, then $\Lambda$ is a closed conical Lagrangian subvariety of $\Tan^*Y$ and $\pr_1({(\rmd f)^*})^{-1}(\Lambda)\subset \Tan^*X$ is a closed, conical and Lagrangian subvariety of $\Tan^*X$.
\end{lemma}

\begin{proof}
The Whitney property of the stratification $\mathcal{S}_Y$ implies that $\Lambda$ is closed in $\Tan^*Y$. The map $\pr_1$ is proper and both $\pr_1, {(\rmd f)^*}$ are equivariant with respect to the $\mathbb{G}_{\mathrm{m}}$-actions rescaling cotangent fibers. Therefore, $\pr_1({(\rmd f)^*})^{-1}(\Lambda)\subset \Tan^*X$ is closed and conical.

Let $S\in\mathcal{S}_Y$. By definition of stratified maps, there exists $T\in\mathcal{S}_X$ such that $f(S)=T$ and $f_{|S}\colon S\rightarrow T$ is smooth. We show that $\pr_1({(\rmd f)^*})^{-1}(\Tan^*_SY)=\Tan^*_TX$. Indeed, if $(\xi,y)\in \Tan^*X\times_XY$, ${(\rmd f)^*}(\xi,y)=(y,\xi\circ \rmd f(y))$. Then, since $f_{|S}$ is submersive, $\xi\circ \rmd f(y)$ vanishes on the tangent space to $S$ at $y$ if and only if $\xi$ vanishes on the tangent space of $T$ at $f(y)$.

In conclusion, $\pr_1({(\rmd f)^*})^{-1}(\Lambda)=\bigcup_{T\in\mathcal{S}_X}\Tan^*_TX$ is closed, conical and Lagrangian in $\Tan^*X$.
\end{proof}

 We let $\Lambda=\bigcup_{S\in\mathcal{S}}\Tan^*_SY$. It contains $\supp(\CCy(\mathscr{F}))$. Then \cite[(2.17)]{schmid1996characteristic},
 \begin{equation}
 \label{equation:properpfcc}
  \CCy(f_*\mathscr{F})=f_*\CCy(\mathscr{F}).
 \end{equation}
 An implicit content of this last equality is that $\pr_1({(\rmd f)^*})^{-1}(\Lambda)$ is a closed, conical, Lagrangian subvariety of $\Tan^*X$. This is Lemma \ref{lemma:onclosedLagr}.
 
 \begin{example}[Pushforward to  point]
  Let $f\colon \BoP^1\rightarrow \pt$ be the map from the projective line to the point. Then, we have $\CCy(\underline{\BoQ}_{\BoP^1}[1])=[\Tan^*_{\BoP^1}\BoP^1]$. Moreover, $f_*\underline{\BoQ}_{\BoP^1}[1]\cong \BoQ[-1]\oplus\BoQ[1]$ and therefore
  \[
   \CCy(f_*\underline{\BoQ}_{\BoP^1}[1])=-2[\Tan^*_{\pt}\pt].
  \]
On the other hand, the cotangent correspondence reads
\[
 \Tan^*\BoP^1\xleftarrow{(\rmd f)^*} \BoP^1\xrightarrow{\pr} \pt
\]
where $(\rmd f)^*$ is the inclusion of the zero-section. Therefore, by \eqref{equation:properpfcc},
\[
 \CCy(f_*\underline{\BoQ}_{\BoP^1}[1])=\pr_*((\rmd f)^*)^![\Tan^*_{\BoP^1}\BoP^1]
\]
and $((\rmd f)^*)^![\Tan^*_{\BoP^1}\BoP^1]$ is the self-intersection of the zero-section of $\Tan^*\BoP^1$. It is given by the canonical divisor of $\BoP^1$, which is $-2$ times the class of a point. Therefore, we have
\[
 \pr_*((\rmd f)^*)^![\Tan^*_{\BoP^1}\BoP^1]
 =-2 [\Tan^*_{\pt}\pt].
\]
This illustrates \eqref{equation:properpfcc}.
\end{example}

\begin{example}[Inclusion of a subspace]
Let $f\colon\BoA^1\rightarrow\BoA^2$ be a linear inclusion of the affine line. Then, we have
\[
 \CCy(\underline{\BoQ}_{\BoA^1}[1])=[\Tan^*_{\BoA^1}\BoA^1]
\]
and
\[
 \CCy(f_*\underline{\BoQ}_{\BoA^1}[1])=[\Tan^*_{\BoA^1}\BoA^2].
\]
The cotangent correspondence reads
\[
 \Tan^*\BoA^1\xleftarrow{(\rmd f^*)}\BoA^1\times_{\BoA^2}\Tan^*\BoA^2\xrightarrow{\pr_2} \Tan^*\BoA^2.
\]
With the identification $\Tan^*\BoA^1\cong \BoA^2$ and $\Tan^*\BoA^2\cong\BoA^4$, we have $\BoA^1\times_{\BoA^2}\Tan^*\BoA^2\cong \BoA^3$ and we can identify
\[
 \rmd f^*\colon (x,y,z)\mapsto (x,y)
\]
and
\[
 \pr_2\colon (x,y,z)\mapsto (x,0,y,z).
\]
We immediately obtain
\[
 \pr_2(\rmd f^*)^![\Tan^*_{\BoA^1}\BoA^1]=[\Tan^*_{\BoA^1}\BoA^2]
\]
which confirms \eqref{equation:properpfcc}.
\end{example}

The \emph{singular support} of a constructible complex $\mathscr{F}\in \CD^{\rmb}_{\rmc}(X)$ can be defined as $\SSup(\SF)=\supp \CCy(\mathscr{F})$ if $\SF$ is a perverse sheaf, and as the union $\SSup(\SF)=\cup_{j\in\BoZ}\SSup(\pH{0}(\SF))$ in general \cite{kashiwara2013sheaves}.

By considering the full triangulated subcategory $\CD^{\rmb}_{\rmc}(X,\Lambda)$ of constructible complexes whose singular support is contained in $\Lambda$, the characteristic cycle can be seen as a map
\[
 \CCy\colon \GK_0(\CD^{\rmb}_{\rmc}(X,\Lambda))\rightarrow \HO^{\BMo}_{\dim X}(\Lambda,\BoZ).
\]

Sometimes it is convenient, instead of fixing the singular support, to fix the stratification. If $\mathcal{S}$ is a stratification of $X$, we let $\CD^{\rmb}_{\rmc,\mathcal{S}}(X)$ be the full subcategory of $\mathcal{S}$-constructible complexes. If $\mathcal{S}$ is Whitney and admissible (in the sense of \cite[4.1]{dimca2004sheaves}) and $\Lambda=\bigsqcup_{S\in\mathcal{S}}\Tan^*_SX$, then $\CD^{\rmb}_{\rmc,\mathcal{S}}(X)=\CD^{\rmb}_{\rmc}(X,\Lambda)$ \cite[Theorem 4.3.15 iv)]{dimca2004sheaves}.

\subsection{Characteristic cycle of equivariant complexes}
\label{subsection:characteristiccycleequivariantcomplexes}
Let $X$ be a smooth algebraic variety acted on by an algebraic group $G$. We let $\CD^{\rmb}_{\rmc,G}(X)$ be the $G$-equivariant constructible derived category \cite{bernstein2006equivariant}. There is a forgetful functor
\[
\begin{matrix}
 \For&\colon& \CD^{\rmb}_{\rmc,G}(X)&\rightarrow& \CD^{\rmb}_{\rmc}(X)\\
 &&\mathscr{F}&\mapsto&\mathscr{F}[\dim G]
\end{matrix}
\]
The shift ensures that $\For$ is perverse $t$-exact; we use the perverse $t$-structure on $\CD^{\rmb}_{\rmc,G}(X)$ such that $\mathscr{L}[\dim X-\dim G]$ is perverse for $G$-equivariant local systems $\mathscr{L}$ on $X$. This convention is convenient since the dimension of the quotient stack $X/G$ is $\dim X-\dim G$. Let $\mathcal{S}_X$ be a $G$-equivariant Whitney stratification and $\CD^{\rmb}_{\rmc,G,\mathcal{S}}(X)$ be the full subcategory of $\mathcal{S}$-constructible $G$-equivariant complexes. We let $\Lambda=\bigsqcup_{S\in\mathcal{S}_X}\Tan^*_SX$. We define the characteristic cycle map for equivariant complexes as the composition
\[
 \CCy\colon \GK_0(\CD^{\rmb}_{\rmc,G,\mathcal{S}}(X))\xrightarrow{\For}\GK_0(\CD^{\rmb}_{\rmc,\mathcal{S}}(X))\xrightarrow{\CCy} \HO^{\BMo}_{\dim X}(\Lambda)\cong \HO^{\BMo,G}_{\dim X-\dim G}(\Lambda),
\]
where $\HO^{\BMo,G}_*(\Lambda)$ is the $G$-equivariant Borel-Moore homology of $\Lambda$.

\section{Induction diagrams and group actions}
\subsection{The induction diagram}
\label{subsection:inductiondiagram}
We consider a correspondence
\[
 X\xleftarrow{f}W\xrightarrow{g}X'
\]
where
\begin{enumerate}
 \item $X, X', W$ are smooth complex varieties,
 \item $f$ is smooth (of relative dimension $d$),
 \item $g$ is proper,
 \item $(f,g)\colon W\rightarrow X\times X'$ is a closed embedding.
\end{enumerate}

\subsection{The equivariant induction diagram}
\label{subsection:equivariantinductiondiagram}
Following \cite[\S 2]{schiffmann2012hall}, we consider a correspondence
\begin{equation}
\label{equation:preequivindd}
 Y\xleftarrow{p}V\xrightarrow{q}X'
\end{equation}
where
\begin{enumerate}
 \item $Y,X',V$ are smooth complex varieties,
 \item $q$ is a closed immersion,
 \item $p$ is smooth (of relative dimension $d$).
\end{enumerate}

Let $P\subset G$ be a pair of algebraic groups. We assume that $P$ acts on $Y$ and $V$, $G$ acts on $X'$ and that $p,q$ are equivariant.

Let $X=G\times^PY$ and $W=G\times^PV$. We obtain a new correspondence, of the type of \S \ref{subsection:inductiondiagram}.
\begin{equation}
\label{equation:equivariantinductiondiagram} 
 X\xleftarrow{f} W\xrightarrow{g}X'
\end{equation}
where $f(h,v)=(h,p(v))$ and $g(h,v)=h\cdot q(v)$.

We have the following properties
\begin{enumerate}
 \item The varieties $X,W,X'$ are smooth,
 \item The maps $f,g$ are $G$-equivariant,
 \item $(f,g)\colon W\rightarrow X\times X'$ is a closed immersion,
 \item $f$ is smooth of relative dimension $d$,
 \item If $G/P$ is proper, then $g$ is proper.
\end{enumerate}
The last statement comes from the fact that the map $g$ is the composition 
\[
\begin{matrix}
 W=G\times^PV&\rightarrow& G\times^PX'&\rightarrow& X'\\
 (h,v)&\mapsto&(h,q(v))&&\\
 &&(h,x')&\mapsto&(h,hx')
 \end{matrix}
\]
The first morphism is a closed immersion and up to the isomorphism $G\times^PX'\rightarrow G/P\times X'$, $(h,x')\mapsto (h,h\cdot x')$, the second morphism is the projection on the second factor and is therefore proper if $G/P$ is proper.

\subsection{The stacky induction diagram}
By quotienting the diagram \eqref{equation:equivariantinductiondiagram} by $G$, we obtain the stacky induction diagram
\begin{equation}
\label{equation:stackyinddiag}
X/G\xleftarrow{\overline{f}} W/G\xrightarrow{\overline{g}}X'/G.
\end{equation}
It is equivalent to the diagram
\begin{equation}
\label{equation:stackyind2}
 Y/P\xleftarrow{\overline{p}}V/P\xrightarrow{\overline{q}}X'/G.
\end{equation}
In our practical situations motivated by Hall algebras, the group $P$ is a parabolic subgroup of some $\GL_n$ and acts on $Y$ through a quotient $P\rightarrow L$ to the Levi subgroup. In this case, we can complete the diagram \eqref{equation:stackyind2}:
\begin{equation}
 \label{equation:stackyind3}
 Y/L\xleftarrow{p'} Y/P\xleftarrow{\overline{p}}V/P\xrightarrow{\overline{q}}X'/G.
\end{equation}
We denote by $\overline{p'}$ the composition $\overline{p'}\coloneqq p'\circ \overline{p}$. Note that the map $p'$ is smooth.

\subsection{The moment maps}
When a group acts on a smooth algebraic variety $X$, there is a natural choice of moment map for the induced Hamiltonian action of $G$ on $\Tan^*X$. It is given by $\mu_{X,G}\colon \Tan^*X\rightarrow\mathfrak{g}^*$, $(x,l)\mapsto (\xi\mapsto l(\xi_x))$ where $\xi$ is the infinitesimal action of $G$ on $X$ (a vector field on $X$). This is our default choice.

We let $\mathfrak{g}$ (resp. $\mathfrak{p}$) be the Lie algebra of $G$ (resp. $P$). The $P$-action on $Y$ induces a $P$-action on $\Tan^*Y$. It is Hamiltonian with moment map
\[
 \mu_{Y,P}\colon \Tan^*Y\rightarrow \mathfrak{p}^*.
\]

Similarly, the $G$-action on $X\coloneqq G\times^PY$ gives a Hamiltonian action on $\Tan^*X$ with moment map
\[
 \mu_{X,G}\colon \Tan^*X\rightarrow \mathfrak{g^*}.
\]

\begin{proposition}
\label{proposition:isoindcotangent}
 We have a natural isomorphism of $G$-varieties
\[
 \mu_{X,G}^{-1}(0)\cong G\times^P\mu_{Y,P}^{-1}(0).
\]
\end{proposition}

Before proving Proposition \ref{proposition:isoindcotangent}, we give the following well-known lemma.

Let $Z$ be a smooth algebraic variety acted on by a group $P$. Then, we have an Hamiltonian action of $P$ on $\Tan^*Z$. We have the moment map
\[
 \mu_{Z,P}\colon \Tan^*Z\rightarrow \mathfrak{p}^*
\]

\begin{lemma}
\label{lemma:cotangentfreequotient}
 If $Z$ is an algebraic variety on which an algebraic group $P$ acts freely, we have $\Tan^*(Z/P)\cong \mu_{Z,P}^{-1}(0)/P$ as stacks.
\end{lemma}
In intuitive terms, $\mu_{Z,P}^{-1}(0)$ is the closed subset of $\Tan^*Z$ of $(z,l)\in\{z\}\times (\Tan_z Z)^*$ such that $l$ vanishes on the tangent space to the $P$-orbit through $z$.

\begin{proof}[Proof of Proposition \ref{proposition:isoindcotangent}]
 The group $G\times P$ acts on the product $G\times Y$ by $(g,p)\cdot (g',y)=(gg'p,p^{-1}y)$. Let 
 \[
  \mu_{G\times Y,P}\colon \Tan^*(G\times Y)\rightarrow \mathfrak{p}^*
 \]
be the moment map for the action of $P\cong\{1\}\times P$ on $\Tan^*(G\times Y)$. The moment map of the action of $G\times P$ on $\Tan^*(G\times Y)\cong G\times \mathfrak{g}^*\times \Tan^*Y$ is
 \[
 \begin{matrix}
  \mu_{G\times Y,G\times P}=\mu_{G,G}\times\mu_{G\times Y,P}&\colon& G\times \mathfrak{g}^*\times \Tan^*Y&\rightarrow& \mathfrak{g}^*\times\mathfrak{p}^*\\
  &&v=(g,\xi,y,\zeta)&\mapsto&[(s,t)\mapsto \xi(\ad(g)(s))+\mu_{G\times X,P}(v)(t)]
  \end{matrix}
 \]
By Lemma \ref{lemma:cotangentfreequotient} applied to the variety $G\times Y$ and the group $P$, we have $\Tan^*(G\times^PX)=\mu_{G\times X,P}^{-1}(0)/P$. Moreover, $\mu_{G\times P}^{-1}(\{0\}\times \mathfrak{p}^*)=G\times \Tan^*Y$. The algebraic variety $G\times \Tan^*Y$ is acted on diagonally by $P$. Consider the commutative diagram
\[
 \begin{tikzcd}
	{\mu_{G\times^PY,G}^{-1}(\{0\}) }\\
	{\mu^{-1}_{G\times Y,G\times P}(\mathfrak{g}^*\times \{0\})/P=\Tan^*(G\times^PY)} & {\mathfrak{g}^*} \\
	{\mu^{-1}_{G\times Y,G\times P}(\mathfrak{g}^*\times \{0\})} & {\mathfrak{g}^*} \\
	{\Tan^*(G\times Y)} & {\mathfrak{g}^*\times\mathfrak{p}^*} \\
	{\mu^{-1}_{G\times Y,G\times P}(\{0\}\times \mathfrak{p}^*)=G\times \Tan^*Y} & {\mathfrak{p}^*} \\
	{\mu_{G\times Y,G\times P}^{-1}(\{0\})=G\times \mu_{Y,P}^{-1}(0)} \\
	{G\times^P \mu_{Y,P}^{-1}(0)}
	\arrow["{\mu_{G\times Y,G\times P}}", from=4-1, to=4-2]
	\arrow["{\mu_{Y,P}\circ \pr_2}", from=5-1, to=5-2]
	\arrow[from=5-1, to=4-1]
	\arrow["{0\times \id}"', from=5-2, to=4-2]
	\arrow[from=6-1, to=5-1]
	\arrow[from=6-1, to=7-1]
	\arrow[from=3-1, to=4-1]
	\arrow["{\id\times 0}", from=3-2, to=4-2]
	\arrow[from=3-1, to=3-2]
	\arrow[from=3-1, to=2-1]
	\arrow["\id",from=2-2, to=3-2]
	\arrow["{\mu_{G\times^PY,G}}", from=2-1, to=2-2]
	\arrow[from=1-1, to=2-1]
\end{tikzcd}
\]
By analysing this diagram from the middle row and going up or down (that is, performing Hamiltonian reduction by the action of $G$ and then by the action of $P$, or in the other way), we obtain
\[
 \mu_{G\times^P Y,G}^{-1}(0)\cong G\times^P\mu_{Y,P}^{-1}(0)
\]
as $G$-varieties.
\end{proof}

\subsection{The cotangent correspondences}
\label{subsection:cotangentcorrespondences}
Recall the induction diagram in \S\ref{subsection:inductiondiagram}. We let $Z=\Tan^*_W(X\times X')$. The natural projections give the \emph{cotangent correspondence}
\[
 \Tan^*X\xleftarrow{\phi}Z\xrightarrow{\psi} \Tan^*X'.
\]

The cotangent correspondence of $f$ reads
\[
 \Tan^*X\xleftarrow{\pr_X}\Tan^*X\times_XW\xrightarrow{(\rmd f)^*} \Tan^*W.
\]
The map $\pr_X$ is smooth (obtained by base-change from $f$) and $(\rmd f)^*$ is a closed immersion (pullback of covectors by a smooth map).

The cotangent correspondence of $g$ reads
\[
 \Tan^*X'\xleftarrow{\pr_{X'}}\Tan^*X'\times_{X'}W\xrightarrow{(\rmd g)^*} \Tan^*W.
\]
The three morphisms
\[
\begin{tikzcd}
	& {\Tan^*_W(X\times X')} \\
	{\Tan^*X} & {\Tan^*X'} & W
	\arrow["{\pr_1}"', from=1-2, to=2-1]
	\arrow["{-\pr_2}", from=1-2, to=2-2]
	\arrow[from=1-2, to=2-3]
\end{tikzcd}
\]
(where $-\pr_2$ is the composition of the natural morphisms $\Tan^*_W(X\times X')\rightarrow \Tan^*(X\times X')\rightarrow \Tan^*X'$ followed by the automorphism of $\Tan^*X'$ acting by $(-1)$ on the fibers) gives maps
\[
 \begin{tikzcd}
	{\Tan^*_W(X\times X')} & {\Tan^*X'\times_{X'}W} \\
	{\Tan^*X\times_XW}
	\arrow["{\phi'}"', from=1-1, to=2-1]
	\arrow["{\psi'}", from=1-1, to=1-2]
\end{tikzcd}
\]
so that the three cotangent correspondences fit in the diagram with Cartesian square
\begin{equation}
 \label{equation:comparison}
 \begin{tikzcd}
	{\Tan^*_W(X\times X')} & {\Tan^*X'\times_{X'}W} & {\Tan^*X'} \\
	{\Tan^*X\times_XW} & {\Tan^*W} \\
	{\Tan^*X}
	\arrow["{\phi'}"', from=1-1, to=2-1]
	\arrow["{\psi'}", from=1-1, to=1-2]
	\arrow["{\pr_X}"', from=2-1, to=3-1]
	\arrow["{\pr_{X'}}", from=1-2, to=1-3]
	\arrow["{(\rmd g)^*}", from=1-2, to=2-2]
	\arrow["{(\rmd f)^*}"', from=2-1, to=2-2]
	\arrow["\ulcorner"{anchor=center, pos=0.125}, draw=none, from=1-1, to=2-2]
	\arrow["\phi"', bend right=80, from=1-1, to=3-1]
	\arrow["-\psi", bend left=30, from=1-1, to=1-3]
\end{tikzcd}
\end{equation}
where $-\psi$ is the composition of $\psi$ and the automorphism of $\Tan^*X'$ acting by $(-1)$ on the fibers.

\begin{lemma}
\begin{enumerate}
 \item The morphism $\psi$ is proper and the morphism $\phi$ is local complete intersection (l.c.i., \cite[\S 6.6]{fulton2013intersection}), of codimension $(\dim X-\dim X')$.
 \item The morphisms $\phi'$ and $(\rmd g)^*$ are l.c.i. of the same codimension $(\dim W-\dim X')$ (i.e. the Cartesian square in \eqref{equation:comparison} has no excess intersection bundle).
\end{enumerate}
\end{lemma}
\begin{proof}
 It is equivalent to prove that $-\psi$ is proper. We have $-\psi=\pr_{X'}\circ \psi'$. The projection $\pr_{X'}$ is proper as it is obtained by base-change from the proper map $W\rightarrow X'$. The map $\psi'$ is a closed immersion as it it a pull-back of the closed immersion $(\rmd f)^*$.
 
 The morphism $\phi$ is l.c.i. as both $\Tan^*_W(X\times X')$ and $\Tan^*X$ are smooth varieties.
 
 The map $\pr_X$ is smooth, as obtained by base-change from $f$. Therefore, $\Tan^*X\times_XW$ is smooth and $\phi'$ is l.c.i. Similarly, $\Tan^*X'\times_{X'}W\rightarrow W$ is obtained by base-change from $\Tan^*X'\rightarrow X'$. The former map is therefore smooth and since $W$ is smooth, $\Tan^*X'\times_{X'}W$ is smooth. Therefore, $(\rmd g)^*$ is l.c.i. The codimensions of $\phi'$ and $(\rmd g)^*$ are given by a straightforward computation and both equal $\dim W-\dim X'$. 
\end{proof}

\subsection{Equivariant cotangent correspondences}
Consider the diagram of \S\ref{subsection:inductiondiagram}. If an algebraic group $G$ acts on the varieties $X,W,X'$ such that $f$ and $g$ are $G$-equivariant, then the group $G$ acts naturally on all the varieties of \S\ref{subsection:cotangentcorrespondences}, all the maps of \S\ref{subsection:cotangentcorrespondences} are $G$-equivariant. This situation appears for example with the induction diagram \eqref{equation:equivariantinductiondiagram}.

\section{Inductions}
\subsection{Induction of constructible complexes}
\label{subsection:inductioncstblecomplexes}
Consider the induction diagram of \S \ref{subsection:inductiondiagram}. We define the induction functor
\[
 \begin{matrix}
  \Ind&\colon&\CD^{\rmb}_{\rmc}(X)&\rightarrow& \CD^{\rmb}_{\rmc}(X')\\
  &&\mathscr{F}&\mapsto&g_*f^*\mathscr{F}[d].
 \end{matrix}
\]
By the decomposition theorem, this functor preserves semisimple complexes. We also denote by $\Ind\colon \GK_0(\CD^{\rmb}_{\rmc}(X))\rightarrow \GK_0(\CD^{\rmb}(X'))$ the map induced on the Grothendieck groups.

For $\Lambda\subset \Tan^*X$ and $\Lambda'\subset \Tan^*X'$ such that $\pr_{X'}((\rmd g)^*)^{-1}(\rmd f)^*\pr_X^{-1}(\Lambda)\subset\Lambda'$, $\Ind$ induces a functor
\[
 \Ind\colon \CD^{\rmb}_{\rmc}(X,\Lambda)\rightarrow \CD^{\rmb}_{\rmc}(X',\Lambda').
\]
and a map $\Ind\colon \GK_0(\CD^{\rmb}_{\rmc}(X,\Lambda))\rightarrow \GK_0(\CD^{\rmb}_{\rmc}(X',\Lambda'))$ (we made use of the inclusion properties of singular support for smooth pullback \cite[Proposition 5.4.5 i)]{kashiwara2013sheaves} and proper pushforward \cite[Proposition 5.4.4]{kashiwara2013sheaves}).

\subsection{Equivariant induction of constructible complexes}
\label{subsection:equivariantinduction}
Consider the diagram \eqref{equation:equivariantinductiondiagram}. We can define an induction functor in the equivariant setting using the same formula as in the non-equivariant setting (\S \eqref{subsection:inductioncstblecomplexes}):
\begin{equation}
\label{equation:inductionequivariant}
 \Ind=g_*f^*[d]\colon \CD^{\rmb}_{\rmc,G}(X)\rightarrow \CD^{\rmb}_{\rmc,G}(X')
\end{equation}

We now define an induction functor
\begin{equation}
\label{equation:inductionfunctor}
 \Ind\colon \CD^{\rmb}_{\rmc,P}(Y)\rightarrow \CD^{\rmb}_{\rmc,G}(X')
\end{equation}
by 
composing \eqref{equation:inductionequivariant} with the induction equivalence
\[
 \Ind_P^G\colon \CD^{\rmb}_{\rmc,P}(Y)\rightarrow \CD^{\rmb}_{\rmc,G}(X)
\]
defined in the following way. Consider the correspondence
\begin{equation}
\label{equation:diagramindPG}
 Y\xleftarrow{\pr_2} G\times Y\xrightarrow{\pr} X
\end{equation}
For $\mathscr{F}\in \CD^{\rmb}_{\rmc,P}(Y)$, $\pr_2^*\mathscr{F}$ is $G\times P$-equivariant on $G\times Y$ and therefore, there exists a unique complex $\mathscr{G}\in \CD^{\rmb}_{\rmc,G}(X')$ such that $\pr_2^*\cong \pr^*\mathscr{G}$. We set $\Ind_{P}^G\mathscr{F}=\mathscr{G}[\dim G-\dim P]$. The shifts ensure that the induction equivalence $\Ind_P^G$ is perverse $t$-exact.

If $L$ is an algebraic group acting on $Y$ such that the action of $P$ factors through a quotient $P\rightarrow L$, we define the induction functor
\begin{equation}
\label{equation:indLG}
 \Ind\colon \CD^{\rmb}_{\rmc,L}(Y)\rightarrow \CD^{\rmb}_{\rmc,G}(X')
\end{equation}
by composing \eqref{equation:inductionfunctor} with the pull-back by the equivariant map $\zeta\colon (Y,P)\rightarrow(Y,L)$,
\begin{equation}
\label{equation:LtoP}
\begin{matrix}
  \CD^{\rmb}_{\rmc,L}(Y)&\rightarrow& \CD^{\rmb}_{\rmc,P}(Y)\\
  \mathscr{F}&\mapsto&\zeta^*\mathscr{F}[\dim L-\dim P]
 \end{matrix}
\end{equation}
(with $\zeta^*$ defined in \cite[\S 6.5]{bernstein2006equivariant}), and again the shift ensures that this functor is perverse $t$-exact.

We decomposed the definition of the induction function \eqref{equation:indLG} to facilitate the commutation of the characteristic cycle and the inductions. The functor \eqref{equation:indLG} coincides with the functor defined using the induction diagram between stacks \eqref{equation:stackyind3}
\[
 \begin{matrix}
  \Ind&\colon& \CD^{\rmb}_{\rmc}(Y/L)&\rightarrow&\CD^{\rmb}_{\rmc}(X'/G)\\
      &      & \SF                &\mapsto    &\overline{q}_*\overline{p'}^*\SF[\dim V/P-\dim Y/L].
 \end{matrix}
\]

\subsection{Induction map in Borel--Moore homology}
\begin{lemma}
 Let $X$ be a smooth variety and $a\colon \Tan^*X\rightarrow \Tan^*X$ the automorphism acting by $(-1)$ on fibers. Let $\Lambda\subset \Tan^*Y$ be an $\mathbb{G}_{\mathrm{m}}$ stable subvariety (for the natural $\mathbb{G}_{\mathrm{m}}$-action rescaling the fibers). Then, the induced automorphism $\HO^{\BMo}_*(\Lambda,\BoZ)\rightarrow \HO^{\BMo}_*(\Lambda,\BoZ)$ is the identity.
\end{lemma}
\begin{proof}
 The action of $\BoZ/2\BoZ$ on $\HO^{\BMo}_*(\Lambda)$ extends to a $\mathbb{G}_{\mathrm{m}}$ action. Since $\mathbb{G}_{\mathrm{m}}$ is a connected algebraic group, this action is trivial.
\end{proof}
\begin{corollary}
\label{corollary:mapsBMcoincide}
 For any $\Lambda\subset Z$ and $\Lambda'\subset \Tan^*X'$ such that $\Lambda'$ is $\mathbb{G}_{\mathrm{m}}$-stable and $\psi(\Lambda)\subset \Lambda'$, the two maps $\psi_*, (-\psi)_*\colon \HO^{\BMo}_*(\Lambda)\rightarrow \HO^{\BMo}_{*}(\Lambda')$ coincide.
\end{corollary}

Let $\Lambda\subset \Tan^*X$ and $\Lambda'\subset \Tan^*X'$. Assume that $(-\psi)\phi^{-1}(\Lambda)\subset \Lambda'$. Then, we can define a linear map
\[
 (-\psi)_*\phi^!\colon \HO^{\BMo}_*(\Lambda)\rightarrow \HO^{\BMo}_{*-\dim X+\dim X'}(\Lambda')
\]
by composing the virtual pull-back map $\phi^!\colon \HO^{\BMo}_*(\Lambda)\rightarrow \HO^{\BMo}_{*-\dim X+\dim X'}(\phi^{-1}(\Lambda))$ coming from the Cartesian square
\[
\begin{tikzcd}
	\Lambda & {\phi^{-1}(\Lambda)} \\
	{\Tan^*X} & Z
	\arrow["\phi", from=2-2, to=2-1]
	\arrow[from=1-2, to=2-2]
	\arrow[from=1-2, to=1-1]
	\arrow[from=1-1, to=2-1]
	\arrow["\ulcorner"{anchor=center, pos=0.125, rotate=-90}, draw=none, from=1-2, to=2-1]
\end{tikzcd}
\]
with the proper push-forward map $(-\psi)_*\colon \HO^{\BMo}_*(\phi^{-1}(\Lambda))\rightarrow \HO^{\BMo}_*(\Lambda')$ coming from the restriction $-\psi\colon \phi^{-1}(\Lambda)\rightarrow(-\psi)\phi^{-1}(\Lambda)$.

In particular, by letting $\Lambda$ and $\Lambda'$ be closed, conical, Lagrangian subvarieties of $\Tan^*X$ and $\Tan^*X'$ respectively such that $\psi\phi^{-1}(\Lambda)\subset \Lambda'$, we obtain a map
\begin{equation}
\label{equation:inductionBorelMoore}
 \Ind\colon \HO^{\BMo}_*(\Lambda)\rightarrow \HO^{\BMo}_{*-\dim X+\dim X'}(\Lambda').
\end{equation}
By taking the limit over all pairs $\Lambda,\Lambda'$ and only considering the Borel--Moore homology spaces of degrees $\dim X$ and $\dim X'$, we obtain a linear map
\[
 \Ind\colon\Lagr^{\mathbb{G}_{\mathrm{m}}}(\Tan^*X)\rightarrow \Lagr^{\mathbb{G}_{\mathrm{m}}}(\Tan^*X').
\]

\subsection{A conjecture}
Recall the induction diagram \S\ref{subsection:inductiondiagram}. We formulate a conjecture which would be crucial for understanding the connection between constructible functions on the seminilpotent stack and the top-CoHA for quiver, or more generally the dashed arrow in \eqref{equation:fundamentaldiagram}.
\begin{conjecture}
\label{conjecture:microlocalmultiplicities}
 We let $[X]\in \Lagr^{\mathbb{G}_{\mathrm{m}}}(\Tan^*X)$ be the fundamental class of $X$. We let $U=\{(w,l')\in W\times_{X'}\Tan^*X'\mid l'\circ \rmd g(w)=0\}$. We have a map $\xi\colon U\rightarrow \Tan^*X'$, $(w,l')\mapsto (g(w),l')$. Then, for any $\Lambda\in\Lagr^{\mathbb{G}_{\mathrm{m}}}(\Tan^*X')$, $\xi^{-1}(x',l')$ is smooth for $(x',l')\in\Lambda$ general enough. We fix $x_{\Lambda}=(x',l')$ satisfying this property. Then, we have
 \[
  \Ind([X])=\sum_{\Lambda'\in\Lagr^{\mathbb{G}_{\mathrm{m}}}(\Tan^*X')}(-1)^{\dim(\Lambda'\cap X')}\chi(\xi^{-1}(x_{\Lambda'}))[\Lambda'].
 \]
\end{conjecture}
We can prove this conjecture for $\Lambda'=\overline{\Tan^*_SX'}$ where $S$ is the smooth locus of the image of $g$.

\begin{theorem}
 Let $U$ be a smooth open subset of $g(W)$ and $\Lambda'=\overline{\Tan^*_UX'}$. Then with the notations of Conjecture \ref{conjecture:microlocalmultiplicities}, the multiplicity of $[\Lambda']$ in $\Ind([X])$ is $(-1)^{\dim U}(\xi^{-1}(x_{\Lambda'}))$.
\end{theorem}
\begin{proof}
 We can assume that $g$ is surjective, since we can restrict $g$ to $g^{-1}(U)\rightarrow U$ instead and by shrinking $U$ even more, that $g$ is a topological fibration (since it is proper). We then have $\Lambda'=\Tan^*_{X'}X'$.
 
 In this case, $\xi^{-1}(x_{\Lambda'})=g^{-1}(x')$. The multiplicity of $[\Lambda']$ in $\Ind([X])$ is given by the degree of $g_*e(h^*N/N')$, where in the diagram
 \[
  \begin{tikzcd}
	{X'} \\
	W & {\Tan^*X'\times_{X'}W} \\
	W & {\Tan^*W}
	\arrow["h", from=2-1, to=2-2]
	\arrow["{(\rmd g)^*}", from=2-2, to=3-2]
	\arrow["{0_{\Tan^*W}}"', from=3-1, to=3-2]
	\arrow["{\id_W}"', from=2-1, to=3-1]
	\arrow["\ulcorner"{anchor=center, pos=0.125}, draw=none, from=2-1, to=3-2]
	\arrow["g", from=2-1, to=1-1]
\end{tikzcd},
 \]
$N$ is the normal bundle of $(\rmd g)^*$, $N'=0$ is the normal bundle of $\id_W$. The vector bundle $h^*N$ on $W$ is the cotangent bundle of $g:W\rightarrow X'$, which gives the result, as the self-intersection of a smooth algebraic variety $U$ in its cotangent bundle is $(-1)^{\dim U}\chi(U)$.
\end{proof}

\subsection{Induction map in equivariant Borel--Moore homology}
\label{subsection:inductionequivariantBM}
Let $Y$ be a $P$-variety, where $P\subset G$ is a subgroup. We let $X=Y\times^PG$. If $\mathcal{S}_Y$ is a $P$-invariant Whitney stratification of $Y$, we let $\mathcal{S}_X=\{S\times^PG : S\in \mathcal{S}_Y\}$. It is a Whitney stratification. We have the cotangent actions of $P$ on $\Tan^*Y$ and $G$ on $\Tan^*X$. If $S\in\mathcal{S}_Y$, $\Tan^*_{S\times^PG}X=G\times^P(\Tan^*_SY)$. Therefore, if $\Lambda_Y=\bigsqcup_{S\in\mathcal{S}_Y}\Tan^*_SY$ and $\Lambda_X=\bigsqcup_{T\in\mathcal{S}_X}\Tan^*_TX$, we have $\Lambda_X=G\times^P\Lambda_Y$. We have an isomorphism in equivariant Borel--Moore homology
\[
 \Ind_P^G\colon \HO^{\BMo,P}_{*}(\Lambda_Y)\rightarrow \HO^{\BMo,G}_{*+\dim G-\dim P}(\Lambda_X).
\]
We have canonical isomorphisms
\[
 \HO^{\BMo,P}_{\dim Y-\dim P}(\Lambda_Y)\cong \HO^{\BMo}_{\dim Y}(\Lambda_Y)
\]
and
\[
 \HO^{\BMo,G}_{\dim X-\dim G}(\Lambda_X)\cong \HO^{\BMo}_{\dim X}(\Lambda_X)
\]
and so we have a natural induction isomorphism
\[
 \Ind_P^G\colon \HO^{\BMo}_{\dim Y}(\Lambda_Y)\cong \HO^{\BMo}_{\dim X}(\Lambda_X).
\]
This isomorphism sends the fundamental class of $\overline{\Tan^*_SY}$ to the fundamental class of $\overline{\Tan^*_{S\times^PG}X}$.

Considering the induction diagram \eqref{subsection:equivariantinductiondiagram}, for $\Lambda_Y\subset \Tan^*Y$, $\Lambda_X=G\times^P\Lambda_Y$ and $\Lambda'\subset \Tan^*X'$ such that $\psi\phi^{-1}\Lambda_X\subset \Lambda'$, the operation $\Ind_P^G$ and \eqref{equation:inductionBorelMoore} gives the induction map in Borel--Moore homology
\begin{equation}
\label{equation:indBMequivariant}
\Ind\colon \HO^{\BMo}_{\dim Y}(\Lambda_Y)\rightarrow \HO^{\BMo}_{\dim X'}(\Lambda').
\end{equation}

\subsection{Induction of constructible functions}
\label{subsection:indconstructiblefunctions}
Consider the diagram of \S \ref{subsection:inductiondiagram}. We define a map
\[
 \Ind\colon \Fun(X)\rightarrow\Fun(X')
\]
in the following way. For $h\in\Fun(X)$ and $x'\in X'$, we let
\[
 \Ind(h)(x')=(-1)^d\sum_{\rmc\in \BoZ}\chi_{\rmc}(g^{-1}(x')\cap f^{-1}h^{-1}(c))\cdot c,
\]
where $\chi_{\rmc}$ is the Euler characteristic with compact supports.

\subsection{Induction of equivariant constructible functions}
\label{subsection:indequicstblefunctions}
For a variety $Z$ with an $H$-action, we let $\Fun(Z,H)$ be the space of $\BoZ$-valued $H$-invariant constructible functions on $Z$.

Consider the equivariant induction diagram \eqref{equation:equivariantinductiondiagram}. We construct an induction map
\begin{equation}
\label{equation:indcstblefunctions}
 \Ind\colon \Fun(Y,P)\rightarrow \Fun(X',G).
\end{equation}

We first construct an induction isomorphism
\begin{equation}
\label{equation:inductionconstructiblefunctions}
 \Ind_P^G\colon\Fun(Y,P)\rightarrow \Fun(X,G).
\end{equation}
Consider the diagram \eqref{equation:diagramindPG}. If $\varphi\in\Fun(Y,P)$, $\pr_2^*\varphi$ is $G\times P$-invariant on $G\times Y$ and therefore descends to a $G$-invariant funtion $\varphi'$ on $X$. We set $\Ind\varphi=(-1)^{\dim G-\dim P}\varphi'$.

The induction map is then the composition of $\Ind_P^G$ with the induction defined in \S \ref{subsection:indconstructiblefunctions} (restricted to $G$-equivariant functions) obtained from the diagram \eqref{equation:equivariantinductiondiagram}.

If the $P$-action on $Y$ comes from an $L$-action and a morphism $P\rightarrow L$, we have an induction
\[
 \Ind\colon \Fun(Y,L)\rightarrow \Fun(X',G)
\]
by composing the map
\[
\begin{matrix}
 \Fun(Y,L)&\rightarrow& \Fun(Y,P)\\
 \varphi&\mapsto&(-1)^{\dim P-\dim L}\varphi.
\end{matrix}
\]
with \eqref{equation:inductionconstructiblefunctions} and the induction defined in \S \ref{subsection:indconstructiblefunctions} (restricted to $G$-equivariant functions).

\subsection{Induction of constructible functions on the cotangent bundles}
\label{subsection:inductioncstblecotangent}
Consider the diagram of \S \ref{subsection:inductiondiagram}. We have as in \S \ref{subsection:cotangentcorrespondences} a cotangent correspondence
\[
 \Tan^*X\xleftarrow{\phi}\Tan^*_W(X\times X')\xrightarrow{\psi}\Tan^*X'.
\]
In the exact same way as in \S \ref{subsection:indconstructiblefunctions}, we define the induction of constructible functions
\[
 \Ind\colon \Fun(\Tan^*X)\rightarrow\Fun(\Tan^*X')
\]
by the formula
\[
 \Ind(h)(x')=(-1)^d\sum_{\rmc\in \BoZ}\chi_{\rmc}(\psi^{-1}(x')\cap \phi^{-1}h^{-1}(c))\cdot c.
\]
Here, we would like to emphasize the presence of the $(-1)^d$ factor, which is omitted in \cite{lusztig2000semicanonical} for example. This sign will later be compensated by the $\Psi$-twist (\S\ref{theorem:algebramorphismSSN}), at least when we choose the natural twist $\Psi=(-1)^{\langle-,-\rangle}$. If $\Lambda\subset \Tan^*X$, $\Lambda'\subset \Tan^*X'$ are such that $\psi\phi^{-1}(\Lambda)\subset\Lambda'$, the same operation gives a map
\[
  \Ind\colon \Fun(\Lambda)\rightarrow\Fun(\Lambda').
\]

\subsection{Induction of equivariant constructible functions on the cotangent bundle}
\label{subsection:inductionequivariantcotangent}
Given the canonical isomorphism
\[
 \mu_{G,X}^{-1}(0)\cong \mu_{P,Y}^{-1}(0)\times^PG,
\]
the formula of \S \ref{subsection:indequicstblefunctions} define an induction isomorphism
\[
 \Ind_P^G\colon \Fun(\mu_{P,Y}^{-1}(0),P)\rightarrow \Fun(\mu_{G,X}^{-1}(0),G).
\]
Since $\psi(\phi^{-1}(\mu_{G,X}^{-1}(0)))\subset \mu_{G,X'}^{-1}(0)$, the induction of constructible function of \S \ref{subsection:inductioncstblecotangent} restricted to $G$-equivariant functions induces an induction
\[
 \Ind\colon\Fun(\mu_{G,X}^{-1}(0),G)\rightarrow\Fun(\mu_{G,X'}^{-1}(0),G).
\]
Altogether, we get the induction map
\[
 \Ind\colon\Fun(\mu_{P,Y}^{-1}(0),P)\rightarrow\Fun(\mu_{G,X'}^{-1}(0),G).
\]
When $P$ acts on $Y$ through a map $P\rightarrow L$, then $P$ acts on both on $\Tan^*Y$ and $\mu_{P,Y}^{-1}(0)=\mu_L^{-1}(0)$ through the map $P\rightarrow L$ and and as in \S \ref{subsection:indequicstblefunctions}, we obtain an induction
\begin{equation}
\label{equation:indcstbleequivariantcotangent}
 \Ind\colon \Fun(\mu^{-1}_{Y,L}(0),L)\rightarrow \Fun(\mu_{G,X'}^{-1}(0),G).
\end{equation}

\section{Constructible complexes and Euler obstruction}

\subsection{The local Euler obstruction}
\label{subsection:localEulerobstruction}

Let $X$ be a smooth algebraic variety with an algebraic Whitney stratification $\mathcal{S}$. We let $\Fun(X)$ be the Abelian group of $\BoZ$-valued constructible functions on $X$ and $\Fun(X,\mathcal{S})$ the group of constructible functions on $X$ locally constant along $\mathcal{S}$. We let $\rmZ(X)$ be the group of algebraic cycles on $X$ (its elements are $\BoZ$-linear combinations of closed algebraic subvarieties of $X$) and $\rmZ(X,\mathcal{S})$ be the subgroup of $\BoZ$-linear combinations of the closed algebraic subvarieties $\overline{S}$ for $S\in\mathcal{S}$. The Euler obstruction is a $\BoZ$-linear map                                                                                                                                                                                                                                                                                                                                                                                                                                                                  
\[
\begin{matrix}
 \Eu&\colon&\rmZ(X,\mathcal{S})&\rightarrow&\Fun(X,\mathcal{S})\\
 &&\overline{S}&\mapsto&\Eu(\overline{S}).
\end{matrix}
\]
For $x\in X$, $\Eu(\overline{S})=0$ if $x\not\in\overline{S}$ and $\Eu(\overline{S})=(-1)^{\dim S}$ if $x\in S$. These properties ensure that $\Eu$ is an isomorphism of Abelian groups. For more properties of the Euler obstruction, see \cite[p.14]{massey2011calculations}. The map $\Eu$ is defined in \cite{macpherson1974chern}. The sign $(-1)^{\dim S}$ in our presentation of the Euler obtruction is necessary to make the diagram \eqref{equation:fundamentaldiagram} commute, see \cite[Theorem 3.5]{massey2011calculations}.

\subsection{The fundamental diagram}
Let $X$ be a smooth algebraic variety. Let $\mathcal{S}_X$ be a Whitney stratification of $X$ and $\Lambda=\bigcup_{S\in\mathcal{S}}\Tan^*_SX\subset \Tan^*X$. We have a commutative diagram
\begin{equation}
\label{equation:fundamentaldiagram}
\begin{tikzcd}
	{\GK_0(\CD^{\rmb}_{\rmc,\mathcal{S}}(X))} & {\HO^{\BMo}_{\dim X}(\Lambda)} \\
	{\Fun(X,\mathcal{S})} & {\rmZ(X,\mathcal{S})} \\
	{} & {} & {\Fun(\Lambda)}
	\arrow["\CCy", from=1-1, to=1-2]
	\arrow["\chi"', from=1-1, to=2-1]
	\arrow["\Eu", from=2-2, to=2-1]
	\arrow["I"', from=1-2, to=2-1]
	\arrow["{\Tan^*_{[-]}X}"', from=2-2, to=1-2]
	\arrow["\res", bend left=30, from=3-3, to=2-1]
	\arrow["{?}"', bend right=30, dashed, from=3-3, to=1-2]
\end{tikzcd}
\end{equation}
in which
\begin{enumerate}
 \item $\chi$ is the stalk Euler characteristic: $\chi(\mathscr{F})=\sum_{i\in\BoZ}(-1)^i\dim \HO^i(\mathscr{F}_x)$,
 \item $\CCy$ is the characteristic cycle map defined in \S \ref{subsection:CCmap},
 \item $\Eu$ is the local Euler obstruction briefly recalled in \S \ref{subsection:localEulerobstruction},
 \item $\res$ is the ``naive" restriction of constructible functions coming from the inclusion $X\rightarrow \Lambda$,
 \item $\Tan^*_{[-]}X$ is the conormal bundle map, sending $[\overline{S}]$ to $[\overline{\Tan^*_SX}]$ for $S\in\mathcal{S}$,
 \item $I$ is the intersection with cotangent fibers: for $\lambda$ an irreducible component of $\Lambda$, $I(\lambda)(x)=I(\Tan^*_xX,\lambda)$ for $x\in S$ (see \cite[Equation 0.5.9]{ginsburg1986characteristic}). Note that Ginzburg works with D-modules, but via the Riemann--Hilbert equivalence, this works for constructible sheaves.
\end{enumerate}
In this diagram, all cycles not involving the dashed arrow commute.

\begin{lemma}
 In the fundamental diagram, $\CCy$, $\chi$, $\res$ are surjective; $I$, $\Eu$, $\Tan^*_{[-]}X$ are isomorphisms.
\end{lemma}
\begin{proof}
 The surjectivity of $\chi$ is immediate: $1_S=\chi(\underline{\BoQ}_S)$. By commutativity, $\chi=I\circ \CCy$ so $I$ must be surjective. By the equality of the ranks of the free $\BoZ$-modules $\Fun(X,\mathcal{S})$ and $\HO^{\BMo}_{\dim X}(\Lambda)$ (both have a $\BoZ$-basis indexed by $\mathcal{S}$), $I$ is an isomorphism. Then, $\CCy=I^{-1}\circ \chi$ is also surjective. The surjectivity of $\res$ is also immediate, as $X\subset \Lambda$. The fact that $\Eu$ is an isomorphism comes from its elementary properties (in particular, $\Eu(\overline{S})$ vanishes outside of $\overline{S}$ and takes value $1$ on $S$ and so $\Eu$ is unitriangular with respect to natural bases of the source and the target). It is also straightforward that $\Tan^*_{[-]}X$ is an isomorphism.
\end{proof}

\begin{remark}
 The maps $\CCy$, $\chi$ and $\res$ are in general \emph{not} injective. To find an example for which the injectivity of $\CCy$, $\chi$ and $\res$ fail, it suffices to take $X=\BoC^*$, $\SF$ the trivial rank one local system on $X$, $\SG$ a nontrivial rank one local system on $X$, $\CS=\{X\setminus\{1\},\{1\}\}$ a stratification of $X$, $f$ the function taking value $0$ on $\Tan^*_XX$ and $1$ on $\Tan^*_{\{1\}}X\setminus \Tan^*_XX$, and $g$ the $0$ function. Then, we have $\CCy(\SF)=\CCy(\SG)$, $\chi(\SF)=\chi(\SG)$ and $\res(f)=\res(g)$.
\end{remark}

\section{Compatibility of operations}

\subsection{Induction and characteristic cycle}
We present here one of the main results of this paper. With the preparatory work of \S\ref{section:CCandSS}, its proof is rather straightforward.
\begin{theorem}
\label{theorem:compCCnonequiv}
 For any $\Lambda\subset \Tan^*X,\Lambda'\subset \Tan^*X'$ satisfying the condition $(*)$, the diagram
 \[
  \begin{tikzcd}
	{\GK_0(\CD^{\rmb}_{\rmc}(X,\Lambda))} & {\GK_0(\CD^{\rmb}_{\rmc}(X',\Lambda'))} \\
	{\HO^{\BMo}_{\dim X}(\Lambda)} & {\HO^{\BMo}_{\dim X'}(\Lambda')}
	\arrow["\CCy"', from=1-1, to=2-1]
	\arrow["\Ind", from=2-1, to=2-2]
	\arrow["\Ind", from=1-1, to=1-2]
	\arrow["\CCy", from=1-2, to=2-2]
\end{tikzcd}
 \]
 commutes.
 By taking the limit over $\Lambda,\Lambda'$, the diagram
 \[
  \begin{tikzcd}
	{\GK_0(\CD^{\rmb}_{\rmc}(X))} & {\GK_0(\CD^{\rmb}_{\rmc}(X'))} \\
	{\Lagr^{\mathbb{G}_{\mathrm{m}}}(\Tan^*X)} & {\Lagr^{\mathbb{G}_{\mathrm{m}}}(\Tan^*X')}
	\arrow["\CCy"', from=1-1, to=2-1]
	\arrow["\Ind", from=2-1, to=2-2]
	\arrow["\Ind", from=1-1, to=1-2]
	\arrow["\CCy", from=1-2, to=2-2]
\end{tikzcd}
 \]
 commutes.
\end{theorem}

\begin{proof}
 Let $\mathscr{F}\in \CD^{\rmb}_{\rmc}(X,\Lambda)$. We have
 \[
 \begin{aligned}
  \CCy(\Ind(\mathscr{F}))&\stackrel{\text{\S \ref{subsection:inductioncstblecomplexes}}}{=}\CCy(g_*f^*\mathscr{F}[d])\\
                       &\stackrel{\text{\S \ref{subsection:CCmap} \S\ref{item:shift}}}{=}(-1)^d\CCy(g_*f^*\mathscr{F})\\
                       &\stackrel{\text{\S \ref{subsection:CCmap} \S\ref{item:properpf}}}{=}(-1)^{d}g_*\CCy(f^*\mathscr{F})\\
                       &\stackrel{\text{\S \ref{subsection:CCmap} \S\ref{item:smoothpb}}}{=}(-1)^{2d}g_*f^*\CCy(\mathscr{F})\\
                       &\stackrel{\text{\S \ref{subsection:CCmap} \S\ref{item:smoothpb} and \S\ref{item:properpf}}}{=}(\pr_{X'})_*((\rmd g)^*)^!((\rmd f)^*)_*(\pr_X)^*\CCy(\mathscr{F}).
 \end{aligned}
 \]
However, since the Cartesian square in \eqref{equation:comparison} has no excess intersection bundle, then $((\rmd g)^*)^!((\rmd f)^*)_*\cong \psi'_*(\phi')^!$. By compatibility of refined pull-backs with composition of l.c.i. morphisms \cite[Theorem 6.5]{fulton2013intersection}, we obtain
\[
 \CCy(\Ind(\mathscr{F}))=(-\psi)_*\phi^!\CCy(\mathscr{F})=\Ind(\CCy(\SF)).
\]
\end{proof}

\subsection{Equivariant induction and characteristic cycle}

\begin{lemma}
\label{lemma:equivariantind}
 Let $Y$ be a $P$-variety, where $P\subset G$ and $X=Y\times^PG$. Let $\mathcal{S}_Y$ be a $P$-equivariant Whitney stratification of $Y$ and $\mathcal{S}_X$ the induced Whitney stratification of $X$. Let $\Lambda_Y=\bigsqcup_{S\in\mathcal{S}_Y}\Tan^*_SY$ and $\Lambda_X=\bigsqcup_{T\in\mathcal{S}_X}\Tan^*_TX$. Then, the diagram
 \[
  \begin{tikzcd}
	{\GK_0(\CD^{\rmb}_{\rmc,P,\mathcal{S}_Y}(Y))} & {\GK_0(\CD^{\rmb}_{\rmc,G,\mathcal{S}_X}(X))} \\
	{\HO^{\BMo}_{\dim Y}(\Lambda_Y)} & {\HO^{\BMo}_{\dim X}(\Lambda_X)}
	\arrow["{\Ind_P^G}", from=2-1, to=2-2]
	\arrow["\CCy"', from=1-1, to=2-1]
	\arrow["\CCy", from=1-2, to=2-2]
	\arrow["{\Ind_P^G}", from=1-1, to=1-2]
\end{tikzcd}
 \]
 commutes.
 
 In the case where the $P$-action on $Y$ comes from an $L$-action and a morphism $P\rightarrow L$, we can replace $P$ by $L$ in the previous diagram.
\end{lemma}
\begin{proof}
 Consider the diagram
 \[
  Y\xleftarrow{\pr_2} G\times Y\xrightarrow{\pr} X.
 \]
Let $\mathscr{F}\in \CD^{\rmb}_{\rmc,P,\mathcal{S}_Y}(Y)$ and $\mathscr{G}\in \CD^{\rmb}_{\rmc,G,\mathcal{S}_X}(X)$ be such that $\pr_2^*\mathscr{F}\cong \pr^*\mathscr{G}$. Then, by the functoriality of the characteristic cycle by smooth pull-back \S\ref{item:smoothpb}, if $\CCy(\mathscr{F})=\sum_{S\in\mathcal{S}_Y}a_S\overline{\Tan^*_SY}$, then $(-1)^{\dim G}\CCy(\pr_2^*\mathscr{F})=\sum_{S\in\mathcal{S}_Y}a_S\overline{\Tan^*_{G\times S}(G\times X)}$ and if $\CCy(\mathscr{G})=\sum_{S\in\mathcal{S}_Y}b_S\overline{\Tan^*_{S\times^PG}X}$, then $\CCy(\pr^*\mathscr{G})=(-1)^{\dim P}\sum_{S\in\mathcal{S}_Y}b_S\overline{\Tan^*_{G\times S}(G\times Y)}$. Therefore, for any $S\in\mathcal{S}_Y$, $a_S=(-1)^{\dim G-\dim P}b_S$, that is $\CCy(\mathscr{G})=(-1)^{\dim G-\dim P}\Ind_P^G\CCy(\mathscr{F})$. Taking the shift by $\dim G-\dim P$ in the definition of $\Ind_P^G\mathscr{F}=\mathscr{G}[\dim G-\dim P]$, we get the first statement of the lemma.

For the second statement, this comes from our conventions regarding the characteristic cycle of equivariant complexes (\S \ref{subsection:characteristiccycleequivariantcomplexes}) and the diagram of functors
\[
 \begin{tikzcd}
	{\CD^{\rmb}_{\rmc,L}(Y)} \\
	{} & {\CD^{\rmb}_{\rmc}(Y)} \\
	{\CD^{\rmb}_{\rmc,P}(Y)}
	\arrow["\For", from=1-1, to=2-2]
	\arrow["\For"', from=3-1, to=2-2]
	\arrow[from=1-1, to=3-1]
\end{tikzcd}
\]
where the vertical map is defined by \eqref{equation:LtoP} and the forgetful functors are defined as in \S \ref{subsection:characteristiccycleequivariantcomplexes}.
\end{proof}

We now have the equivariant version of Theorem \ref{theorem:compCCnonequiv}.

\begin{theorem}
 \label{theorem:compCCequiv}
 Consider the equivariant induction diagram of \S \ref{subsection:equivariantinductiondiagram} and the induction maps of \S\S \ref{subsection:equivariantinduction} and \ref{subsection:inductionequivariantBM}.
 The square
 \[
  \begin{tikzcd}
	{\GK_0(\CD^{\rmb}_{\rmc,P}(Y,\Lambda))} & {\GK_0(\CD^{\rmb}_{\rmc,G}(X',\Lambda'))} \\
	{\HO^{\BMo}_{\dim X}(\Lambda)} & {\HO^{\BMo}_{\dim X'}(\Lambda')}
	\arrow["\CCy"', from=1-1, to=2-1]
	\arrow["\Ind", from=2-1, to=2-2]
	\arrow["\Ind", from=1-1, to=1-2]
	\arrow["\CCy", from=1-2, to=2-2]
\end{tikzcd}
 \]
 commutes.
 By taking the limit over $\Lambda,\Lambda'$, the diagram
 \[
  \begin{tikzcd}
	{\GK_0(\CD^{\rmb}_{\rmc,P}(Y))} & {\GK_0(\CD^{\rmb}_{\rmc,G}(X'))} \\
	{\Lagr_P^{\mathbb{G}_{\mathrm{m}}}(\Tan^*Y)} & {\Lagr_G^{\mathbb{G}_{\mathrm{m}}}(\Tan^*X')}
	\arrow["\CCy"', from=1-1, to=2-1]
	\arrow["\Ind", from=2-1, to=2-2]
	\arrow["\Ind", from=1-1, to=1-2]
	\arrow["\CCy", from=1-2, to=2-2]
\end{tikzcd}
 \]
 commutes.
 
 In the case when the $P$-action on $Y$ comes from an $L$-action and a map $P\rightarrow L$, we can replace $P$ by $L$ in the above diagram.
\end{theorem}
\begin{proof}
 Thanks to Lemma \ref{lemma:equivariantind} and Theorem \ref{theorem:compCCnonequiv} the left and right interior squares of the diagram 
 \[
  \begin{tikzcd}
	{\GK_0(\CD^{\rmb}_{\rmc,P}(Y,\Lambda))} & {\GK_0(\CD^{\rmb}_{\rmc,G}(X,\Ind_P^G\Lambda))} & {\GK_0(\CD^{\rmb}_{\rmc,G}(X',\Lambda'))} \\
	{\HO^{\BMo}_{\dim Y}(\Lambda)} & {\HO^{\BMo}_{\dim X}(\Ind_P^G\Lambda)} & {\HO^{\BMo}_G(\Lambda')}
	\arrow["{\Ind_P^G}", from=1-1, to=1-2]
	\arrow["\Ind", from=1-2, to=1-3]
	\arrow["{\Ind_P^G}", from=2-1, to=2-2]
	\arrow["\Ind", from=2-2, to=2-3]
	\arrow["\CCy"', from=1-1, to=2-1]
	\arrow["\CCy"', from=1-2, to=2-2]
	\arrow["\CCy", from=1-3, to=2-3]
\end{tikzcd}
 \]
commute. This proves the commutativity of the two diagrams in the theorem. To replace $P$ by $L$, we use the last statement of Lemma \ref{lemma:equivariantind}.
\end{proof}

\subsection{Other compatibilities}

\begin{lemma}
\label{lemma:compatibilities}
 Consider the induction diagram of \S\ref{subsection:inductiondiagram}. In the diagram \eqref{equation:fundamentaldiagram} for $X$ (and its analogue for $X'$, with $\Lambda$, $\Lambda'$ satisfying the conditions of \S \ref{subsection:inductioncstblecotangent}), we have the following compatibilities:
 \begin{enumerate}
  \item \label{item:compCC} $\CCy\circ \Ind=\Ind\circ \CCy$,
  \item \label{item:compres} $\Ind\circ\res=\res\circ\Ind$,
  \item \label{item:compchi}$\chi\circ \Ind=\Ind\circ\chi$,
  \item \label{item:compI}$I\circ\Ind=\Ind\circ I$.  
 \end{enumerate}
\end{lemma}
\begin{proof}
 The formula \eqref{item:compCC} is nothing else than Theorem \ref{theorem:compCCnonequiv}.
 
 The formula \eqref{item:compres} comes from the diagram
 \[
 \begin{tikzcd}
	X & W & {X'} \\
	{\Tan^*X} & {\Tan^*_W(X\times X')} & {\Tan^*X'}
	\arrow["f"', from=1-2, to=1-1]
	\arrow["g", from=1-2, to=1-3]
	\arrow["\phi"', from=2-2, to=2-1]
	\arrow["\psi", from=2-2, to=2-3]
	\arrow[from=1-3, to=2-3]
	\arrow[from=1-2, to=2-2]
	\arrow[from=1-1, to=2-1]
	\arrow["\ulcorner"{anchor=center, pos=0.125, rotate=0}, draw=none, from=1-2, to=2-3]
\end{tikzcd}
 \]
where vertical arrows are the inclusion of the zero-sections and the right-most square is Cartesian. Indeed, $\psi^{-1}(X')=W$: if $(w,l,l')\in \Tan^*_W(X\times X')$ ($l\in \Tan^*_{f(w)}X$, $l'\in \Tan^*_{g(w)}X'$) is such that $\psi(w,l,l')\in X'$, that is $l'=0$, then $l'\circ (\rmd g(w))=0$ and hence $l\circ(\rmd f(w))=0$. But since $f$ is smooth, $\rmd f(w)$ is surjective so $l=0$. Therefore, $(w,l,l')=(w,0,0)\in W\subset \Tan^*_W(X\times X')$. Let $\varphi\in \Fun(\Tan^*X)$ and $x'\in X'$. Then, $\Ind(\varphi)(x')=(-1)^d\sum_{\rmc\in\BoZ}\chi_{\rmc}(\phi^{-1}\varphi^{-1}(c)\cap \psi^{-1}(x'))\cdot c$. Since $\psi^{-1}(x')=g^{-1}(x')$, $\phi^{-1}\varphi^{-1}(c)\cap \psi^{-1}(x')=f^{-1}(\res(\varphi))^{-1}(c)\cap g^{-1}(x')$, so that $\Ind(\varphi)(x')=\Ind(\res(\varphi))(x')$.

We now prove the formula \eqref{item:compchi}. The relevant facts are:
\begin{enumerate}
 \item The compactly supported Euler characteristic of constructible sheaves satisfies the cut and paste relation as consequence of the long exact sequence in cohomology of the distinguished triangle $j_!j^*\rightarrow\id\rightarrow i_*i^*\rightarrow$ in the constructible derived category, where $j$ is an open immersion and $i$ the embedding of the closed complement,
 \item \label{item:eulercharlocsys}\cite[Proposition 2.5.4 ii)]{dimca2004sheaves} If $X$ is a topological space and $\mathscr{L}$ a local system on $X$, the twisted Euler characteristic $\chi(X,\mathscr{L})=\sum_{i\in \BoZ}(-1)^i\dim \HO^i(\mathscr{L})$ satisfies the equality
 \[
  \chi(X,\mathscr{L})=\rank(\mathscr{L})\chi(X).
 \]
 Consequently, if $\mathscr{F}$ is a constructible sheaf on $X$ with locally constant cohomology sheaves,
 \[
  \chi(X,\mathscr{F})=\chi(\mathscr{F}_x)\chi(X)
 \]
 for any $x\in X$, where $\chi(\SF_{x})$ is the alterned sum of the dimensions of the graded components of $\SF_x$.
\end{enumerate}
Let $\mathscr{F}\in \CD^{\rmb}_{\rmc,\mathcal{S}}(X)$. For $S\in\mathcal{S}_X$, we fix $s\in S$. Then, for $x'\in X'$,
 \[
 \begin{aligned}
  \chi(\Ind(\mathscr{F}))(x')&=(-1)^d\chi((g_*f^*\mathscr{F})_{x'})\\
                             &=(-1)^d\chi((f^*\mathscr{F})_{g^{-1}(x')})\\
                             &=(-1)^d\sum_{S\in\mathcal{S}}\chi_{\rmc}((f^*\mathscr{F}_S)_{g^{-1}(x')})\\
                             &=(-1)^d\sum_{S\in\mathcal{S}}\chi_{\rmc}((f^{-1}(S)\cap{g^{-1}(x')})\chi(\mathscr{F}_s)\\
                             &=(\Ind\chi(\mathscr{F}))(x').
 \end{aligned}
 \]
The subscript ``$\rmc$'' appeared at the third equality since $g^{-1}(x')$ is compact and so the Euler characteristic and compact Euler characteristic coincide. At the fourth equality, $s$ denotes an arbitrary point of the stratum $S$.
 
Last, Formula \eqref{item:compI} is a consequence of the surjectivity of $\CCy$, Formulas \eqref{item:compCC} and \eqref{item:compchi} and the commutativity of the fundamental diagram \eqref{equation:fundamentaldiagram}. Indeed, if $u\in \HO^{\BMo}_{\dim X}(\Lambda)$, there exists $\mathscr{F}\in \CD^{\rmb}_{\rmc}(X)$ such that $u=\CCy(\mathscr{F})$. Then,
\[
\begin{aligned}
 \Ind I(u)&=\Ind \chi(\mathscr{F})\\
           &= \chi(\Ind\mathscr{F})\\
           &= I(\CCy(\Ind\mathscr{F}))\\
           &=I(\Ind \CCy(\mathscr{F}))\\
           &=I(\Ind u).
 \end{aligned}
\]
\end{proof}

\subsection{Other compatibilities in the equivariant setting}
We first state the compatibilities for the induction operations $\Ind_P^G$ defined in \S\S \ref{subsection:equivariantinduction}, \ref{subsection:inductionequivariantBM}, \ref{subsection:indequicstblefunctions} and \ref{subsection:inductionequivariantcotangent}. All these inductions are denoted in the same way, although they act on different spaces. The context determines which induction is considered.
\begin{lemma}
\label{lemma:compatibilitiesPG}
 We have
 \begin{enumerate}
  \item $\CCy\circ \Ind_P^G=\Ind_P^G\circ \CCy$,
  \item $\Ind_P^G\circ\res=\res\circ\Ind_P^G$
  \item $\chi\circ \Ind_P^G=\Ind_P^G\circ \chi$,
  \item $I\circ \Ind_P^G=\Ind_P^G\circ I$.
 \end{enumerate}
 When the $P$-action comes from an $L$ action and a morphism $P\rightarrow L$, we can replace $P$ by $L$ in all these equalities.
\end{lemma}
\begin{proof}
 The proof of this lemma involves the same arguments as the proof of Lemma \ref{lemma:compatibilities}.
\end{proof}

We now assemble Lemmas \ref{lemma:compatibilities} and \ref{lemma:compatibilitiesPG}.
\begin{proposition}
\label{proposition:compatibilitiesinductions}
 Consider the equivariant induction diagram \eqref{equation:preequivindd} and the equivariant induction operations \eqref{equation:indBMequivariant}, \eqref{equation:indcstbleequivariantcotangent}, \eqref{equation:inductionconstructiblefunctions} and \eqref{equation:inductionfunctor}. We have the following compatibilities:
 \begin{enumerate}
  \item \label{item:compCC2} $\CCy\circ \Ind=\Ind\circ \CCy$,
  \item \label{item:compres2} $\Ind\circ\res=\res\circ\Ind$,
  \item \label{item:compchi2}$\chi\circ \Ind=\Ind\circ\chi$,
  \item \label{item:compI2}$I\circ\Ind=\Ind\circ I$.  
 \end{enumerate}
\end{proposition}
\begin{proof}
 The proof is a straightforward combination of Lemmas \ref{lemma:compatibilities} and \ref{lemma:compatibilitiesPG}.
\end{proof}

\section{Restrictions}

\subsection{(Equivariant) restriction of constructible functions}
\label{subsection:equivariantrestcstblefunctions}
We define an operation of restriction
\[
 \Res\colon \Fun(X')\rightarrow \Fun(Y)
\]
by setting $\Res=(-1)^dp_!q^*$, that is
\[
 \Res(\varphi)(y)=(-1)^d\sum_{\rmc\in\BoZ}\chi_{\rmc}(p^{-1}(y)\cap (\varphi\circ q)^{-1}(c))c.
\]
The same formula gives restriction maps
\[
 \Res\colon \Fun(X',G)\rightarrow \Fun(Y,P)
\]
and
\[
\Res\colon \Fun(X',G)\rightarrow \Fun(Y,L)
\]
if $L\subset P$ is a subgroup, by composing with the forgetful map
\[
\begin{matrix}
 \Fun(Y,P)&\rightarrow&\Fun(Y,L)\\
 \varphi&\mapsto&(-1)^{\dim P-\dim L}\varphi.
\end{matrix}
\]

\subsection{(Equivariant) restriction of constructible functions on cotangent bundles}
\label{subsection:equivrestrictionfctscotangent}
We assume that $p$ is a vector bundle and we let $i\colon Y\rightarrow V$ be its zero-section. We obtain a closed immersion $q\circ i\colon Y\rightarrow X'$.

We define an operation
\[
 \Res\colon\Fun(\Tan^*X')\rightarrow \Fun(\Tan^*Y)
\]
by setting $(\Res(\varphi))(y,l) = (-1)^d\varphi(q\circ i(y),l\circ (d(q\circ i)(y)))$.

The same formula gives operations
\[
  \Res\colon\Fun(\Tan^*X',G)\rightarrow \Fun(\Tan^*Y,P),
\]
and, by composing with the forgetful map
\[
 \begin{matrix}
 \Fun(\Tan^*Y,P)&\rightarrow&\Fun(\Tan^*Y,L)\\
 \varphi&\mapsto&(-1)^{\dim P-\dim L}\varphi,
\end{matrix}
\]

\[
  \Res\colon\Fun(\Tan^*X',G)\rightarrow \Fun(\Tan^*Y,L).
\]
The restrictions 
\[
  \Res\colon\Fun(\Lambda')\rightarrow \Fun(\Lambda)
\]
and
\[
  \Res\colon\Fun(\Lambda',G)\rightarrow \Fun(\Lambda,L)
\]
are defined in the exact same way for any $\Lambda\subset Y$, $\Lambda'\subset X'$ such that $q\circ i(\Lambda)\subset \Lambda'$.

\subsection{The equivariant restriction functor}
\label{subsection:equivariantrestrictionfunctor}
Consider the diagram \eqref{equation:stackyind3}. We assume that $P$ is a parabolic subgroup and $L$ a Levi quotient of $P$, so that the kernel of $P\rightarrow L$ is unipotent. For convenience, we formulate the definition of the restriction functors in terms of constructible sheaves on the stacks (as in \cite[\S1.3]{schiffmann2012lectures}). We have the restriction functors
\[
\begin{matrix}
 \Res_1&\colon& \CD^{\rmb}_{\rmc}(X'/G)&\rightarrow&\CD^{\rmb}_{\rmc}(Y/L)\\
     &      & \mathscr{F}  &\mapsto    &\overline{p'}_!\overline{q}^*\mathscr{F}[\dim V/P-\dim Y/L]
\end{matrix}
\]
and
\[
\begin{matrix}
 \Res_2&\colon& \CD^{\rmb}_{\rmc}(X'/G)&\rightarrow&\CD^{\rmb}_{\rmc}(Y/L)\\
     &      & \mathscr{F}  &\mapsto    &\overline{p'}_*\overline{q}^!\mathscr{F}[\dim Y/L-\dim V/P].
\end{matrix}
\]
It is clear that $\Res_2=\BD\circ \Res_1\circ \BD$.

\subsection{Restriction functor and Euler characteristic}
\begin{proposition}
 We have a commutative diagram
\[
\begin{tikzcd}
	{\GK_0(\CD^{\rmb}_{\rmc}(X'/G))} & {\GK_0(\CD^{\rmb}_{\rmc}(Y/L))} \\
	{\Fun(X',G)} & {\Fun(Y,L)}
	\arrow["\Res", from=1-1, to=1-2]
	\arrow["\chi"', from=1-1, to=2-1]
	\arrow["\chi", from=1-2, to=2-2]
	\arrow["\Res"', from=2-1, to=2-2]
\end{tikzcd}
 \]
\end{proposition}
\begin{proof}
 Unraveling the definition of the functor $\Res\colon\CD^{\rmb}_{\rmc}(X'/G)\rightarrow\CD^{\rmb}_{\rmc}(Y/L)$, we see that it is obtained by the composition
 \[
  \CD^{\rmb}_{\rmc,G}(X')\xrightarrow{q^*}\CD^{\rmb}_{\rmc,P}(V)\xrightarrow{p_![d]}\CD^{\rmb}_{\rmc,P}(Y)\xrightarrow{\For[\dim L-\dim P]}\CD^{\rmb}_{\rmc,L}(Y)
 \]
 where $\For$ is the naive forgetful functor restricting the $P$-equivariant structure to $L\subset P$. The commutativity of the square is now clear.
\end{proof}

\section{The generalised Kac--Moody Lie algebra of a quiver}
\label{section:generalisedKMofaquiver}

\subsection{Borcherds--Bozec Lie algebra}
\label{subsection:borcherdsbozecLiealgebra}
Let $Q=(Q_0,Q_1)$ be a quiver. Bozec defined in \cite{bozec2015quivers} a Borcherds Lie algebra associated to this datum generalising the Kac-Moody algebra of a loop-free quiver.

We decompose the set of vertices $Q_0=Q_0^{\real}\sqcup Q_0^{\iso}\sqcup Q_0^{\hyp}$ where the set of \emph{real} vertices $Q_0^{\real}$ is the set of vertices carrying no loops, the set of \emph{isotropic} vertices $Q_0^{\iso}$ is the set of vertices carrying exactly one loop and the set of \emph{hyperbolic} vertices $Q_0^{\hyp}$ the set of vertices carrying at least two loops. The set of isotropic and hyperbolic vertices $Q_0^{\imm}=Q_{0}^{\iso}\sqcup Q_0^{\hyp}$ is the set of \emph{imaginary} vertices.

The set of positive simple roots of the Borcherds-Bozec algebra of $Q$, $\mathfrak{g}_Q$, is
\[
 I_{\infty}=(Q_0^{\real}\times \{1\})\sqcup(Q_0^{\imm}\times \BoZ_{>0}).
\]
There are two projections $p_1\colon I_{\infty}\rightarrow Q_0$ and $p_2\colon I_{\infty}\rightarrow \BoZ_{>0}$ given by $p_1((i',n))=i'$ and $p_2((i',n))=n$. There is a natural projection
\[
 \begin{matrix}
  p&\colon&\BoZ^{(I_{\infty})}&\rightarrow&\BoZ^{Q_0}\\
  &&(f\colon I_{\infty}\rightarrow\BoZ)&\mapsto&\left((p_2)_*f\colon i'\mapsto\sum_{(i',n)\in I_{\infty}}nf(i',n)\right).
 \end{matrix}
\]

The Euler form of $Q$ is the bilinear form
\[
\begin{matrix}
 \langle-,-\rangle&\colon& \BoZ^{Q_0}\times\BoZ^{Q_0}&\rightarrow&\BoZ\\
&&(\dd,\ee)&\mapsto&\sum_{i'\in Q_0}\dd_{i'}\ee_{i'}-\sum_{i'\xrightarrow{\alpha}j'}\dd_{i'}\ee_{j'}.
\end{matrix}
\]

The lattice $\BoZ^{Q_0}$ has a bilinear form given by the symmetrised Euler form, $(-,-)$ ($(\dd,\ee)=\langle\dd,\ee\rangle+\langle\ee,\dd\rangle$). We endow $\BoZ^{(I_{\infty})}$ with the bilinear form $p^*(-,-)$ obtained by pulling-back the symmetrised Euler form. In explicit terms,
\[
 (1_{(i',n)},1_{(j',m)})=mn(1_{i'},1_{j'}).
\]
There is a Borcherds algebra associated to the datum $(\BoZ^{(I_{\infty})},p^*(-,-))$. It is the Lie algebra over $\BoQ$ with generators $h_{i'},e_i,f_i$, with $i'\in Q_0$, $i\in I_{\infty}$, satisfying the following set of relations.
\[
 \begin{aligned}
{[h_{i'},h_{j'}]}&=0 &\text{ for $i',j'\in Q_0$}\\
  [h_{j'},e_{(i',n)}]&=n(1_{j'},1_{i'})e_{(i',n)}&\text{ for $j'\in Q_0$ and $(i',n)\in I_{\infty}$}\\
  [h_{j'},f_{(i',n)}]&=-n(1_{j'},1_{i'})f_{(i',n)}&\text{ for $j'\in Q_0$ and $(i',n)\in I_{\infty}$} \\
  \ad(e_j)^{1-(j,i)}(e_i)=\ad(f_j)^{1-(j,i)}(f_i)&=0&\text{ for $j\in Q_0^{\real}\times \{1\}$, $i\neq j$}\\
  [e_i,e_j]=[f_i,f_j]&=0 &\text{ if $(i,j)=0$}\\
  [e_i,f_j]&=\delta_{i,j}nh_{i'} &\text{ for $i=(i',n)$}
 \end{aligned}
\]
The Lie algebra $\mathfrak{g}_{Q}$ has a triangular decomposition
\[
\mathfrak{g}_Q=\mathfrak{n}_Q^-\oplus\mathfrak{h}\oplus\mathfrak{n}_Q^+ 
\]
where $\mathfrak{n}_Q^-$ (resp. $\mathfrak{n}_Q^+$, resp. $\mathfrak{h}$) is the Lie subalgebra generated by $e_i, i\in I_{\infty}$ (resp. $f_i, i\in I_{\infty}$, resp. $h_{i'}, i'\in Q_0$) and we will only be interested in its positive part $\mathfrak{n}_Q^{+}$. It is generated by $e_i, i\in I_{\infty}$ with Serre relations
\begin{equation}
\label{equation:serrerelspospart}
 \begin{aligned}
  \ad(e_j)^{1-(j,i)}(e_i)&=0&\text{ for $j\in Q_0^{\real}\times \{1\}$, $i\neq j$}\\
  [e_i,e_j]&=0 &\text{ if $(i,j)=0$}.
 \end{aligned}
\end{equation}
\begin{remark}
If $Q$ is a totally negative quiver, that is each vertex carries at least two loops and any two vertices are connected by at least one arrow, then it has no real vertex and $(i,j)<0$ for any $i,j\in I_{\infty}$, so that $\mathfrak{n}_Q^+$ is the free Lie algebra generated by $e_i, i\in I_{\infty}$. This is a crucial remark for \cite{davison2022bps}.
\end{remark}

\subsection{The enveloping algebra}
By considering the associative algebra generated by $h_{i'}, e_i, f_i$ for $i'\in Q_0$ and $i\in I_{\infty}$, satisfying the relations described in \S \ref{subsection:borcherdsbozecLiealgebra}, we obtain the enveloping algebra $\UEA(\mathfrak{g}_Q)$. It has a triangular decomposition $\UEA(\mathfrak{g}_Q)=\UEA(\mathfrak{n}_Q^-)\otimes \UEA(\mathfrak{h})\otimes \UEA(\mathfrak{n}_Q^+)$. As an enveloping algebra, $\UEA(\mathfrak{g}_Q)$ is a Hopf algebra: $\Delta(x)=x\otimes 1+1\otimes x$ for $x\in\mathfrak{g}_Q$.

\subsection{(Noncommutative) Symmetric functions}
\label{subsection:noncommutativesymfunct}
We introduce the very basics of the theory of noncommutative symmetric functions and refer the reader to \cite{gelfand1995noncommutative} for more details. We let $\Sym\coloneqq \BoQ\langle\Lambda_1,\Lambda_2,\hdots\rangle$ be the free associative algebra generated by an infinite sequence of indeterminates $(\Lambda_k)_{k\geq 1}$. 
We let 
\[
\lambda(t)=\sum_{k\geq 1}\Lambda_kt^k.
\]
The complete homogeneous symmetric functions $S_k$ are defined by the formula
\[
 \sigma(t)\coloneqq \sum_{k\geq 1}S_kt^k=\lambda(-t)^{-1}.
\]
The power sum symmetric functions of the first kind $\Psi_k$ are defined by
\[
 \psi(t)=\sum_{k\geq 1}\Psi_kt^{k-1},
\]
\[
 \frac{\mathrm{d}}{\mathrm{d}t}\sigma(t)=\sigma(t)\psi(t).
\]

The algebra $\Sym$ has a comultiplication determined by the formulas $\Delta\Psi_k=\Psi_k\otimes 1+1\otimes \Psi_k$; in other words, $\Psi_k$ is primitive. We have
\begin{equation}
\label{equation:comultS}
 \Delta S_k=\sum_{p+q=k}S_p\otimes S_q.
\end{equation}

Power sum symmetric functions of the first kind may be expressed in terms of complete homogeneous symmetric functions using quasi-determinants. In formulas, we have
\[
 nS_n=
 \begin{vmatrix}
  \Psi_1&\Psi_2&\hdots&\Psi_{n-1}&\boxed{\Psi_n}\\
  -1    &\Psi_1&\hdots&\Psi_{n-2}&\Psi_{n-1}    \\
  0     &-2    &\hdots&\Psi_{n-3}&\Psi_{n-2}    \\
  \vdots&\vdots&\ddots&\vdots    &\vdots        \\
  0     &0     &\hdots&-n+1      &\Psi_1.
 \end{vmatrix}
\]
In particular, we have
\[
 2S_2=\Psi_2+\Psi_1^2
\]
and
\[
 3S_3=\Psi_3+\Psi_1\Psi_2+\frac{1}{2}\Psi_2\Psi_1+\frac{1}{2}\Psi_1^3.
\]

We can derive an explicit formula for $S_n$.
\begin{proposition}
\label{proposition:Sn}
 For any $n\geq 1$, we have
 \[
  S_n=\frac{1}{n}\sum_{1\leq j_1<j_2<\hdots<j_k<n}\Psi_{j_1}\left(\prod_{r=1}^{k-1}\frac{\Psi_{j_{r+1}-j_r}}{j_r}\right)\frac{\Psi_{n-j_k}}{j_k}.
 \]

\end{proposition}
\begin{proof}
 This is a calculation using \cite[Proposition 2.6]{gelfand1995noncommutative}.
\end{proof}

\begin{corollary}
\label{corollary:lincombSr}
 For any $r\geq 1$, $S_r=\Psi_r+\sum_{\substack{s\geq 2\\r_1+\hdots+r_s=r \\
 r_i\geq 1}}a_{r_1,\hdots,r_s}\prod_{j=1}^s\Psi_{r_j}$
 where $a_{r_1,\hdots,r_s}\in\BoQ$.
\end{corollary}
\begin{proof}
 This comes directly from Proposition \ref{proposition:Sn}
\end{proof}

\subsection{New generators}
\label{subsection:newgenerators}
For $i'\in Q_0^{\real}$ and $n\geq 1$, we let $\tilde{e}_{(i',n)}=e_{(i',n)}$. 

For $i'\in Q_0^{\imm}$ and $n\geq 1$, we have an algebra morphism
\[
 \Sym\rightarrow \UEA^{\BoZ}(\mathfrak{n}_Q^+)
\]
sending $\Psi_r$ to $e_{i',r}$.
we let $\tilde{e}_{(i',n)}$ be the image of $S_n$ by this morphism. By Corollary \ref{corollary:lincombSr}, the set $\{\tilde{e}_{(i',n)}\colon (i',n)\in I_{\infty}\}$ is a set of generators of $\UEA(\mathfrak{n}^+)$.

Moreover, the generators $\tilde{e}_{i',n}$ satisfy the same relations as the generators $e_{i',n}$ \eqref{equation:serrerelspospart}. This follows from the following lemma.

\begin{lemma}
\label{lemma:relationsnewgenerators}
 Let $a,b,c$ be elements of a ring and $M,N\in \BoN$ such that
 \[
  \ad(a)^{M+1}(b)=0
 \]
 and
 \[
  \ad(a)^{N+1}(c)=0
 \]
 Then,
  \[
  \ad(a)^{M+N+1}(bc)=0.
 \]
\end{lemma}
\begin{proof}
 This follows by induction on $M,N$ and the formula
 \[
  \ad(a)(bc)=\ad(a)(b)c+b\ad(a)(c)
 \]
since $\ad(a)$ acts by derivations.
\end{proof}

\begin{lemma}
\label{lemma:comultnewgens}
 Let $(i',n)\in I_{\infty}$. Then,
 \[
  \Delta(\tilde{e}_{(i',n)})=\sum_{p+q=n}\tilde{e}_{(i',p)}\otimes\tilde{e}_{(i',q)}.
 \]
\begin{proof}
 This comes from the very definition of the generators $\tilde{e}_{i',n}$ and \eqref{equation:comultS}.
\end{proof}

\end{lemma}

\subsection{Integral forms}
We let $\UEA^{\BoZ}(\mathfrak{n}_Q^+)$ be the $\BoZ$-subalgebra of $\UEA(\mathfrak{n}_Q^+)$ generated by $e_{i}$, $i\in I_{\infty}$ and $e_{i',n}\coloneqq\frac{e_{i',n}}{n!}$ for $i'\in Q_0^{\real}$, $n\geq 1$.

We use the set of generators defined in \S\ref{subsection:newgenerators} to define a new  integral form $\tilde{\UEA}^{\BoZ}(\mathfrak{n}_Q^+)$ of $\UEA(\mathfrak{n}_Q^+)$. If $i'\in Q_0^{\real}$ $n\geq 1$, we let $\tilde{e}_{(i',n)}=\frac{\tilde{e}_{(i',1)}^n}{n!}$. We define the integral form of $\tilde{\UEA}(\mathfrak{n}^+_Q)$ as the $\BoZ$-subalgebra generated by $\tilde{e}_{(i',n)}$ for $i'\in Q_0$ and $n\geq 1$. It is denoted by $\UEA^{\BoZ}(\mathfrak{n}_Q^+)$. For any $\dd\in Q_0$, the graded component $\UEA^{\BoZ}(\mathfrak{n}_Q^+)$ is a free $\BoZ$-module of finite rank. Moreover, by Lemma \ref{lemma:comultnewgens}, $\UEA^{\BoZ}(\mathfrak{n}_Q^+)$ is a $\BoZ$-bialgebra.

\subsection{Psi-twist}
In this section, we explain $\Psi$-twists of algebras and bialgebras. There does not seem to be references for this in the literature.
\label{subsection:Psi-twist}
\subsubsection{Twisted algebras}
Let $A$ be an $M$-graded $R$-algebra, where $R$ is a ring, and $\Psi\colon M\times M\rightarrow R$ a multiplicative bilinear form: $\Psi(m+n,p)=\Psi(m,p)\Psi(n,p)$ and $\Psi(m,n+p)=\Psi(m,n)\Psi(m,p)$ for any $m,n,p\in M$. We denote by $\star$ the product of $A$. Then, we denote by $A^{\Psi}$ the $\Psi$-twisted algebra obtained from $A$, with multiplication $\star_{\Psi}$. It is given by:
\[
 a\star_{\Psi}b\coloneqq \Psi(m,n)a\star b
\]
for any $a\in A_m$ and $b\in A_n$.

\subsubsection{Twisted bialgebra}
\label{subsubsection:twistedbialgebras}
Let $\xi\colon M\times M\rightarrow R$ be a multiplicative bilinear form. If $A$ is a $\xi$-twisted $R$-bialgebra (i.e. the multiplication on $A\otimes A$ is given for homogeneous $a,b,a',b'$ by $(a\otimes b)\star(a'\otimes b')=\xi(\deg(b),\deg(a'))(aa')\otimes (bb')$) with comultiplication $\Delta=\sum_{u,v\in M}\Delta_{u,v}$, and $\Psi$ takes values in $R^{\times}$, then $A^{\Psi}$ is a $\xi'$-twisted bialgebra for the twisted comultiplication $\Delta^{\Psi}=\sum_{u,v\in M}\frac{1}{\Psi(u,v)}\Delta_{u,v}$, where $\xi'(u,v)=\frac{\Psi(v,u)}{\Psi(u,v)}\xi(u,v)$. This follows from the calculation
\[
\begin{aligned}
 \Delta^{\Psi}_{u,v}(a\star_{\Psi}b)&=\frac{\Psi(\deg(a),\deg(b))}{\Psi(u,v)}\Delta_{u,v}(a\star b)\\
 &=\frac{\Psi(\deg(a),\deg(b))}{\Psi(u,v)}\sum_{\substack{u'+u''=u\\v'+v''=v\\u'+v'=\deg(a)\\u''+v''=\deg(b)}}\Delta_{u',v'}(a)\star \Delta_{u'',v''}(b)\\
 &=\sum_{\substack{u'+u''=u\\v'+v''=v\\u'+v'=\deg(a)\\u''+v''=\deg(b)}}\frac{\Psi(u'+v',u''+v'')}{\Psi(u,v)}\Delta_{u',v'}(a)\star \Delta_{u'',v''}(b)\\
 &=\sum_{\substack{u'+u''=u\\v'+v''=v\\u'+v'=\deg(a)\\u''+v''=\deg(b)}}\frac{\Psi(u'+v',u''+v'')\Psi(u',v')\Psi(u'',v'')}{\Psi(u,v)}\Delta^{\Psi}_{u',v'}(a)\star \Delta^{\Psi}_{u'',v''}(b)\\
 &=\sum_{\substack{u'+u''=u\\v'+v''=v\\u'+v'=\deg(a)\\u''+v''=\deg(b)}}\frac{\Psi(u'+v',u''+v'')\Psi(u',v')\Psi(u'',v'')}{\Psi(u,v)\Psi(u',u'')\Psi(v',v'')}\frac{\Psi(u'',v')}{\Psi(v',u'')}\Delta^{\Psi}_{u',v'}(a)\star_{\Psi} \Delta^{\Psi}_{u'',v''}(b)
\end{aligned}
\]
valid for homogenous $a,b$ and the fact that
\[
 \frac{\Psi(u'+v',u''+v'')\Psi(u',v')\Psi(u'',v'')}{\Psi(u,v)\Psi(u',u'')\Psi(v',v'')}\frac{\Psi(u'',v')}{\Psi(v',u'')}=1.
\]
for any $u',v',u'',v''$ such that $u=u'+u''$ and $v=v'+v''$.

When $A$ is a $\xi$-twisted bialgebra, the notation $A^{\Psi}$ denotes the $\xi'$-twisted bialgebra with $\xi'$ defined as above (i.e. with twisted product and twisted coproduct on $A\otimes A$).

\subsubsection{Modification of the $\xi$-twist by an automorphism}
\label{subsubsection:modificationxitwist}
Let $A$ be a $\xi$-twisted bialgebra as in \S\ref{subsubsection:twistedbialgebras}. We denote by $A\otimes^{\xi}A$ the $\xi$-twisted algebra structure on $A\otimes A$. We let $f$ be a ring automorphism of $A$ which preserves  globally $R$ and $f\circ \xi$ the induced bilinear map. We let $A\otimes ^{f\circ\xi}A$ be the $f\circ\xi$-twisted $R$-bialgebra structure on $A\otimes A$. The map $f\otimes f$ gives an algebra isomorphism
\[
 f\otimes f\colon A\otimes^{\xi} A\rightarrow A\otimes^{f\circ \xi}A.
\]
Indeed, we have
\[
 (f\otimes f)((a\otimes b)\star^{\xi}(c\otimes d))=(f\circ \xi)(\deg(a),\deg(b))(f(a)f(c)\otimes f(b)f(d))=(f(a)\otimes f(b))\star^{f\circ\xi}(f(c)\otimes f(d)).
\]

The modified coproduct $(f\otimes f)\circ \Delta\circ f^{-1}$ makes $A$ a $f\circ \xi$-twisted bialgebra.

\subsection{The quantum group}
\label{subsection:thequantumgroup}

\subsubsection{The quantum group and its integral form}
We only define the positive part of the quantum group of generalised Kac--Moody type associated to $Q$. It is the $\BoQ(q)$-algebra $\UEA_q(\mathfrak{n}_Q^+)$ generated by $e_{i}, i\in I_{\infty}$ with the relations
\[
 \begin{aligned}
  \sum_{k+l=1-(j,i)}(-1)^k{{1-(j,i)}\choose{k}}_qe_j^ke_ie_j^l&=0&\text{ for $j\in Q_0^{\real}\times \{1\}$, $i\neq j$}\\
  [e_i,e_j]&=0 &\text{ if $(i,j)=0$}.
 \end{aligned}
\]
The quantum numbers, factorials and binomial coefficients are given by:
\[
 [n]_q=\frac{q^n-q^{-n}}{q-q^{-1}}, \quad [n]_q!=\prod_{j=1}^n[n]_q, \quad \text{and}\quad {{n}\choose{k}}_q=\frac{[n]_q!}{[k]_q![n-k]_q!}.
\]

The \emph{Lusztig integral form} of the quantum group is the $\BoZ[q,q^{-1}]$-subalgebra of $\UEA_q(\mathfrak{n}_Q^+)$ generated by
\[
 e_i \text{ for $i\in I_{\infty}$}, \quad\text{and}\quad \frac{e_{i',1}^n}{[n]_q!} \text{ for $i'\in Q_0^{\real}, n\geq 1$}.
\]
We denote by $\UEA_q^{\BoZ}(\mathfrak{n}_Q^+)$ the integral form of the positive part of the quantum group.

\subsubsection{Coproduct}
The quantum group $\UEA_q(\mathfrak{n}_Q^+)$ and its integral form $\UEA_q^{\BoZ}(\mathfrak{n}_Q^+)$ are $\xi_q$-twisted bialgebras, where $\xi_q(m,n)=q^{(m,n)}$ and $(m,n)=\langle m,n\rangle+\langle n,m\rangle$ is the symmetrised Euler form of $Q$ (see \cite[Chapter 1]{lusztig2010introduction}).

The coproduct $\Delta\colon \UEA_q(\mathfrak{n}_Q^+)\rightarrow \UEA_q(\mathfrak{n}_Q^+)\otimes_{\BoQ(q)}\UEA_q(\mathfrak{n}_Q^+)$ is the unique algebra morphism such that $\Delta(e_i)=e_i\otimes 1+1\otimes e_i$ for $i\in I_{\infty}$.

This coproduct is not the restriction to $\UEA_q(\mathfrak{n}_Q^+)$ of the coproduct on the full quantum group $\UEA_q^{\BoZ}(\mathfrak{n}_Q^+)$ (see \cite[Chapter 1]{lusztig2010introduction}; although Lusztig deals with quantum groups of Kac--Moody types, it adapts to quantum groups of generalised Kac--Moody type readily).

\subsection{Connection between the enveloping algebra and the specialization at $-1$ of the quantum group}
\label{subsection:connection}

In this section, we give an explanation for the appearance of the Psi-twist in Theorems \ref{theorem:realisationdiagram} and \ref{theorem:maintheoremexpanded}.

We let $\UEA^{\BoZ}_{-1}(\mathfrak{n}_Q^+)$ be the specialization at $-1$ of the positive part of the quantum group, that is $\UEA_{-1}^{\BoZ}(\mathfrak{n}_Q^+)=\UEA_{q}^{\BoZ}(\mathfrak{n}_Q^+)\otimes_{\BoZ[q,q^{-1}]}\BoZ$ where $\BoZ$ is seen as a $\BoZ[q,q^{-1}]$-algebra via $q\mapsto -1$. The algebra $\UEA_{-1}(\mathfrak{n}_Q^+)\coloneqq \UEA_{-1}^{\BoZ}(\mathfrak{n}_Q^+)\otimes_{\BoZ}\BoQ$ is the algebra over $\BoQ$ generated by $e_i$, $i\in I_{\infty}$ with the relations
\[
 \begin{aligned}
  \sum_{k+l=1-(j,i)}(-1)^{k}{{1-(j,i)}\choose{k}}_{q=-1}e_j^{k}e_ie_j^{l}&=0&\text{ for $j\in Q_0^{\real}\times \{1\}$, $i\neq j$}\\
  [e_i,e_j]&=0 &\text{ if $(i,j)=0$}
 \end{aligned}
\]

The integral form $\UEA_{-1}^{\BoZ}(\mathfrak{n}_Q^+)$ is the $\BoZ$-subalgebra of $\UEA_{-1}(\mathfrak{n}_Q^+)$ generated by $e_i$, $i\in I_{\infty}$ and $\frac{e_{i'}^n}{n!}$, $i'\in Q_0^{\real}$, $n\geq 1$. By \S\ref{subsection:thequantumgroup}, $\UEA_{-1}^{\BoZ}(\mathfrak{n}_Q^+)$ is a $\xi_{-1}$-twisted $\BoZ$-bialgebra where $\xi_{-1}(m,n)=(-1)^{(m,n)}$ for $m,n\in\BoN^{Q_0}$ and the comultiplication is determined by $\Delta(e_i)=e_i\otimes 1+1\otimes e_i$ for $i\in I_{\infty}$.
\begin{lemma}
\label{lemma:oppositebinomial}
For $k,n\in\BoN$, we have 
\[
{{n}\choose{k}}_{-q}=\left\{
\begin{aligned}
-{{n}\choose{k}}_q &\text{ if $n,k$ are even or $n$ is odd}\\
{{n}\choose{k}}_q &\text{ if $n$ is even and $k$ odd.}
\end{aligned}
\right.
\]
\end{lemma}

\begin{proof}
  We have $[l]_q=\frac{q^{l}-q^{-l}}{q-q^{-1}}=\sum_{m=0}^{l-1}q^{l-1-2m}$. Therefore, $[l]_{-q}=(-1)^{l-1}[l]_q$. Then, $[l]_{-q}!=(-1)^{\lfloor\frac{l-1}{2}\rfloor}[l]_q!$. Finally, ${{n}\choose{k}}_{-q}=(-1)^{\lfloor\frac{n-1}{2}\rfloor+\lfloor\frac{k-1}{2}\rfloor+\lfloor\frac{n-k-1}{2}\rfloor}{{n}\choose{k}}_q$. A direct calculation shows that $\lfloor\frac{n-1}{2}\rfloor+\lfloor\frac{k-1}{2}\rfloor+\lfloor\frac{n-k-1}{2}\rfloor$ is odd if $n,k$ are even or if $n$ is odd, and even if $n$ is even and $k$ odd. The lemma follows.
\end{proof}

Let $\psi\colon\BoN^{Q_0}\times\BoN^{Q_0}\rightarrow \BoZ$ be a bilinear form such that for any $a,b\in\BoN^{Q_0}$, $\psi(a,b)+\psi(b,a)\equiv\langle a,b\rangle+\langle b,a\rangle=(a,b)\pmod{2}$. We let $\Psi(a,b)=(-1)^{\psi(a,b)}$.

\begin{proposition}
\label{proposition:oppositeparameter}
We have an isomorphism
\[
 \UEA_{q}^{\BoZ}(\mathfrak{n}_Q^+)^{\Psi}\cong \UEA_{-q}^{\BoZ}(\mathfrak{n}_Q^+).
\]
In particular, setting $q=1$, we obtain an isomorphism
\[
 \UEA^{\BoZ}(\mathfrak{n}_Q^+)^{\Psi}\cong \UEA_{-1}^{\BoZ}(\mathfrak{n}_Q^+).
\]
$\UEA_{q}^{\BoZ}(\mathfrak{n}_Q^+)^{\Psi}$ and $\UEA_{-q}^{\BoZ}(\mathfrak{n}_Q^+)$ are $\xi_{-1}$-twisted bialgebras. These morphisms are bialgebra morphisms.
\end{proposition}
\begin{proof}
By Lemma \ref{lemma:oppositebinomial} the Serre relation
\begin{equation}
\label{equation:serrerelation}
  \sum_{k+l=1-(j,i)}(-1)^{k}{{1-(j,i)}\choose{k}}_{-q}e_j^{k}e_ie_j^{l}=0
\end{equation}
at $-q$ is equivalent to
\[
  \sum_{k+l=1-(j,i)}(-1)^{k}{{1-(j,i)}\choose{k}}_{q}e_j^{k}e_ie_j^{l}=0
\]
if $1-(j,i)$ is odd and to
\[
   \sum_{k+l=1-(j,i)}{{1-(j,i)}\choose{k}}_{q}e_j^{k}e_ie_j^{l}=0
\]
if $1-(j,i)$ is even.

Now, for $k+l=1-(j,i)$, denoting by $
\star_{\psi}$ the $\Psi$-twisted product and by $\star$ the untwisted product, we have the equality
\[
 e_j^{\star_{\Psi} k}\star_{\Psi} e_i\star_{\Psi} e_j^{\star_{\Psi} l}=(-1)^{\frac{k(k-1)}{2}\psi(j,j)+\frac{l(l-1)}{2}\psi(j,j)+k\psi(j,i)+\psi(kj+i,lj)}e_j^k\star e_i\star e_j^l.
\]
in the algebra $\UEA_q^{\BoZ}(\mathfrak{n}_Q^+)^{\Psi}$. We have,
\[
 S\coloneqq\frac{k(k-1)}{2}\psi(j,j)+\frac{l(l-1)}{2}\psi(j,j)+k\psi(j,i)+\psi(kj+i,lj)\cong \frac{(k+l)^2-(k+l)}{2}\psi(j,j)+k\psi(j,i)+l\psi(i,j)\pmod{2}
\]

Note that by the condition on $\psi$, $\psi(j,i)+\psi(i,j)\cong (j,i)\pmod{2}$.

The first term of the right-hand-side does not depend on $k,l$ but only on their sum $k+l=1-(j,i)$, and for $1-(j,i)$ even,
\[
 k\psi(j,i)+l\psi(i,j)\cong k(\psi(j,i)+\psi(i,j))\pmod{2}\cong k\pmod{2}
\]
while if $1-(j,i)$ is odd,
\[
 k\psi(j,i)+l\psi(i,j)\equiv k(\psi(j,i)+\psi(i,j))+\psi(i,j)\pmod{2}\equiv \psi(i,j)\pmod{2}
\]
and this quantity does not depend on $k,l$. Therefore, $e_i$ and $e_j$ satisfy the Serre relation \ref{equation:serrerelation} in $\UEA_q(\mathfrak{n}^+_Q)$ if and only if they satisfy the relation
\[
  \sum_{k+l=1-(j,i)}(-1)^k{{1-(j,i)}\choose{k}}_{q=-1}e_j^{\star_{\Psi} k}\star_{\Psi} e_i\star_{\Psi} e_j^{\star_{\Psi} l}=0
\]
if $1-(j,i)$ is odd and
\[
  \sum_{k+l=1-(j,i)}{{1-(j,i)}\choose{k}}_{q=-1}e_j^{\star_{\Psi} k}\star_{\Psi} e_i\star_{\Psi} e_j^{\star_{\Psi} l}=0
\]
if $1-(j,i)$ is even. This proves the proposition.
\end{proof}

\section{Four realisations of $\UEA^{\BoZ}(\mathfrak{n}_Q^+)$}
We let $Q=(Q_0,Q_1)$ be a quiver.

\subsection{Stacks associated to quiver}
\label{subsection:stacksassociatedtoquivers}
For $\dd\in\BoN^{Q_0}$, we let
\[
 X_{Q,\dd}=\bigoplus_{i\xrightarrow{\alpha}j}\Hom(\BoC^{\dd_i},\BoC^{\dd_j})
\]
be the representation space of $\dd$-dimensional representations of the quiver $Q$.

It is acted on by the product of general linear groups $\prod_{i\in Q_0}\GL_{\dd_i}$ by conjugation and the quotient $\mathfrak{M}_{Q,\dd}=X_{Q,\dd}/\GL_{\dd}$ is the stack of $\dd$-dimensional representations of $Q$.

For $\dd,\ee\in\BoN^{Q_0}$, we fix an isomorphism $\BoC^{\dd+\ee}\cong \BoC^{\dd}\oplus\BoC^{\ee}$ and we let $F_{Q,\dd,\ee}$ be the closed subvariety of $X_{\dd+\ee}$ of $(x_{\alpha})_{\alpha\in Q_1}$ preserving the subspace $\BoC^{\ee}$. It is naturally acted on by the parabolic subgroup $P_{\dd,\ee}\subset \GL_{\dd+\ee}$ of invertible endomorphisms of $\BoC^{\dd+\ee}$ preserving $\BoC^{\ee}$.

The cotangent space $\Tan^*X_{Q,\dd}$ is identified by the trace bilinear form to $X_{\overline{Q},\dd}$. The natural action of $\GL_{\dd}$ on $X_{\overline{Q},\dd}$ is Hamiltonian and the moment map is
\[
  \begin{matrix}
   \mu_{\dd}&\colon& \Tan^*X_{\dd}&\rightarrow& \mathfrak{gl}_{\dd}\\
   &&(x,x^*)&\mapsto&\sum_{\alpha\in Q_1}[x_{\alpha},x_{\alpha^*}],
 \end{matrix}
\]
where for $\alpha\in Q_1$ we denoted by $\alpha^*$ the opposite arrow in $Q_1^{\op}$. The variety $\mu_{\dd}^{-1}(0)$ is the representation space of the preprojective algebra $\Pi_Q$ and the quotient $\mathfrak{M}_{\Pi_Q,\dd}=\mu_{\dd}^{-1}(0)/\GL_{\dd}$ is the stack of $\dd$-dimensional representations of $\Pi_Q$.

We define the $\dd$-dimensional strongly seminilpotent variety (resp. stack) as the closed subvariety $\Lambda_{\dd}^{\SSN}$ of $\mu_{\dd}^{-1}(0)$ (resp. closed substack $\mathfrak{N}_{\Pi_Q,\dd}^{\SSN}$ of $\mathfrak{M}_{\Pi_Q,\dd}$) whose $\BoC$-points correspond to representations $M$ of $\Pi_Q$ admitting a filtration $0=M_0\subset M_1\subset\hdots\subset M_r=M$ by subrepresentations such that for any $1\leq i\leq r$, $\dim M_i/M_{i-1}=1$ and for any $\alpha\in Q_1$, $x_{\alpha}M_i\subset M_{i-1}$, $x_{\alpha^*}M_i\subset M_{i}$.

We have $\mathfrak{N}_{\Pi_Q,\dd}^{\SSN}=\Lambda_{\dd}^{\SSN}/\GL_{\dd}$ and $\Lambda_{\dd}^{\SSN}$ is a Lagrangian subvariety of $\Tan^*X_{Q,\dd}$ \cite{bozec2020number}.

\subsection{Constructible functions on the stack of representations of $Q$}
This has been studied by Schofield in his unpublished note \cite{schofield}, for quivers without loops. The construction is recalled in \cite[\S 10]{lusztig1991quivers}.
We let 
\[
\Fun(\mathfrak{M}_Q)=\bigoplus_{\dd\in\BoN^{Q_0}}\Fun(\mathfrak{M}_{Q,\dd})=\bigoplus_{\dd\in\BoN^{Q_0}}\Fun(X_{Q,\dd},\GL_{\dd})
\]
be the group of $\BoZ$-valued constructible functions on $\mathfrak{M}_Q$.

\subsubsection{The product}
\label{subsubsection:theproductcstbleQ}
For $\dd,\ee\in\BoN^{Q_0}$, we have the diagram
\begin{equation}
\label{equation:inddiagramquiver}
 X_{Q,\dd}\times X_{Q,\ee}\leftarrow F_{\dd,\ee}\rightarrow X_{Q,\dd+\ee}.
\end{equation}
This diagram fits in the formalism described in \S \ref{subsection:equivariantinductiondiagram} with $Y=X_{Q,\dd}\times X_{Q,\ee}$, $V=F_{\dd,\ee}$ and $X'=X_{Q,\dd+\ee}$, $P=P_{\dd,\ee}$, $G=\GL_{\dd+\ee}$ and $P_{\dd,\ee}$ acts on $X_{Q,\dd}\times X_{Q,\ee}$ through its Levi quotient $P_{\dd,\ee}\rightarrow \GL_{\dd}\times\GL_{\ee}$. We have a multiplication on $\Fun(\mathfrak{M}_Q)$ using the definition of \S \ref{subsection:indequicstblefunctions}. If $\varphi_1,\varphi_2\in\Fun(\mathfrak{M}_Q)$, we let $\varphi_1\star \varphi_2$ be their product.

\subsubsection{The spherical subalgebra}
\label{subsubsection:sphericalsubalgebraQ}
We let $\Fun^{\sph}(\mathfrak{M}_Q)$ be the $\BoZ$-subalgebra of $\Fun(\mathfrak{M}_Q)$ generated by $1_{ne_{i'}}$, the characteristic function of $X_{Q,ne_{i'}}$, for $i'\in Q_0$ and $n\geq 1$. We call it \emph{the spherical subalgebra}.

\subsection{Perverse sheaves on the stack of representations of $Q$}
We let $\CD^{\rmb}_{\rmc,\GL_{\dd}}(X_{Q,\dd})$ be the category of $\GL_{\dd}$-equivariant constructible complexes on $X_{\dd}$. We let $\GK_0(\CD^{\rmb}_{\rmc}(\mathfrak{M}_Q)):=\bigoplus_{\dd\in\BoN^{Q_0}}\GK_0(\CD^{\rmb}_{\rmc,\GL_{\dd}}(X_{\dd}))$. We will also consider the \emph{split Grothendieck group} of $\CD^{\rmb}_{\rmc,\GL_{\dd}}(X_{\dd})$ which we denote by $\GK_{\oplus}(\CD^{\rmb}_{\rmc,\GL_{\dd}}(X_{\dd}))$. As a $\BoZ$-module, it is generated by isomorphism classes of $\GL_{\dd}$-equivariant bounded contructible complexes on $X_{\dd}$ modulo the relations $[C]=[C']+[C'']$ if $C\cong C'\oplus C''$. We let $\GK_{\oplus}(\CD^{\rmb}_{\rmc}(\mathfrak{M}_Q)):=\bigoplus_{\dd\in\BoN^{Q_0}}\GK_{\oplus}(\CD^{\rmb}_{\rmc,\GL_{\dd}}(X_{\dd}))$. We have a structure of $\BoZ[q,q^{-1}]$-module on $\GK_{\oplus}(\CD^{\rmb}_{\rmc,\GL_{\dd}}(X_{\dd}))$ given by $q^n\cdot [C]\coloneqq [C[n]]$ for any $n\in\BoZ$.

\subsubsection{The induction}
\label{subsubsection:theinductionppsheaves}
The formalism of \S \ref{subsection:equivariantinduction} applied to the diagram \eqref{equation:inddiagramquiver} gives an induction functor
\[
 \CD^{\rmb}_{\rmc,\GL_{\dd}\times\GL_{\ee}}(X_{Q,\dd}\times X_{Q,\ee})\rightarrow \CD^{\rmb}_{\rmc,\GL_{\dd+\ee}}(X_{Q,\dd+\ee}).
\]
Using the functor
\[
 \begin{matrix}
  \CD^{\rmb}_{\rmc,\GL_{\dd}}(X_{Q,\dd})\times \CD^{\rmb}_{\rmc,\GL_{\ee}}(X_{Q,\ee})&\rightarrow& \CD^{\rmb}_{\GL_{\dd}\times \GL_{\ee}}(X_{Q,\dd}\times X_{Q,\ee})\\
  (\mathscr{F},\mathscr{G})&\mapsto&\mathscr{F}\boxtimes^L\mathscr{G},
 \end{matrix}
\]
we obtain a $\BoZ$-bilinear multiplication on $\GK_0(\CD^{\rmb}_{\rmc}(\mathfrak{M}_Q))$ and a $\BoZ[q,q^{-1}]$-bilinear multiplication on $\GK_{\oplus}(\CD^b_c(\FM_Q))$.

\subsubsection{The Lusztig subcategory}
Following Lusztig \cite{lusztig1991quivers}, we define the full subcategory $\mathcal{Q}\subset \CD^{\rmb}_{\rmc}(\mathfrak{M}_Q)$ of semisimple complexes as the smallest full subcategory of $\CD^{\rmb}_{\rmc}(\mathfrak{M}_Q)$ stable under the induction functor, direct summands and shifts and containing $\mathcal{IC}(\mathfrak{M}_{Q,ne_{i'}})=\underline{\BoQ}_{\mathfrak{M}_{Q,ne_{i'}}}[-\langle ne_{i'},ne_{i'}\rangle]$ for $i'\in Q_0$, $n\geq 1$. We let $\mathcal{P}$ be the category of perverse sheaves contained in $\mathcal{Q}$. We have $\GK_0(\mathcal{Q})=\GK_0(\mathcal{P})$. This is a subalgebra of $\GK_0(\CD^{\rmb}_{\rmc}(\mathfrak{M}))$. We have $\GK_{\oplus}(\CQ)=\GK_0(\CQ)\otimes_{\BoZ}\BoZ[q,q^{-1}]$. This is a subalgebra of $\GK_{\oplus}(\CD^{\rmb}_{\rmc}(\FM_Q))$.

\begin{lemma}
\label{lemma:compsplitGG}
We have
\[
\GK_0(\CQ)=\GK_{\oplus}(\CQ)\otimes_{\BoZ[q,q^{-1}]}\BoZ
\]
where $\BoZ$ is seen as a $\BoZ[q,q^{-1}]$-algebra with $q\mapsto -1$.
\end{lemma}
\begin{proof}
 This comes from the definition of $\CQ$, $\CP$, the fact that $\CP$ is a semisimple category and that in the Grothendieck group, $[\CF[n]]=(-1)^n[\CF]$.
\end{proof}

\begin{proposition}[{\cite[Proposition 1.13]{bozec2015quivers}}]
\label{proposition:generatorsK0}
The $\BoZ[q,q^{-1}]$-algebra $\GK_{\oplus}(\mathcal{Q})$ is generated by $[\mathcal{IC}(\mathfrak{M}_{Q,ne_{i'}})]$, $(i',n)\in I_{\infty}$.
\end{proposition}

\begin{corollary}
\label{corollary:generationK0}
 The $\BoZ$-algebra $\GK_0(\CQ)$ is generated by $[\mathcal{IC}(\mathfrak{M}_{Q,ne_{i'}})]$, $(i',n)\in I_{\infty}$.
\end{corollary}

\subsection{Constructible functions on the seminilpotent stack}
\label{subsection:constructiblefunctionsSSNstack}
We let $\Fun(\mathfrak{N}_{\Pi_Q}^{\SSN})=\bigoplus_{\dd\in\BoN^{Q_0}}\Fun(\Lambda^{\SSN}_{\dd},\GL_{\dd})$ be the set of $\BoZ$-valued constructible functions on the strictly seminilpotent stack.

\subsubsection{The product}
\label{subsubsection:productcstblenilstack}
Let $\dd,\ee\in\BoN^{Q_0}$. We let $\Lambda_{\dd,\ee}=\Lambda_{\dd+\ee}^{\SSN}\cap F_{\overline{Q},\dd,\ee}$. We have a diagram
\begin{equation}
\label{equation:indseminilvar}
 \Lambda_{\dd}\times\Lambda_{\ee}\leftarrow \Lambda_{\dd,\ee}\rightarrow \Lambda_{\dd+\ee}.
\end{equation}

We let $Y=X_{\dd}\times X_{\ee}$, $X'=X_{\dd+\ee}$, $V=X_{\dd,\ee}$, $X=\GL_{\dd+\ee}\times^{P_{\dd,\ee}}Y$, $W=\GL_{\dd+\ee}\times^{P_{\dd,\ee}}F_{\dd,\ee}$.
\begin{lemma}
 We have $\Tan^*_{W}(X\times X')\cong \GL_{\dd+\ee}\times^{P_{\dd,\ee}}F_{\overline{Q},\dd,\ee}$.
\end{lemma}
\begin{proof}
 This appears for the $S_g$ quivers (quiver with one vertex and $g$ loops) in \cite{schiffmann2012hall} and it can be generalised to arbitrary quivers in a straightforward way.
\end{proof}

Therefore, with the formalism of \S \ref{subsection:inductioncstblecotangent}, we obtain a product on $\Fun(\mathfrak{N}_{\Pi_Q}^{\SSN})$. For various compatibilities, it is worth noticing that the diagram \eqref{equation:indseminilvar} embeds in a commutative square where the right square is Cartesian:
\[
 \begin{tikzcd}
	{\Lambda_{\dd}\times\Lambda_{\ee}} & {\Lambda_{\dd,\ee}} & {\Lambda_{\dd+\ee}} \\
	{X_{\overline{Q},\dd}\times X_{\overline{Q},\ee}} & {F_{\overline{Q},\dd,\ee}} & {X_{\overline{Q},\dd+\ee}}
	\arrow[from=1-1, to=2-1]
	\arrow[from=1-2, to=1-1]
	\arrow[from=1-2, to=2-2]
	\arrow[from=2-2, to=2-1]
	\arrow[from=1-3, to=2-3]
	\arrow[from=2-2, to=2-3]
	\arrow[from=1-2, to=1-3]
	\arrow["\lrcorner"{anchor=center, pos=0.125}, draw=none, from=1-2, to=2-3]
\end{tikzcd}
\]

\subsubsection{The spherical subalgebra}
\label{subsubsection:spherical}
We let $\Fun^{\sph}(\mathfrak{N}_{\Pi_Q}^{\SSN})=\bigoplus_{\dd\in\BoN^{Q_0}}\Fun^{\sph}(\Lambda_{\dd}^{\sph},\GL_{\dd})$ be the subalgebra generated by $1_{\Tan^*_{X_{Q,ne_{i'}}}X_{Q,ne_{i'}}}$ for $i'\in Q_0$ and $n\geq 0$, where $1_{\Tan^*_{X_{Q,ne_{i'}}}X_{Q,ne_{i'}}}$ is the characteristic function of $X_{Q,ne_i}\subset \Lambda_{ne_i}^{\SSN}$.

\begin{lemma}[{\cite[Proposition 1.16]{bozec2016quivers}}]
\label{lemma:irrcompseminilpotent}
 For any $Z\in \Irr(\Lambda_{\dd}^{\SSN})$, there exists a unique $f_Z\in \Fun^{\sph}(\Lambda_{\dd}^{\SSN},\GL_{\dd})$ such that $f_Z$ takes the value $1$ on a general point of $Z$ and the value $0$ on a general point of any other irreducible component of $\Lambda_{\dd}^{\SSN}$. Moreover, $\{f_Z\mid Z\in \Irr(\Lambda_{\dd}^{\SSN})\}$ forms a basis of $\Fun^{\sph}(\mathfrak{N}_{\Pi_Q,\dd})$.
\end{lemma}

\subsection{The degree $\mathbf{0}$ cohomological Hall algebra of the seminilpotent stack}
\label{subsubsection:degree0CoHA}
We let $\HO^{\BMo}_{\toph}(\mathfrak{N}_{\Pi_Q}^{\SSN},\BoZ):=\bigoplus_{\dd\in\BoN^{Q_0}}\HO^{\BMo}_{\toph}(\mathfrak{N}_{\Pi_Q,\dd}^{\SSN},\BoZ)=\bigoplus_{\dd\in\BoN^{Q_0}}\HO^{\BMo}_{\toph}(\Lambda_{\dd}^{\SSN},\BoZ)$.

\subsubsection{The product}
\label{subsubsection:cohaproduct}
For $\dd,\ee\in\BoN^{Q_0}$, the formalism of \S \ref{subsection:inductioncstblecotangent} associated to the diagram
\[
X_{Q,\dd}\times X_{Q,\ee}\leftarrow F_{\dd,\ee}\rightarrow X_{Q,\dd+\ee}.
\]
together with the closed immersions $\Lambda_{\dd}^{\SSN}\subset \Tan^*X_{Q,\dd}$, $\Lambda_{\ee}^{\SSN}\subset \Tan^*X_{Q,\ee}$ and the inclusion $\psi\phi^{-1}((\Lambda_{\dd}^{\SSN}\times \Lambda_{\ee}^{\SSN})\times^{P_{\dd,\ee}}\GL_{\dd+\ee})\subset \Lambda_{\dd+\ee}^{\SSN}$ gives a map
\[
 \HO^{\BMo}_{\toph}(\Lambda_{\dd}^{\SSN},\BoZ)\otimes \HO^{\BMo}_{\toph}(\Lambda_{\ee}^{\SSN},\BoZ)\rightarrow \HO^{\BMo}_{\toph}(\Lambda_{\dd+\ee}^{\SSN},\BoZ).
\]
The maps obtained for various $\dd,\ee$ combine together to give a product on $\HO^{\BMo}_{\toph}(\mathfrak{N}_{\Pi_Q}^{\SSN},\BoZ)$.

\subsection{The CoHA realisation of $\UEA^{\BoZ}(\mathfrak{n}_Q^+)$}

\begin{proposition}
\label{proposition:qgps}
  There exists a unique morphism of algebras
 \[
 \begin{matrix}
  \UEA_q^{\BoZ}(\mathfrak{n}_Q^+)&\rightarrow& \GK_{\oplus}(\CQ)\\
  e_{(i',n)}&\mapsto&[\underline{\BoQ}_{X_{Q,ne_{i'}}}[-\langle ne_{i'},ne_{i'}\rangle]]=[\mathcal{IC}(\mathfrak{M}_{Q,ne_{i'}})].
 \end{matrix}
 \]
It is bijective.
\end{proposition}
\begin{proof}
We only need to show that the quantum Serre relations hold for $F_{i',1}\coloneqq[\mathcal{IC}(X_{Q,e_{i'}})]$ and $F_{j',n}\coloneqq[\mathcal{IC}(X_{Q,ne_{j'}})]$ in $\GK_{\oplus}(\CD^{\rmb}_{\rmc}(\mathfrak{M}_Q))$ for $i'\in Q_0^{\real}$, $j\in Q_0$ and $n\geq 1$. It comes from Lemma \ref{lemma:serrerel} below.
\end{proof}
\begin{remark}
 Notice that this morphism is \emph{not} a bialgebras morphism if $Q$ has at least one loop since for $n\geq 2$ and $i'$ an imaginary vertex, $e_{(i',n)}$ is primitive while $\IC(X_{Q,ne_{i'}})$ is not. If $i'$ has exactly one loop, it is possible to correct the generators $e_{i',n}$ to obtain a bialgebra morphism. Namely, we could take $\tilde{e}_{i',n}=q^{\frac{n(n-1)}{2}}S_n(e_{i',\bullet})$. However, if a vertex has at least two loops, we do not know how to modify the generators. This would require a deeper understanding of $q$-deformations of the ring of noncommutative symmetric functions.
\end{remark}

First, we have a generalisation of \cite[\S9]{lusztig1991quivers} to arbitrary quivers (possibly carrying loops).
\begin{lemma}
\label{lemma:serrerel}
 Let $i'\in Q_0^{\real}$ and $i'\neq j'\in Q_0$. Let $t$ be the number of arrows (in either directions) between $i'$ and $j'$. Then for any $n\geq 1$, we have
 \[
  \sum_{p=0}^{nt+1}(-1)^pF_{i',1}^{(p)}\star F_{j',n}\star F_{i',1}^{(nt+1-p)}=0
 \]
 where for $k\in \BoN$, $F_{i',1}^{(k)}\coloneqq[\mathcal{IC}(X_{Q,ke_{i'}})]=\frac{[\mathcal{IC}(X_{Q,e_{i'}})]^k}{[k]_q!}$.
\end{lemma}
\begin{proof}
 The proof goes along the lines of \cite[\S 9]{lusztig1991quivers} in a straighforward way.
\end{proof}

\begin{theorem}[{\cite[Theorem 3.34]{bozec2016quivers}}]
\label{theorem:algebramorphismSSN}
  There exists a unique algebra morphism
 \[
  \begin{matrix}
   \delta&\colon&\UEA^{\BoZ}(\mathfrak{n}_Q^+)&\rightarrow& \Fun^{\sph}(\mathfrak{N}_{\Pi_Q}^{\SSN})^{\Psi}\\
   &&\tilde{e}_{(i',n)}&\mapsto&(-1)^{\langle ne_{i'},ne_{i'}\rangle}1_{\Tan^*_{X_{Q,ne_{i'}}}X_{Q,ne_{i'}}}.
  \end{matrix}
 \]
 It is bijective.
\end{theorem}
\begin{proof}
 The bijectivity is proven in {\cite[Theorem 3.34]{bozec2016quivers}} in the case when $\Psi=\langle-,-\rangle$. For general $\Psi$, we twist this isomorphism by $(-1)^{\langle-,-\rangle+\Psi}$. We would like to emphasize that the $\Psi$-twist here appears because of the $(-1)^d$ factor in the definition of the induction of constructible functions on cotangent bundles in \S\ref{subsection:inductioncstblecotangent}.

\end{proof}

\begin{remark}
 \begin{enumerate}
  \item In \cite{bozec2015quivers}, Bozec works with $\BoQ$-valued constructible functions on $\mathfrak{N}_{\Pi_Q}^{\SSN}$. It is easily seen that his proofs actually work with $\BoZ$-coefficients.
  \item Bozec proves the analogous results for $e_{(i',n)}\mapsto1_{ne_{i'}}$, that is the Serre relation
  \[
   \sum_{p=0}^{N+1}(-1)^p1_{pe_{i'}}\star 1_{e_{j'}}\star 1_{(N+1-p)e_{i'}}
  \]
  holds in $\Fun^{\sph}(\mathfrak{N}_{\Pi_Q}^{\SSN})$. We have a very slightly different claim since we used the generators $\tilde{e}_{(i',n)}$ and introduced the sign $(-1)^{\langle ne_{i'},ne_{i'}\rangle}$. However, the presentation of $\UEA^{\BoZ}(\mathfrak{n}^+)$ by the generators $\tilde{e}_{(i',n)}$ is the same as the one by the generators $e_{(i',n)}$ by Lemma \ref{lemma:relationsnewgenerators} (we see the difference between the generators only on the comultiplication) and since the parity of $p^2+(N+1-p)^2=(N+1)^2+2p^2-2p(N+1)$ does not depend on $p$, the generators $(-1)^{\langle ne_{i'},ne_{i'}\rangle}1_{ne_{i'}}$ of $\Fun^{\sph}(\mathfrak{N}_{\Pi_Q}^{\SSN})$ also satisfy Serre relations.

 \end{enumerate}
\end{remark}

\begin{theorem}
\label{theorem:maintheoremexpanded}
 \begin{enumerate}
  \item There exists a unique algebra morphism
\[
 \begin{matrix}
 \alpha&\colon& \UEA^{\BoZ}(\mathfrak{n}_Q^+)&\rightarrow& \GK_0(\mathcal{P})^{\Psi}\\
  &&e_{(i',n)}&\mapsto&[\underline{\BoQ}_{X_{Q,ne_{i'}}}[-\langle ne_{i'},ne_{i'}\rangle]]=[\mathcal{IC}(\mathfrak{M}_{Q,ne_{i'}})].
 \end{matrix}
 \]
It is bijective.

  \item  There exists a unique algebra morphism
 \[
  \begin{matrix}
  \gamma&\colon& \UEA^{\BoZ}(\mathfrak{n}^+)&\rightarrow& \Fun^{\sph}(\mathfrak{M}_Q)^{\Psi}\\
   &&e_{(i',n)}&\mapsto&(-1)^{\langle ne_{i'},ne_{i'}\rangle}1_{X_{ne_{i'}}}.
  \end{matrix}
 \]
 It is bijective.
 
 \item   There exists a unique algebra morphism
 \[
  \begin{matrix}
   \beta&\colon&\UEA^{\BoZ}(\mathfrak{n}_Q^+)&\rightarrow& \HO^{\BMo}_{\toph}(\mathfrak{N}_{\Pi_Q}^{\SSN})^{\Psi}\\
   &&e_{(i',n)}&\mapsto&[\Tan^*_{X_{Q,ne_{i'}}}X_{Q,ne_{i'}}].
  \end{matrix}
 \]
 It is bijective
 \end{enumerate}
 These maps fit in the following commutative diagram of algebra morphisms.
 \begin{equation}
\label{equation:diagramrealisations}
 \begin{tikzcd}
	{\UEA^{\BoZ}(\mathfrak{n}_Q^+)} \\
	& {\GK_0(\mathcal{Q})^{\Psi}} & {\HO^{\BMo}_{\toph}(\mathfrak{N}_{\Pi_Q}^{\SSN},\BoZ)^{\Psi}} \\
	& {\Fun^{\sph}(\mathfrak{M}_Q)^{\Psi}}\\
	&\Fun^{\sph}(\mathfrak{N}_{\Pi_Q})^{\Psi}
	\arrow["\chi"', from=2-2, to=3-2]
	\arrow["{\Tan^*_{[-]}\mathfrak{M}\circ \Eu^{-1}}"'{pos=0.5}, from=3-2, to=2-3]
	\arrow["\CCy", from=2-2, to=2-3]
	\arrow["\gamma"',bend right=30, dashed, from=1-1, to=3-2]
	\arrow["\alpha",dashed, from=1-1, to=2-2]
	\arrow["\beta",bend left=30, dashed, from=1-1, to=2-3]
	\arrow["\res",from=4-2, to=3-2]
	\arrow["\delta"', dashed, bend right=40,from=1-1, to=4-2]
	\arrow["?"',dotted,bend right=60,from=4-2, to=2-3]
\end{tikzcd}
\end{equation}
\end{theorem}
\begin{proof}
 We let $\mathcal{S}_{\dd}$ be a Whitney stratification of $X_{Q,\dd}$ such that all perverse sheaves in $\mathcal{P}_{\dd}$ are $\mathcal{S}_{\dd}$-constructible. Such a stratification exist since for fixed $\dd$, $\mathcal{P}_{\dd}$ is a semisimple category with finitely many simple objects. Since the singular support of objects of $\mathcal{P}_{\dd}$ is contained in $\Lambda_{\dd}^{\SSN}$ (by the argument of \cite[\S 13]{lusztig1991quivers} applied to quivers with loops), $\Fun^{\sph}(X_{\dd},\GL_{\dd})\subset\chi(\GK_0(\mathcal{P}_{\dd}))$ (by the definitions of $\mathcal{P}$, of spherical constructible functions and the compatibility Proposition \ref{proposition:compatibilitiesinductions}, (3), which implies that $\chi$ is an algebra morphism), and by commutativity of the square of the diagram \ref{equation:fundamentaldiagram} applied to $X_{\dd}$, the restriction of the bijective map
\[
 \Tan^*_{[-]}X_{\dd}\circ \Eu^{-1}\colon \Fun(X_{\dd},\mathcal{S}_{\dd})\rightarrow \HO^{\BMo}_{\toph}(\Lambda_{\dd},\BoZ)                                                                                                                                                                                                                                                               \]
to $\Fun^{\sph}(X_{\dd},\GL_{\dd})$ gives a map
\[
 \Fun^{\sph}(X_{\dd},\GL_{\dd})\rightarrow \HO^{\BMo}_{\toph}(\Lambda_{\dd}).
\]
This explains the map $\Tan^*_{[-]}\FM\circ\Eu^{-1}\colon \Fun^{\sph}(\FM_Q)^{\Psi}\rightarrow \HO^{\BMo}_{\toph}(\mathfrak{N}_{\Pi_Q}^{\SSN},\BoZ)^{\Psi}$ in the diagram \eqref{equation:diagramrealisations}.

The map $\alpha$ is obtained by performing the operation $\otimes_{\BoZ[q,q^{-1}]}\BoZ$, where $q\mapsto -1$, to the morphism of Proposition \ref{proposition:qgps}, twisting by $\Psi$ the multiplication on both sides and using Proposition \ref{proposition:oppositeparameter} to change the specialisation at $-1$ of the quantum group into the enveloping algebra. The map $\alpha$ is surjective by Corollary \ref{corollary:generationK0}.

Then, we set $\beta=\CCy\circ \alpha$ and $\gamma=\chi\circ \alpha=\res\circ\delta$.

Assuming temporary that $\beta$ is surjective, we can deduce that all maps in the diagram \eqref{equation:diagramrealisations} are isomorphisms:

\begin{enumerate}
 \item $\beta$ is an isomorphism since by Corollary \ref{corollary:equalitycharacters}, $\UEA^{\BoZ}(\mathfrak{n}_Q^+)$ and $\HO^{\BMo}_{\toph}(\mathfrak{N}^{\SSN}_{\Pi_Q},\BoZ)$ are free graded $\BoZ$-modules with the same character and it is surjective.
 
 \item $\alpha$ is an isomorphism since it is surjective and $\beta=\CCy\circ\alpha$ so it must be injective,
 \item $\CCy=\beta\circ\alpha^{-1}$ is an isomorphism,
 \item $\gamma$ is an isomorphism since it is surjective and $\Tan^*_{[-]}\mathfrak{M}\circ \Eu^{-1}\circ\gamma$ is an isomorphism,
 \item $\chi=\gamma\circ\alpha^{-1}$ is an isomorphism,
 \item $\Tan^*_{[-]}\mathfrak{M}=\gamma\circ \beta^{-1}$ is an isomorphism,
 \item $\res=\gamma\circ\delta^{-1}$ is an isomorphism.
\end{enumerate}
The surjectivity of $\beta$ is proven in \S \ref{subsection:surjectivitybeta}
\end{proof}
This theorem implies the following known corollaries.
\begin{corollary}[Bozec]
 \label{corollary:equalitycharacters}
 For any $\dd\in\BoN^{Q_0}$, $\rank_{\BoZ} \UEA^{\BoZ}(\mathfrak{n}_Q^+)[\dd]=\lvert \Irr(\Lambda_{\dd}^{\SSN})\rvert$.
\end{corollary}

\begin{corollary}
 Let $A_{Q,\dd}^{\SSN}(q)\in \BoN[q]$ be the seminilpotent Kac polynomial \cite{bozec2020number}. Then,
 \[
  \UEA^{\BoZ}(\mathfrak{n}_Q^+)[\dd]=A_{Q,\dd}^{\SSN}(0).
 \]
\end{corollary}
\begin{proof}
 The number of irreducible components of $\Lambda_{\dd}^{\SSN}$ is given by $A_{Q,\dd}^{\SSN}(0)$ by \cite{bozec2020number}.
\end{proof}

We conjecture the following.

\begin{conjecture}
 The definition $f_{Z}\mapsto (-1)^{\langle \dd,\dd\rangle}[Z]$ for $Z\in \Irr(\mathfrak{N}_{\Pi_Q,\dd}^{\SSN})$ (where $f_Z$ is defined by Lemma \ref{lemma:irrcompseminilpotent}) gives a definition for the dotted arrow in the diagram \eqref{equation:diagramrealisations} that makes the diagram commute.
\end{conjecture}

\subsection{Surjectivity of the map $\beta$}

\label{subsection:surjectivitybeta}
We use an argument parallel to that of the proof of \cite[Proposition 3.4]{bozec2015quivers}: we prove the result for one vertex quiver and use (implicitly) the crystal structure to obtain the result for any quiver.

\begin{lemma}
\label{lemma:surjbetaonevertex}
 Let $Q$ be a quiver with one vertex and $g$ loops. Then, the map
 \[
  \beta\colon \UEA^{\BoZ}(\mathfrak{n}_Q^+)\rightarrow \HO^{\BMo}_{\toph}(\mathfrak{N}_{\Pi_Q}^{\SSN},\BoZ)
 \]
is surjective.
\end{lemma}
\begin{proof}
We call $i$ the unique vertex of $Q$.

We have three cases: $g=0$, $g=1$ and $g\geq 2$ (respectively, real, isotropic and hyperbolic cases).  
 
If $g=0$, for any $d\geq 0$, $\Lambda_{d}^{\SSN}$ has a unique irreducible component, which is the image of $e_i^{(d)}=e_{(i',d)}$ under $\beta$, so $\beta$ is surjective.
 
 If $g=1$, the result is known and reduces to Springer theory for $\gl_d$, $d\geq 1$. In this case, $\Lambda_d^{\SSN}=\bigsqcup_{\mathcal{O}\subset\mathfrak{gl}_d}\Tan^*_{\mathcal{O}}\mathfrak{gl}_d$ where the sum runs over nilpotent orbits in $\mathfrak{gl}_d$. We order nilpotent orbits of $\gl_{d}$ by the dominance order $\mathcal{O}\leq \mathcal{O}'$ if $\mathcal{O}\subset\overline{\mathcal{O}'}$. We note that up to the sign, $\beta(e_{i,n_1}\star\hdots\star e_{i,n_r})=\CCy((\pi_{n_1,\hdots,n_r})_*\underline{\BQ}_{G\times^{P_{\underline{n}}}\mathfrak{n}_{\underline{n}}})$, where $\underline{n}=(n_1,\hdots,n_r)$, $P_{\underline{n}}$ is the parabolic subgroup of $\GL_{d}$ preserving a flag of subspaces of $\BoC^{d}$ with subquotients of dimensions given by $n_j$, $\mathfrak{n}_{\underline{n}}$ is the unipotent radical of the Lie algebra of $P_{\underline{n}}$ and
 \[
  \pi_{n_1,\hdots,n_r}\colon G\times^{P_{\underline{n}}}\mathfrak{n}_{\underline{n}}\rightarrow \mathfrak{gl}_d
 \]
 is the Springer map. Moreover any nilpotent orbit of $\mathfrak{gl}_d$ has a resolution of singularities given by a map $\pi_{n_1,\hdots,n_r}$. It proves that $[\overline{\Tan^*_{\mathcal{O}}\mathfrak{gl}_d}]$ appears with with multiplicity $\pm 1$ in $\beta(e_{i,n_1}\star\hdots e_{i,n_r})$ and the irreducible component $[\overline{\Tan^*_{\mathcal{O}'}\mathfrak{gl}_d}]$ appears in $\beta(e_{i,n_1}\star\hdots e_{i,n_r})$ only if $\mathcal{O}'\subset \mathcal{O}$. We deduce the surjectivity of $\beta$ in this case by induction on the orbit.
 
 If $g\geq 2$, we parameterise irreducible components of $\Lambda_{d}^{\SSN}$ by compositions of $d$ as in \cite{bozec2016quivers}. If $c$ is a composition of $d$, the corresponding irreducible component $\Lambda_{c}$ is the closure of the subset of $\Lambda_{d}^{\SSN}$ of elements such that $c_k=\dim W_{k-1}/W_{k}$, where $W_0=\BoC^{\dd}$ and $W_k$ is the smallest subspace of $\BoC^{\dd}$ containing $\sum_{i=1}^gx_i(W_{k-1})$ and stable under $x_i, y_i$ (a representation of $\Pi_Q$ is denoted $(x_i,y_i)$, $1\leq i\leq g$).
 
 Given a general element $(x_i,y_i)$ of $\Lambda_{c}$, $\BoC^{\dd}$ has a unique filtration $\BoC^{\dd}=W_0\supset W_1\hdots\supset W_r=\{0\}$ such that $x_i(W_j)\subset W_{j+1}$ and $y_i(W_j)\subset W_j$. Moreover, we order the set of compositions of $d$ as follows: $c=(c_1,\hdots,c_r)\leq c'=(c'_1,\hdots,c'_s)$ if $\sum_{j\geq t} c_j\leq \sum_{j\geq t} c'_j$ for any $t\geq 0$. We have
 \[
  [\Lambda_{(c_1)}]\star \hdots\star [\Lambda_{(c_r)}]
  =[\Lambda_{\rmc}]+\triangle
 \]
where $\triangle$ is a $\BoZ$-linear combination of irreducible components $\Lambda_{c'}$ of $\Lambda_{d}^{\SSN}$ such that $c'<c$. We used that in the (iterated) induction diagram
\[
 \Lambda_{c_1}\times\hdots\times\Lambda_{c_r}\xleftarrow{\phi}\tilde{\Lambda}_{c_1,\hdots,c_r}\xrightarrow{\psi}\Lambda_{d},
\]
$\psi\phi^{-1}(\Lambda_{(c_1)}\times\hdots\times\Lambda_{(c_r)})\subset \bigcup_{c'\leq c}\Lambda_{(c')}$.

We deduce that $\beta$ is surjective. 
\end{proof}
From now on, we drop the superscript $\SSN$ to lighten the notation: we write $\Lambda_{\dd}=\Lambda_{\dd}^{\SSN}$. Following \cite{bozec2016quivers}, for $x\in\Lambda_{\dd}$ and $i\in Q_0$, we let $\Im(x)_i$ be the smallest subspace of $\BoC^{\dd}$ that contains $\bigoplus_{j\neq i}\BoC^{\dd_j}$ and it stable under $x$. Recall the correspondence
\[
 (\Lambda_{\dd}\times\Lambda_{\ee})\times^{P_{\dd,\ee}}\GL_{\dd+\ee}\xleftarrow{\phi}\tilde{\Lambda}_{\dd,\ee}\xrightarrow{\psi}\Lambda_{\dd+\ee}.
\]
Let $\dd\in\BoN^{Q_0}$, $i\in Q_0$, and $l\geq 0$. We let
\[
 \Lambda_{\dd,i,l}=\{x\in\Lambda_{\dd}\mid \codim \Im(x)_i=l\}
\]

We have a correspondence
\begin{equation}
 \label{equation:correspondencephii}
 (\Lambda_{le_i}\times\Lambda_{\dd-le_i,i,0})\times^{P_{le_i,\dd-le_i}}\GL_{\dd}\xleftarrow{\phi_i}\tilde{\Lambda}_{\dd,i,l}\xrightarrow{\psi_i}\Lambda_{\dd,i,l}.
\end{equation}
where $\tilde{\Lambda}_{\dd,i,l}=\psi^{-1}(\Lambda_{\dd,i,l})=\phi^{-1}((\Lambda_{le_i}\times\Lambda_{\dd-le_i})\times^{P_{le_i,\dd-le_i}}\GL_{\dd})$. The map $\psi_i$ is an isomorphism and $\phi_i$ is smooth and surjective.

For $\Lambda\in \Irr(\Lambda_{\dd})$, we let $c(\Lambda)$ be the number $\codim\Im(x)_i$ for a general $x\in\Lambda$.

\begin{proposition}
 Let $Q$ be an arbitrary quiver. Then, the map
 \[
  \beta \colon \UEA^{\BoZ}(\mathfrak{n}_Q^+)\rightarrow \HO^{\BMo}_{\toph}(\mathfrak{N}_{\Pi_Q}^{\SSN},\BoZ)
 \]
 is surjective.
\end{proposition}
\begin{proof}
We prove the result by induction on $\dd\in\BoN^{Q_0}$ and decreasing $1\leq c(\Lambda)\leq \dd_i$. If $\dd$ is concentrated at one vertex, this is Lemma \ref{lemma:surjbetaonevertex}.

Let $\Lambda\in \Irr(\Lambda_{\dd})$. Let $i\in Q_0$ be such that $\codim \Im(x)_i=l>0$ for a general point $x\in\Lambda_{\dd}$. The map $\psi$ induces an isomorphism $\psi^{-1}(\Lambda\cap\Lambda_{\dd,i,l})\rightarrow \Lambda\cap\Lambda_{\dd,i,l}$. The subvariety $\psi^{-1}(\Lambda\cap\Lambda_{\dd,i,l})$ is irreducible, so that there exists $\Lambda_1\in \Irr(\Lambda_{le_i})$ and $\Lambda_2\in \Irr(\Lambda_{\dd-le_i})$ such that
\[
 \psi^{-1}(\Lambda\cap\Lambda_{\dd,i,l})=\phi^{-1}((\Lambda_1\times(\Lambda_2\cap \Lambda_{\dd-le_i,i,0}))\times^{P_{le_i,\dd-le_i}}\GL_{\dd}).
\]
In other terms, we have a correspondence
\[
(\Lambda_1\times(\Lambda_2\cap \Lambda_{\dd-le_i,i,0}))\times^{P_{le_i,\dd-le_i}}\GL_{\dd}\xleftarrow{\tilde{\phi}_i}\psi^{-1}(\Lambda\cap\Lambda_{\dd,i,l})\xrightarrow{\tilde{\psi}_i} \Lambda\cap\Lambda_{\dd,i,l}
\]
where $\tilde{\phi}_i$ is smooth with connected fibers and $\tilde{\psi}_i$ is an isomorphism.

The relative dimension of $\tilde{\phi}_i$ is $\dim \Lambda_{\dd}-\dim\Lambda_{le_i}-\dim\Lambda_{\dd-le_i}+\dim \GL_{\dd}-\dim P_{le_i,\dd-le_i}$. This is the same as the relative dimension of $\phi_i$ in the diagram \eqref{equation:correspondencephii}. Therefore, $[\Lambda_1]\star[\Lambda_2]=[\Lambda]+\square$, where $\square$ is a linear combination of irreducible components $\Lambda'$ of $[\Lambda_{\dd}]$ such that $c(\Lambda'_i)>l$. Here, we used the fact that $\psi\phi^{-1}(\Lambda_{le_i}\times\Lambda_{\dd-le_i})\subset \{x\in\Lambda_{\dd}\mid \codim\Im(x)_i\geq l\}$. By induction, $\square$ is in the image of $\beta$ and so are $[\Lambda_1],[\Lambda_2]$ (by induction on the dimension vector). Therefore, $[\Lambda]$ is also in the image of $\beta$.
\end{proof}

\section{Bialgebras morphisms}

\subsection{The restriction for constructible complexes}
\label{subsubsection:restrictionfunctor}
The diagram \eqref{equation:inddiagramquiver} with the operations of restriction $\Res_1, \Res_2$ of \S \ref{subsection:equivariantrestrictionfunctor} defines two restriction functors
\[
 \CD^{\rmb}_{\rmc,\GL_{\dd+\ee}}(X_{Q,\dd+\ee})\rightarrow \CD^{\rmb}_{\rmc,\GL_{\dd}\times\GL_{\ee}}(X_{Q,\dd}\times X_{Q,\ee}).
\]
By \cite[Proposition 1.12]{bozec2015quivers}, the functor $\Res_1$ sends $\mathcal{Q}_{\dd+\ee}$ in $\mathcal{Q}_{\dd}\boxtimes\mathcal{Q}_{\ee}$, inducing maps
\[
\Res\colon \GK_{\oplus}(\mathcal{Q})\rightarrow \GK_{\oplus}(\mathcal{Q})\otimes \GK_{\oplus}(\mathcal{Q}). 
\]
and
\[
\Res\colon \GK_0(\mathcal{Q})\rightarrow \GK_0(\mathcal{Q})\otimes \GK_0(\mathcal{Q}). 
\]

\begin{proposition}[{\cite[\S 9.2.11]{lusztig2010introduction}}]
\label{proposition:resctblesheavesalg}
 The map $\Res_1\colon \GK_{\oplus}(\mathcal{Q})\rightarrow \GK_{\oplus}(\mathcal{Q})\otimes \GK_{\oplus}(\mathcal{Q})$ makes $\GK_{\oplus}(\CQ)$ a $\xi$-twisted bialgebra where $\xi(m,n)=q^{-(m,n)}$.
\end{proposition}
\begin{proof}
 The proof of \cite[\S 9.2.11]{lusztig2010introduction} works for arbitrary quivers. Together with Proposition \ref{proposition:generatorsK0}, we obtain the result.
\end{proof}
By the procedure described in \S\ref{subsubsection:modificationxitwist}, with the automorphism $D$ of $\GK_{\oplus}(\CQ)$ induced by the Verdier duality, which restricts to $\BoZ[q,q^{-1}]$ to the automorphism $q\mapsto q^{-1}$, we obtain the $\xi'$-twisted comultiplication $\Res'\coloneqq (D\otimes D)\circ \Res\circ D$ on $\GK_{\oplus}(\CQ)$. By \S\ref{subsection:equivariantrestrictionfunctor}, $\Res'=\Res_2$.

\begin{lemma}
 The automorphism $D\otimes_{\BoZ[q,q^{-1}]}\BoZ$ of $\GK_0(\CQ)$ ($q\mapsto -1$) is trivial.
\end{lemma}
\begin{proof}
 This follows from Proposition \ref{proposition:generatorsK0}, the fact that intersection complexes are Verdier self-dual and that the induction functor commutes with Verdier duality.
\end{proof}

\begin{corollary}
\label{corollary:resccalgmor}
 The maps $\GK_0(\CQ)\rightarrow \GK_0(\CQ)\otimes\GK_0(\CQ)$ induced by $\Res_1$ and $\Res_2$ are algebra morphisms and coincide.
\end{corollary}

\subsection{The restriction for constructible functions on the stack of representations}
\label{subsection:restrictioncstbleQ}
We have a restriction operation
\[
 \Fun(X_{Q,\dd+\ee},\GL_{\dd+\ee})\rightarrow \Fun(X_{\dd}\times X_{\ee},\GL_{\dd}\times \GL_{\dd})
\]
obtained with the formalism of \S \ref{subsection:equivariantrestcstblefunctions} applied to the diagram \eqref{equation:inddiagramquiver}. It comes from the restriction of constructible complexes by applying the fiberwise Euler characteristic function. It sends $\Fun^{\sph}(X_{\dd+\ee},\GL_{\dd+\ee})$ in $\Fun^{\sph}(X_{\dd},\GL_{\dd})\times \Fun^{\sph}(X_{\ee},\GL_{\ee})$. It gives a coproduct on $\Fun^{\sph}(\mathfrak{M}_Q)$.

\begin{lemma}
\label{lemma:rescstblefunctionsalgmor}
 The restriction $\Res\colon \Fun^{\sph}(\FM_Q)\rightarrow\Fun^{\sph}(\FM_Q)\otimes\Fun^{\sph}(\FM_Q)$ is an algebra morphism.
\end{lemma}
\begin{proof}
 By definition, $\chi\colon \GK_0(\CP)\rightarrow\Fun^{\sph}(\FM_Q)$ is an algebra isomorphism. The lemmas follows then from Corollary \ref{corollary:resccalgmor}.
\end{proof}

\subsection{The restriction for constructible functions on the stack of representations of the preprojective algebra}
\label{subsubsection:coproductseminilpotent}
To define the coproduct, we follow the lines of \cite[\S 6]{geiss2005semicanonical}.

The map $p$ of \eqref{equation:inddiagramquiver} is a vector bundle; we let $i$ be its zero-section.

The formalism of \S \ref{subsection:equivrestrictionfctscotangent} gives a map
\[
 \Res_{\dd,\ee}\colon\Fun(\Lambda_{\dd+\ee}^{\SSN},\GL_{\dd+\ee})\rightarrow \Fun(\Lambda_{\dd}^{\SSN}\times\Lambda_{\ee}^{\SSN},\GL_{\dd}\times\GL_{\ee}).
\]

\begin{lemma}
\label{lemma:resalgebramorphism}
 The map $\Res_{\dd,\ee}$ sends $\Fun^{\sph}(\Lambda_{\dd+\ee}^{\SSN},\GL_{\dd+\ee})$ in $\Fun^{\sph}(\Lambda_{\dd}^{\SSN},\GL_{\dd})\otimes\Fun^{\sph}(\Lambda_{\ee}^{\SSN},\GL_{\ee})$. We let $\Res_{\dd,\ee}^{\Psi}=(-1)^{\Psi(\dd,\ee)}\Res_{\dd,\ee}$. The map
 \[
  \Res^{\Psi}=\bigoplus_{\dd,\ee\in\BoN^{Q_0}}\Res^{\Psi}_{\dd,\ee}\colon \Fun^{\sph}(\FN^{\SSN})^{\Psi}\rightarrow \Fun^{\sph}(\FN^{\SSN})^{\Psi}\otimes\Fun^{\sph}(\FN^{\SSN})^{\Psi}
 \]
is an algebra morphism (where $\Fun^{\sph}(\FN^{\SSN})^{\Psi}\otimes\Fun^{\sph}(\FN^{\SSN})^{\Psi}$ has the untwisted multiplication).
\end{lemma}
\begin{proof}
 The proof of \cite[Lemma 6.1]{geiss2005semicanonical} works for any quiver and gives, for $\Psi=(-1)^{\langle-,-\rangle}$,
 \begin{equation}
 \label{equation:rescotangent}
  \Res^{\Psi}(1_{i_1,n_1}\star\hdots\star 1_{i_r,n_r})=\sum_{\substack{p_j+q_j=n_j}}(1_{i_1,p_1}\star\hdots\star 1_{i_r,p_r})\otimes(1_{i_1,q_1}\star\hdots\star 1_{i_r,q_r}),
 \end{equation}
 where $1_{(i',n)}=1_{\Tan^*_{X_{Q,ne_{i'}}}X_{Q,ne_{i'}}}$. Since the functions $1_{(i',n)}$ generate $\Fun^{\sph}(\mathfrak{N}_{\Pi_Q})^{\Psi}$ (by definition), $\Res^{\Psi}$ is an algebra morphism. This proves the lemma when $\Psi=(-1)^{\langle-,-\rangle}$. The case of arbitrary $\Psi$ whose symmetrisation is the symmetrised Euler form is deduced from the formalism of \S\ref{subsection:Psi-twist} applied to the twisting multiplicative bilinear form $(-1)^{\langle-,-\rangle+\Psi}$.
\end{proof}

We obtain therefore obtain a comultiplication on $\Fun^{\sph}(\mathfrak{N}_{\Pi_Q}^{\SSN})^{\Psi}$.

\subsection{The bialgebra property}
To compare the bialgebras structures, we specify the $\Psi$-twist to $\Psi=(-1)^{\langle-,-\rangle}$. The reason comes that the new generators $\tilde{e}_{i',n}$ are not primitive and so they are not canonically defined independently of the twist. In this section, we prove Theorem \ref{theorem:comultiplications}.
\begin{proof}[Proof of Theorem \ref{theorem:comultiplications}]
 Given Lemma \ref{lemma:resalgebramorphism}, Corollary \ref{corollary:resccalgmor} and Lemma \ref{lemma:rescstblefunctionsalgmor}, we only need to show the compatibility of the comultiplications on the generators. 
 
 We prove that $\delta$ is a bialgebra morphism. By Lemma \ref{lemma:comultnewgens}, we have
 \[
  \Delta\tilde{e}_{i',n}=\sum_{p+q=n}\tilde{e}_{i',p}\otimes \tilde{e}_{i',q}.
 \]

 For $r=1$, the formula \eqref{equation:rescotangent} gives
 \[
  \Res^{\Psi}((-1)^{\langle ne_{i'},ne_{i'}\rangle}1_{i',n})=\sum_{p+q=n}((-1)^{\langle pe_{i'},pe_{i'}\rangle}1_{i',p})\otimes ((-1)^{\langle qe_{i'},qe_{i'}\rangle}1_{i',q}),
 \]
 since the signs on both sides compensate each other.
 
To show that $\alpha$ is a bialgebra morphism, it suffices to prove that the comultiplications agree on generators. By the definition of the restriction functor (\S \ref{subsubsection:restrictionfunctor}), we have
 \[
  \Res(\mathcal{IC}(\mathfrak{M}_{Q,ne_{i'}}))=\bigoplus_{p+q=n}\mathcal{IC}(\mathfrak{M}_{Q,pe_{i'}})\boxtimes\mathcal{IC}(\mathfrak{M}_{Q,qe_{i'}})[-pq\langle e_{i'},e_{i'}\rangle].
 \]
 Therefore, in $\GK_0(\CP)^{\Psi}$, we have
 \[
  \Delta^{\Psi}[(\mathcal{IC}(\mathfrak{M}_{Q,ne_{i'}}))]=\sum_{p+q=n}[\mathcal{IC}(\mathfrak{M}_{Q,pe_{i'}})]\otimes[\mathcal{IC}(\mathfrak{M}_{Q,qe_{i'}})].
 \]

By comparing with Lemma \ref{lemma:comultnewgens}, we obtain the result.

By construction, the stalk Euler characteristic $\chi\colon \GK_0(\CQ)^{\Psi}\rightarrow \Fun^{\sph}(\FM_Q)^{\Psi}$ is a bialgebra morphism and therefore, $\gamma\circ g=\chi\circ \alpha\circ g$ is also a bialgebra morphism. This finishes the proof.
\end{proof}

\section{Positive characteristics}
The definition of the characteristic cycle map has recently been extended to the $\ell$-adic setting by Saito in \cite{saito2017characteristic}, refining the definition of the singular suport in the $\ell$-adic setting \cite{beilinson2016constructible}. Moreover, the trace of Frobenius of a pure weight zero $\ell$-adic sheaf on a $\BoF_q$-variety can be considered as an enhancement of the Euler characteristic as the Euler characteristic can be recoverered by replacing the eigenvalue $\omega^{i/2}$ of $\Fr$ acting on $\HO^i(\mathscr{F}_x)$ (with $\lvert\omega\rvert=q^{1/2}$ by $(-1)^i$).

It seems therefore legitimate to ask whether one can define an Euler obstruction $\Eu$ in the $\ell$-adic setting making the diagram
\[
\begin{tikzcd}
	{\GK_0(\CD^{\rmb}_{\rmc,w=0}(X))} & {\HO^{\BMo}_{\toph}(X,\BoQ_{\ell})} \\
	{\prod_{n\geq 1}\Fun(X(\mathbb{F}_{q^n}))} & {\rmZ(X)}
	\arrow["{\Tr(\Fr,-)}"', from=1-1, to=2-1]
	\arrow["{\Tan^*_{[-]}X}"', from=2-2, to=1-2]
	\arrow["\CCy", from=1-1, to=1-2]
	\arrow["\Eu"', from=2-2, to=2-1]
\end{tikzcd}
\]
commute, leading to connections between the different realisations of quantum groups and Hall algebras. We leave this question for further investigations.

\printbibliography
\end{document}